\documentclass[11pt,reqno]{amsart}%
\usepackage{amsmath, amssymb}
\usepackage{amsfonts}
\usepackage{mathrsfs}
\usepackage[arrow,matrix,curve,cmtip,ps]{xy}
\usepackage{amsthm}
\usepackage{cite}
\usepackage{esint}
\usepackage{graphicx}%
\providecommand{\U}[1]{\protect\rule{.1in}{.1in}}
\allowdisplaybreaks
\newtheorem{theorem}{Theorem}[section]
\newtheorem{lemma}[theorem]{Lemma}
\newtheorem{proposition}[theorem]{Proposition}
\newtheorem{corollary}[theorem]{Corollary}

\newtheorem*{theorem*}{Theorem}
\theoremstyle{remark}
\newtheorem{remark}[theorem]{Remark}
\newtheorem{definition}[theorem]{Definition}

\numberwithin{equation}{section}

\newcommand{\N}{\mathbb{N}}
\newcommand{\R}{\mathbb{R}}

\newcommand{\T}{\mathbb{T}}
\newcommand{\vertiii}[1]{{\left\vert\kern-0.25ex\left\vert\kern-0.25ex\left\vert #1 \right\vert\kern-0.25ex\right\vert\kern-0.25ex\right\vert}}
\begin{document}
\title{A $L^2$ to $L^\infty$ approach for the Landau Equation}

\author{Jinoh Kim}
\address{Department of Mathematics, Pohang University of Science and Technology, South Korea}
\email{hugurk@postech.ac.kr}
\author{Yan Guo}
\address{Division of Applied Mathmatics Brown University, Providence, USA.}
\email{guoy@cfm.brown.edu}
\author{Hyung Ju Hwang}
\address{Department of Mathematics, Pohang University of Science and Technology, South Korea}
\email{hjhwang@postech.ac.kr}

\begin{abstract}
	Consider the Landau equation with Coulomb potential in a periodic box.
	We develop a new $L^{2}\rightarrow L^{\infty }$ framework to construct global unique solutions near Maxwellian with small $L^{\infty }\ $norm.
	The first step is to establish global $L^{2}$ estimates with strong velocity weight and time decay, under the assumption of $L^{\infty }$ bound, which is further controlled by such $L^{2}$ estimates via De Giorgi's method \cite{golse2016harnack} and \cite{mouhot2015holder}.
	The second step is to employ estimates in $S_{p}$ spaces to control velocity derivatives to ensure uniqueness, which is based on Holder estimates via De Giorgi's method \cite{golse2016harnack}, \cite{golse2015holder}, and \cite{mouhot2015holder}. 
\end{abstract}

\maketitle

\section{introduction}

\label{sec:introduction}

We consider the following generalized Landau equation:
\begin{equation}		\label{Landau near vaccum}
	\begin{split}
		\partial_t F + v \cdot \nabla_{x} F &= Q(F,F)
		= \nabla_{v} \cdot \left\{ \int_{\R^3} \phi(v-v')[F(v')\nabla_{v} F(v) - F(v) \nabla_{v} F(v')] dv'\right\},\\
		F(0,x,v)&=F_0(x,v),
	\end{split}
\end{equation}
where $F(t,x,v)\ge 0$ is the spatially periodic distribution function for particles at time $t\ge0$, with spatial coordinates $x=(x_1,x_2,x_3)\in [-\pi, \pi]^3 = \T^3$ and velocity $v=(v_1,v_2,v_3)\in \R^3$. The non-negative matrix $\phi$ is 
\begin{equation}		\label{Eq : phi}
	\phi^{ij}(v)=\left\{  \delta_{i,j}-\frac{v_{i}v_{j}}{|v|^{2}}\right\}|v|^{-1}.
\end{equation}
As in the Boltzmann equation, it is well-known that Maxwellians are steady states to \eqref{Landau near vaccum}.
Let $\mu$ be a normalized Maxwellian
\begin{equation}		\label{Eq : mu}
	\mu(v) = e^{-|v|^2},
\end{equation}
and set
\begin{equation}		\label{Eq : F}
	F(t,x,v) = \mu(v) + \mu^{1/2}(v)f(t,x,v).
\end{equation}
Then the standard perturbation $f(t,x,v)$ to $\mu$ satisfies
\begin{equation}
f_{t}+v\cdot\partial_{x}f+Lf=\Gamma(f,f), \label{Landau}%
\end{equation}%
\begin{equation}
f(0,x,v)=f_{0}(x,v), \label{initial}%
\end{equation}
where $f_{0}$ is the
initial data satisfying the conservation laws:
\begin{equation}
\int_{\mathbb{T}^{3}\times\mathbb{R}^{3}}f_{0}(x,v)\sqrt{\mu}=\int
_{\mathbb{T}^{3}\times\mathbb{R}^{3}}v_{i}f_{0}(x,v)\sqrt{\mu}=\int
_{\mathbb{T}^{3}\times\mathbb{R}^{3}}|v|^{2}f_{0}(x,v)\sqrt{\mu}=0.
\label{conservation laws}%
\end{equation}

The linear operator $L$ and the nonlinear part $\Gamma$ are defined as%
\begin{equation}		\label{Eq : L}
	L = -A -K,
\end{equation}
\begin{equation}		\label{Eq : A}
	\begin{split}
		Af  &  :=\mu^{-1/2}\partial_{i}\left\{  \mu^{1/2}\sigma^{ij}[\partial_{j}f+v_{j}f]\right\} \\
		&  =\partial_{i}[\sigma^{ij}\partial_{j}f]-\sigma^{ij}v_{i}v_{j}f+\partial_{i}\sigma^{i}f,
	\end{split}
\end{equation}
\begin{equation}		\label{Eq : K}
	Kf:=-\mu^{-1/2}\partial_{i}\left\{  \mu\left[  \phi^{ij}\ast\left\{  \mu^{1/2}[\partial_{j}f+v_{j}f]\right\}  \right]  \right\}  ,
\end{equation}
\begin{equation}		\label{Eq : Gamma}
	\begin{split}
		\Gamma\lbrack g,f]  &  :=\partial_{i}\left[  \left\{  \phi^{ij}\ast\lbrack\mu^{1/2}g]\right\}  \partial_{j}f\right]  -\left\{  \phi^{ij}\ast\lbrack	v_{i}\mu^{1/2}g]\right\}  \partial_{j}f\\
		&  \quad-\partial_{i}\left[  \left\{  \phi^{ij}\ast\lbrack\mu^{1/2}%
		\partial_{j}g]\right\}  f\right]  +\left\{  \phi^{ij}\ast\lbrack v_{i}\mu^{1/2}\partial_{j}g]\right\}  f,
	\end{split}
\end{equation}
\begin{equation}		\label{Eq : sigma_u}
	\sigma_{u}^{ij}(v) := \phi^{ij}*u = \int_{\mathbb{R}^{3}}\phi^{ij}(v-v^{\prime})u(v^{\prime})dv^{\prime},
\end{equation}
\begin{equation}		\label{Eq : sigma and sigma^i}
	\sigma = \sigma_{\mu}, \quad \sigma^{i} = \sigma^{ij}v_{j}.
\end{equation}
To get $L^{\infty}$ estimates, we rearrange
(\ref{Landau}) as follows:
\begin{equation}
f_{t}+v\cdot\nabla_{x}f=\bar{A}_{f}f+\bar{K}_{f}f, \label{rearrange Landau}%
\end{equation}%
\begin{equation}		\label{Eq : bar A}
	\begin{split}
		\bar{A}_{g}f  &  :=\partial_{i}\left[  \left\{  \phi^{ij}\ast\lbrack\mu	+\mu^{1/2}g]\right\}  \partial_{j}f\right]\\
		&  \quad-\left\{  \phi^{ij}\ast\lbrack v_{i}\mu^{1/2}g]\right\}  \partial_{j}f-\left\{  \phi^{ij}\ast\lbrack\mu^{1/2}\partial_{j}g]\right\}
		\partial_{i}f\\
		&  =: \nabla_{v}\cdot(\sigma_{G}\nabla_{v}f)+a_{g}\cdot\nabla_{v}f,
	\end{split}
\end{equation}
\begin{equation}		\label{Eq : bar K}
	\begin{split}
	\bar{K}_{g}f  &  :=Kf+\partial_{i}\sigma^{i}f  -\sigma^{ij}v_{i}v_{j}f\\
	&  \quad-\partial_{i}\left\{  \phi^{ij}\ast\lbrack\mu^{1/2}\partial	_{j}g]\right\}  f+\left\{  \phi^{ij}\ast\lbrack v_{i}\mu^{1/2}\partial_{j}g]\right\}  f.
	\end{split}
\end{equation}

Define the weighed norm and weighted energy associated with \eqref{Landau}:
\begin{equation}		\label{Eq : weighted L^p}
	w:=(1+|v|),\quad|f|_{p,\vartheta}^{p}:=\int_{\mathbb{R}^{3}}w^{p\vartheta}f^{p}dv,\quad\Vert f\Vert_{p,\vartheta}^{p}:=\int_{\mathbb{T}^{3}\times\mathbb{R}^{3}}w^{p\vartheta}f^{p}dxdv.
\end{equation}
\begin{equation}		\label{Eq : weighted sigma 1}
	|f|_{\sigma,\vartheta}^{2}:=\int_{\mathbb{R}^{3}}w^{2\vartheta}\left[\sigma^{ij}\partial_{i}f\partial_{j}f+\sigma^{ij}v_{i}v_{j}f^{2}\right]  dv,
\end{equation}
\begin{equation}		\label{Eq : weighted sigma 2}
	\Vert f\Vert_{\sigma,\vartheta}^{2}:=\iint_{\mathbb{T}^{3}\times\mathbb{R}^{3}}w^{2\vartheta}\left[  \sigma^{ij}\partial_{i}f\partial_{j}f+\sigma^{ij}v_{i}v_{j}f^{2}\right]  dvdx,
\end{equation}
\begin{equation}		\label{Rw : weighted L^infty}
	|f|_{\infty,\vartheta}=\sup_{\mathbb{R}^{3}}w^{\vartheta}(v)f(v),\quad\Vert f\Vert_{\infty,\vartheta}=\sup_{\mathbb{T}^{3}\times\mathbb{R}^{3}}w^{\vartheta}(v)f(x,v).
\end{equation}

\[%
\begin{split}
|f|_{2}:=|f|_{2,0},  &  \quad\Vert f\Vert_{2}:=\Vert f\Vert_{2,0},\\
|f|_{\sigma}:=|f|_{\sigma,0},  &  \quad\Vert f\Vert_{\sigma}:=\Vert
f\Vert_{\sigma,0},\\
|f|_{\infty}:=|f|_{\infty,0},  &  \quad\Vert f\Vert_{\infty}:=\Vert
f\Vert_{\infty,0},
\end{split}
\]%
\begin{equation}		\label{Eq : inner L^2}
	\langle f,g\rangle:=\int_{\mathbb{R}^{3}}fgdv,\quad(f,g):=\int_{\mathbb{T}^{3}\times\mathbb{R}^{3}}fgdxdv,
\end{equation}
\begin{equation}		\label{Eq : inner sigma 1}
	\langle f,g \rangle_{\sigma}:=\int_{\mathbb{R}^{3}}\left[\sigma^{ij}\partial_{i}f\partial_{j}g+\sigma^{ij}v_{i}v_{j}fg\right]  dv,
\end{equation}
\begin{equation}		\label{Eq : inner sigma 2}
	( f, g)_{\sigma}:=\iint_{\mathbb{T}^{3}\times\mathbb{R}^{3}}\left[  \sigma^{ij}\partial_{i}f\partial_{j}g+\sigma^{ij}v_{i}v_{j}fg\right]  dvdx,
\end{equation}

\begin{equation}		\label{Eq : E}
	\mathcal{E}_{\vartheta}(f(t)) := \frac{1}{2}\|f(t)\|_{2,\vartheta}^{2} + \int_{0}^{t} \|f(s)\|_{\sigma,\vartheta}^{2} ds.
\end{equation}

\begin{remark}
\label{Rmk : adjoint operator} If $g\in C^{1}_{v}$, $f\in C^{2}_{v}$, and
$\varphi\in C^{2}_{v}$ with a compact support in $\mathbb{R}%
^{3}$, then
\begin{multline*}
	\int_{\R^3} \left(\bar A_{g}(f) + \bar K_{g}(f) \right) \varphi = - \langle f, \varphi \rangle_{\sigma} - \int_{\R^3} \sigma_{\mu^{1/2}g} \partial_{i}f \partial_{j} \varphi - (a_{g} \cdot \nabla_{v} \varphi) f\\
	+ \left(K\varphi + \partial_{i} \sigma^{i} \varphi - \partial_{i} \left\{\phi^{ij} * [\mu^{1/2} \partial_{j} g] \right\} \varphi + \left\{ \phi^{ij}* [v_{i} \mu^{1/2}\partial_{j}g] \right\}\varphi \right) f dv.
\end{multline*}
\end{remark}

We first consider the linearized Landau equation with a given $g$:
\begin{equation}
\partial_{t}f+v\cdot\nabla_{v}f+Lf=\Gamma(g,f). \label{linear landau}%
\end{equation}
If $f$ satisfies \eqref{linear landau}, then for every $\vartheta\in
\mathbb{R}$, $f^{\vartheta}:=w^{\vartheta}f$ satisfies
\[
\partial_{t}f^{\vartheta}+v\cdot\nabla_{x}f^{\vartheta}=\bar{A}_{g}%
^{\vartheta}f^{\vartheta}+\bar{K}_{g}^{\vartheta}f,
\]%
\[
f^{\vartheta}(0)=w^{\vartheta}f_{0}=:f_{0}^{\vartheta},
\]
where
\begin{equation}		\label{Eq : bar K_g^theta}
	\bar{K}_{g}^{\vartheta}f=w^{\vartheta}\bar{K}_{g}f+\left(  2\frac{\partial_{i}w^{\vartheta}\partial_{j}w^{\vartheta}}{w^{2\vartheta}}\sigma_{G}^{ij}-\frac{\partial_{ij}w^{\vartheta}}{w^{\vartheta}}\sigma_{G}^{ij}-\frac{\partial_{j}w^{\vartheta}}{w^{\vartheta}}\partial_{i}\sigma_{G}^{ij}-\frac{\partial_{i}w^{\vartheta}}{w^{\vartheta}}a_{g}^{i}\right)
f^{\vartheta},
\end{equation}
\begin{equation}		
\label{Eq : bar A_g^theta}
\bar{A}_{g}^{\vartheta}:=\bar{A}_{g}-2\frac{\partial_{i}w^{\vartheta}%
}{w^{\vartheta}}\sigma_{G}^{ij}\partial_{j}.	
\end{equation}
Note that $\bar{A}_{g}^{0}=\bar{A}_{g}$ and $f^{0}=f$.
%section introduction (end)
Later we will derive an energy estimate for the equation:
\begin{equation}
h_{t}+v\cdot\nabla_{x}h=\bar{A}_{g}^{\vartheta}h \label{Eq : equation for h}%
\end{equation}

Here we introduce the main result.

\begin{definition}
\label{Def : weak sol} Let $f(t,x,v) \in L^{\infty}((0,\infty)\times
\mathbb{T}^{3} \times\mathbb{R}^{3},w^{\vartheta}(v)dtdxdv)$ be a periodic function in
$x\in\mathbb{T}^{3} = [-\pi,\pi]^{3}$ satisfying
\begin{equation}		%\label{Eq : sigma norm bdd condition}
	\int_{0}^{t}\|f(s)\|_{\sigma, \vartheta}^2 ds < \infty.
\end{equation}
We say that $f$ is a weak solution of
the Landau equation \eqref{Landau}, \eqref{initial} on $(0,\infty
)\times\mathbb{T}^{3} \times\mathbb{R}^{3}$ if for all $t\in(0,\infty)$ and
all $\varphi\in C^{1,1,1}_{t,x,v}\left(  (0,\infty)\times\mathbb{T}^{3}
\times\mathbb{R}^{3}\right)  $ such that $\varphi(t,x,v)$ is a periodic
function in $x\in\mathbb{T}^{3} = [-\pi,\pi]^{3}$ and $\varphi(t,x,\cdot)$ is
compactly supported in $\mathbb{R}^{3}$, it satisfies
\begin{equation}
\label{Eq : weak sol}%
\begin{split}
&  \iint_{\mathbb{T}^{3} \times\mathbb{R}^{3}} f(t,x,v) \varphi(t,x,v) dxdv -
\iint_{\mathbb{T}^{3} \times\mathbb{R}^{3}} f_{0}(x,v) \varphi(0,x,v) dxdv\\
& = - (f,\varphi)_{\sigma} + \iiint_{(0,t)\times\mathbb{T}^{3} \times\mathbb{R}^{3}}
f(s,x,v)\bigg(\partial_{s} \varphi(s,x,v) + v \cdot\nabla_{x}\varphi(s,x,v)+a_{f}(s,x,v)\cdot \nabla_{v} \varphi(s,x,v)\\
&\quad \quad + K \varphi(s,x,v) + \partial_{i}\sigma^{i}(s,x,v) \varphi(s,x,v) - \partial_{i}\left\{\phi^{ij}*[\mu^{1/2}\partial_{j}f]\right\}(s,x,v)\varphi(s,x,v)\\
&\quad \quad + \left\{\phi^{ij}*[v_{i}\mu^{1/2}\partial_{j}f]\right\}(s,x,v)\varphi(s,x,v) \bigg) - \sigma_{\mu^{1/2}f}(s,x,v) \partial_{i}f(s,x,v) \partial_{j} \varphi(s,x,v) dsdxdv.
\end{split}
\end{equation}
\end{definition}

\begin{theorem}
[Main results]\label{Thm : main result} There exist $\vartheta^{\prime}$ and 
$0<\varepsilon_{0}\ll 1$ such that for some $\vartheta \ge \vartheta^{\prime}$ if $f_{0}$ satisfies
\begin{equation}
\label{Eq : initial condition}
	\|f_{0}\|_{\infty,\vartheta} \le \varepsilon_0, \quad \|f_{0t}\|_{\infty, \vartheta} + \|D_v f_0\|_{\infty, \vartheta}  < \infty,
\end{equation}
where $f_{0t} := -v\cdot \nabla_x f_{0} + \bar A_{f_{0}}f_{0}$.
\begin{enumerate}
\item Then there exists a unique weak solution $f$ of \eqref{Landau},
\eqref{initial} on $(0,\infty)\times\mathbb{T}^{3} \times\mathbb{R}^{3}$.
\item Let $F(t,x,v) = \mu(v) + \sqrt \mu(v) f(t,x,v)$.
If $F(0) \ge 0$, then $F(t) \ge 0$ for every $t\ge 0$.

\item Moreover for any $t>0$, $\vartheta_{0}\in\mathbb{N}$, and $\vartheta\ge \vartheta'$, there exist $C$, $C_{\vartheta, \vartheta_{0}}$,
$l_0(\vartheta_0)$, and $0<\alpha<1$ such that $f$ satisfies
\begin{equation}
\label{Eq : energy estimate}
	\sup_{0 \le s\le\infty}\mathcal{E}_{\vartheta}(f(s)) \le C 2^{2\vartheta}\mathcal{E}_{\vartheta}(0),
\end{equation}
\begin{equation}
\label{Eq : decay estimate}
	\|f(t)\|_{2,\vartheta} \le C_{\vartheta,\vartheta_{0}} \mathcal{E}_{\vartheta+\vartheta_{0}/2}(0)^{1/2}\left(  1+ \frac{t}{\vartheta_{0}}\right)  ^{-\vartheta_{0}/2},
\end{equation}
\begin{equation}
\label{Eq : L^infty estimate}
	\|f(t)\|_{\infty, \vartheta} \le C_{\vartheta,\vartheta_{0}}(1+t)^{-\vartheta_{0}}\|f_{0}\|_{\infty,\vartheta+l_{0}},
\end{equation}
\begin{equation}
\label{Eq : Holder}%
	\|f\|_{C^{\alpha}(\Omega)} \le C\left(\|f_{0t}\|_{\infty, \vartheta} + \|f_{0}\|_{\infty, \vartheta}\right),
\end{equation}
and
\begin{equation}
\label{Eq : D_v f bddness}%
	\|D_{v} f\|_{L^{\infty}\big((0,\infty)\times\mathbb{T}^{3} \times\mathbb{R}^{3})\big)}  \le C\left(\|f_{0t}\|_{\infty, \vartheta} + \|D_v f_0\|_{\infty, \vartheta} + \|f_{0}\|_{\infty, \vartheta}\right).
\end{equation}
\end{enumerate}
\end{theorem}

Motivated by the study of global well-posedness for the Landau equation in
a bounded domain with physical boundary conditions, our current study is the
first step to develop a $L^{2}\rightarrow L^{\infty }$ framework with
necessary analytic tools in a simpler periodic domain. There have been many
results for Landau equations in either a periodic box or whole domain
\cite{boblylev2013particle}, \cite{bobylev2015some}, \cite{carrapatoso2016cauchy}, \cite{chen2009smoothing}, \cite{desvillettes2015entropy}, \cite{guo2002landau}, \cite{ha2015l2}, \cite{herau2011anisotropic}, \cite{herau2013global}, \cite{liu2014regularizing}, \cite{luo2016spectrum}, \cite{strain2006almost}, \cite{strain2013vlasov}, and \cite{wu2014global}, in which high-order Sobolev norms can be employed. On the other
hand, in a bounded domain, even with the velocity diffusion, the solutions
can not be smooth up to the grazing set \cite{hwang2014fokker}. New mathematical
tools involving weaker norms are needed to be developed. In the case for
Boltzmann equations, a $L^{2}\rightarrow L^{\infty }$ framework has been
developed to construct unique global solutions in bounded domains \cite{guo2010decay}.

Our work can be viewed as a similar $L^{2}\rightarrow L^{\infty }$ approach
for the Landau equation. Our techniques are inspired by recent remarkable
progresses of \cite{golse2016harnack}, \cite{golse2015holder}, and \cite{mouhot2015holder}, in which a general machinery in the spirit of De
Giorgi, has been developed to the Fokker-Planck equations, even to the Landau
equation \cite{mouhot2015holder}, to bootstrap $L^{\infty }$ and Holder space $C^{0,\alpha }$
from a $L^{2}$ weak solution$.$ Unfortunately, to our knowledge, there is
still no construction for $L^{\infty }$ global weak solutions to the Landau
equation.

Our paper settles the global existence and uniqueness for a $L^{2}$ weak
solution with a small weighted $L^{\infty }$ perturbation of a Maxwellian
initially. Our method is an intricate combination of different tools. Our
starting point is a design of an iterating sequence 
\begin{eqnarray*}
(\partial _{t}+v\cdot \nabla _{x}) f^{n+1} &=&-Lf^{n+1}+\Gamma (f^{n},f^{n+1}) \\
&\equiv &A_{f^{n}}(f^{n+1})+K_{f^{n}}(f^{n+1}),
\end{eqnarray*}%
where $A_{f^{n}}(f^{n+1})$ contains \textit{all} the derivatives and $%
K_{f^{n}}(f^{n+1})$ has \textit{no derivative} of $f^{n}$ and $f^{n+1},$ so
that $f^{n}$ appears in the coefficients of the Landau operator for $f^{n+1}.
$ The crucial lemma states that if $||f^{n}||_{\infty }$ is sufficiently
small, the main part of $A_{f^{n}}(f^{n+1})$ retains the same analytical
properties of the linearized Landau operator $A.$ 

We first establish global energy estimates and time-decay under the
assumption $f^{n}$ is small, in Section \ref{sec:l2 decay}.
\begin{theorem}		\label{Thm : energy estimate}
	Suppose that $\|g\|_{\infty} < \varepsilon$.
	Let $\vartheta\in 2^{-1}\mathbb{N} \cup\{0\}$ and $f$ be a classical solution of \eqref{initial}, \eqref{conservation laws}, and \eqref{linear landau}.
	Then there exist $C$ and $\varepsilon=\varepsilon(\vartheta)>0$ such that 
	\begin{equation}	\label{Eq : energy estimate linear}
		\sup_{0 \le s< \infty}\mathcal{E}_{\vartheta}(f(s)) \le C 2^{2\vartheta} \mathcal{E}_{\vartheta}(0),
	\end{equation}
	and
	\begin{equation}	\label{Eq : decay estimate linear}
		\|f(t)\|_{2,\vartheta} \le C_{\vartheta,k} \left(\mathcal{E}_{\vartheta+k/2}(0) \right)^{1/2}\left(  1+ \frac{t}{k}\right)^{-k/2}%
	\end{equation}
	for any $t>0$ and $k\in\mathbb{N}$.
\end{theorem}
It is important to note that, thanks to the nonlinearity,
the velocity weight can be arbitrarily strong. The proof of this step is a
combination of energy estimates with positivity estimates for $\textbf{P}f$ \cite{esposito2013non}, \cite{guo2002landau} and
a time-decay estimate \cite{strain2006almost}, but in the absence of high-order Sobolev
regularity.

We next bootstrap such a $L^{2}$ bound to a $L^{\infty }$ bound.
\begin{theorem}	\label{Thm : L^infty estimate} 
	Suppose that $\|g\|_{\infty} < \varepsilon$.
	Let $f$ be a weak solution of \eqref{initial}, \eqref{conservation laws}, and \eqref{linear landau} in a periodic box and $\vartheta\in\mathbb{N} \cup\{0\}$, $\vartheta_{0}
	\in\mathbb{N}$.
	Then there exist $\varepsilon$, $l_{0}(\vartheta_{0})>0$ and $C_{\vartheta, \vartheta_0}$ such that
	\begin{equation}	\label{Eq : L^infty estimate linear}
		\|f(t)\|_{\infty, \vartheta+\vartheta_{0}}\le C_{\vartheta,\vartheta_0}(1+t)^{-\vartheta_{0}}\|f_{0}\|_{\infty,\vartheta+l_{0}}.
\end{equation}
\end{theorem}
It is important to note that even though
there is a finite loss of velocity weight, we are still able to close the
estimates thanks to the strong gain of velocity weight in \eqref{Eq : decay estimate linear}. The proof  of such a
$L^{\infty }$ estimate locally in $x$ and $v$ \ is an adaptation of recent work of
\cite{golse2015holder}, \cite{mouhot2015holder}. It is well-known that the Landau operator is delicate to study and
estimate for large velocities. Together with the maximum principle of the Landau
operator as well as strong time decay for $L^2$ norm in \eqref{Eq : decay estimate linear}, we are able to control the `tails' of solutions for large velocities, and
obtain global (in $x$ and $v$) $L^{\infty }$ estimate.

Unfortunately, unlike in the Boltzmann case (see \cite{guo2010decay}), in order to establish the
convergence of $\{f^{n}\}$ and more importantly, uniqueness of our solution,
such a $L^{\infty }$ bound is not sufficient due to the presence of velocity
derivative in the nonlinear Landau equation. We need to further control $%
||\nabla _{v}f^{n}||_{\infty }$ as in Lemma \ref{lemma : iteration}, which follows from $S_{p}$ estimates 
established in \cite{polidoro1998sobolev}. One crucial requirement for such $S_{p}$ estimates (as
the classical $W^{2,p}$ estimate in the elliptic theory), is the $%
C^{0,\alpha }$ estimate (uniform in $x$ and $v$) for the coefficients
containing $f^{n}.$ We establish 
\begin{theorem}		\label{Thm : Holder landau}
	Let $f$ be a solution of \eqref{initial}, \eqref{conservation laws}, and \eqref{linear landau}. 
	Then there exist $\varepsilon$, $\vartheta'$, $C$, $\alpha>0$  such that if $g$ satisfies \eqref{Eq : condition for g}, then we have
	\begin{equation*}
		\|f\|_{C^{\alpha}(\Omega)} \le C(f_0),
	\end{equation*}
	where
	\begin{equation}		\label{Eq : C(f_0)}
		C(f_0) = C\left( \|f_{0t}\|_{\infty, \vartheta} + \|f_{0}\|_{\infty, \vartheta}\right).
	\end{equation}
\end{theorem}
Again, we follow the methods in \cite{golse2016harnack}, \cite{golse2015holder}, and \cite{mouhot2015holder} to establish such an estimate locally
in $(x,v),$ and use a delicate change of coordinates (6.16) locally to
capture precisely the isotropic behavior of the Landau operator, thanks to
Lemma \ref{lemma : eigenvalue estimate} and our strong weighted $L^{\infty }$ estimates to obtain uniform $%
C^{0,\alpha }$ estimate$.$ It is well-known that the Landau equation is
degenerate for $|v|\rightarrow \infty$ and our strong energy estimate provides
the control of  velocity (tails) of the Landau solutions. An additional
regularity condition $||f_{0t}||_{\infty ,\theta }<+\infty $ is needed for
such a Holder estimate, but no smallness is required. A further bound $%
||f_{0v}||_{\infty ,\theta }<+\infty $ is needed to apply the $S_{p}$ theory
in a non-divergent form.

Such a $L^{2}\rightarrow L^{\infty }$ framework is robust and is currently
being applied to the study of several other problems in the kinetic theory. 
\section{Basic Estimates} % (fold)
\label{sec:basic_estimates}

For the reader's convenience, we summarize and modify some basic estimates.
We will adapt techniques in \cite{guo2002landau}.

For every $v,\nu\in\mathbb{R}^{3}$, define
\begin{equation}		\label{Eq : D_u}
	D_{u}(\nu;v) := \nu^{T} \sigma_{u}(v) \nu
\end{equation}
and $P_{v}$ is the projection onto the vector $v$ as
\begin{equation}
\label{Eq : P_v}P_{v} g := \sum\langle g_{j}, v_{j} \rangle\frac{v_{i}%
}{|v|^{2}}, \quad1\le i \le3.
\end{equation}

\begin{proposition}		\label{Prop : L2Linfty}
	There exists a uniform constant $C$ such that for every function $f$ and a constant $\vartheta\in \R$, we have
	\begin{equation*}
		\|f\|_{2, \vartheta} \le C\|f\|_{\infty, \vartheta+2}.
	\end{equation*}
\end{proposition}
\begin{proof}
	\begin{equation*}
		\begin{split}
			\|f\|_{2, \vartheta}^2 & = \iint_{\T^3 \times \R^3} (w^{\vartheta}(v) f(x,v))^2 dxdv\\
			& = \iint_{\T^3 \times \R^3} w^{-4}(v) (w^{\vartheta +2 }(v) f(x,v))^2 dxdv\\
			& \le \|f\|_{\infty, \vartheta+2} \iint_{\T^3 \times \R^3} w^{-4}(v) dxdv\\
			& \le C\|f\|_{\infty, \vartheta+2},
		\end{split}
	\end{equation*}
	for some constant $C>0$.
\end{proof}
\begin{lemma}
[Lemma 2 in \cite{guo2002landau}]\label{Lemma : Guo2002 lemma2} Let
$\vartheta>-3$, $a(v) \in C^{\infty}(\mathbb{R}^{3})$ and $b(v)\in C^{\infty
}(\mathbb{R}^{3}\setminus\{0\})$. Assume for any positive multi-index $\beta$, there is $C_{\beta
}>0$ such that
\[%
\begin{split}
|\partial_{\beta}a(v)|  &  \le C_{\beta}|v|^{\vartheta-|\beta|},\\
|\partial_{\beta}b(v)|  &  \le C_{\beta}e^{-\tau_{\beta}|v|^{2}},
\end{split}
\]
with some $\tau_{\beta}>0$. Then there is $C^{*}_{\beta}>0$ such that
\[
|\partial_{\beta}[a*b](v)| \le C^{*}_{\beta}[1+|v|]^{\vartheta- \beta}.
\]

\end{lemma}

\begin{lemma}
[Lemma 3 in \cite{guo2002landau}]\label{Lemma : Guo2002 lemma3} If $u = \mu$
or $\sqrt\mu$, then
\begin{equation}
\label{Eq : D_mu}D_{u}(\nu;v) = \lambda_{1}(v)|P_{v} \nu|^{2} + \lambda_{2}(v)
|(I-P_{v})\nu|^{2}.
\end{equation}
Moreover, there exists $C$ such that
\[
\frac{1}{C}(1+|v|)^{-3}\le\lambda_{1}(v) \le C(1+|v|)^{-3},
\]
and
\[
\frac{1}{C}(1+|v|)^{-1} \le\lambda_{2}(v) \le C(1+|v|)^{-1}.
\]

\end{lemma}

We can derive upper and lower bounds of eigenvalues for $\sigma+ \sigma
_{\sqrt\mu g}$ by adapting ideas in the proof of Theorem 3 in
\cite{guo2002landau}.

\begin{lemma}
\label{lemma : eigenvalue estimate} Let $g$ be a given function in $L^\infty((0,\infty)\times\T^3\times \R^3)$ and $G = \mu+
\sqrt\mu g$. 
Let $\sigma_{G}$ be the matrix defined as in \eqref{Eq : sigma_u}.
Then there exists $0<\varepsilon\ll 1$ such that if $g$ satisfies
\begin{equation}
\label{Eq : condition for g}\sup_{0\le s\le\infty}\|g(s)\|_{\infty}%
\le\varepsilon
\end{equation}
then
\[%
\begin{split}
D_{G}(\nu;v)  &  \ge\frac{1}{2C} \left(  (1+|v|)^{-3}|P_{v} \nu|^{2} +
(1+|v|)^{-1}|(I-P_{v})\nu|^{2}\right)  ,\\
D_{G}(\nu;v)  &  \le2C \left(  (1+|v|)^{-3}|P_{v} \nu|^{2} +
(1+|v|)^{-1}|(I-P_{v})\nu|^{2}\right)  ,
\end{split}
\]
for every $v \in\mathbb{R}^{3}$. Thus $\sigma_{G}(v)$ has three non-negative
eigenvalues. Moreover, $\lambda(v)$, eigenvalue of $\sigma_{G}(v)$, has the following estimate
\[
\frac{1}{C}(1+|v|)^{-3}\le\lambda(v) \le C(1+|v|)^{-1},
\]
for some constant $C>0$.
\end{lemma}

\begin{proof}
Let $u=\sqrt\mu g$. Then we claim that there exists $C^{\prime}>0$ such that
\begin{equation}
\label{Eq : D_(mu^1/2 g)}|D_{u}(\nu;v)|\le C^{\prime}\|g\|_{\infty}\left(
(1+|v|)^{-3}|P_{v} \nu|^{2} + (1+|v|)^{-1}|(I-P_{v})\nu|^{2}\right)
.
\end{equation}

Consider
\[%
\begin{split}
D_{u}(\nu;v)  &  = \sum_{i,j}\int_{2|v^{\prime}| > |v|} \nu_{i} \nu_{j}
\phi^{ij}(v-v^{\prime})\sqrt\mu(v^{\prime}) g(v^{\prime}) dv^{\prime} + \sum_{i,j}\int_{2|v^{\prime}| \le|v|} \nu_{i} \nu_{j} \phi
^{ij}(v-v^{\prime})\sqrt\mu(v^{\prime}) g(v^{\prime}) dv^{\prime}\\
&  = (I) + (II).
\end{split}
\]
Note that for $2|v^{\prime}| > |v|$, $\sqrt\mu(v^{\prime}) \le C^{\prime}%
\mu(v^{\prime}/4) \mu(v/4)$. Therefore,
\begin{equation}
\label{Eq : (I)}%
\begin{split}
|(I)|  &  \le C^{\prime}\mu\left(  \frac{v}{4}\right)  \|g\|_{\infty}|\nu
|^{2}\int_{\mathbb{R}^{3}} \phi^{ij}(v-v^{\prime}) \mu\left(  \frac{v^{\prime
}}{4}\right)  dv^{\prime}\\
&  \le C^{\prime-1}\mu\left(  \frac{v}{4}\right)  \|g\|_{\infty}%
|\nu|^{2}.
\end{split}
\end{equation}

To control $(II)$, we expand $\phi^{ij}(v-v^{\prime})$ to get
\[
\phi^{ij}(v-v^{\prime ij}(v) - \sum_{k}\partial_{k} \phi^{ij}(v)v_{k}^{\prime
}+ \frac{1}{2}\sum_{k,l}\partial_{kl}\phi^{ij}(\bar v)v_{k}^{\prime}%
v_{*}^{\prime},
\]
for some $\bar v$ in a line segment of $v$ and $v-v^{\prime}$. Then we
have
\[%
\begin{split}
(II)  &  = \sum_{i,j} \nu_{i} \nu_{j} \phi^{ij}(v)\int_{2|v^{\prime}| \le
|v|}\sqrt\mu(v^{\prime}) g(v^{\prime}) dv^{\prime}\\
&  \quad- \sum_{i,j} \nu_{i} \nu_{j} \sum_{k} \partial_{k}\phi^{ij}%
(v)\int_{2|v^{\prime}| \le|v|}v_{k}^{\prime}\sqrt\mu(v^{\prime}) g(v^{\prime})
dv^{\prime}\\
&  \quad+ \frac{1}{2}\sum_{i,j}\int_{2|v^{\prime}| \le|v|} \nu_{i} \nu_{j}
\sum_{k,l}\partial_{kl}\phi^{ij}(\bar v)v_{k}^{\prime}v_{*}^{\prime}\sqrt
\mu(v^{\prime}) g(v^{\prime}) dv^{\prime}\\
&  =(II)_{1} + (II)_{2} + (II)_{3}.
\end{split}
\]
Since
\[
\sum_{i} \phi^{ij}(v)v_{i} = \sum_{j} \phi^{ij}(v)v_{j} = 0,
\]
we have
\begin{equation}
\label{Eq : (II)_1}%
\begin{split}
|(II)_{1}|  &  = \big((I-P_{v}) \nu\big)^{T} \phi(v) \big((I-P_{v})
\nu\big) \int_{2|v^{\prime}|\le|v|}\sqrt\mu(v^{\prime}) g(v^{\prime})
dv^{\prime}\\
&  \le C\|g\|_{\infty}(1+|v|)^{-1}|(I-P_{v}) \nu|^{2}.
\end{split}
\end{equation}
Note that
\[
\sum_{i,j} \partial_{k} \phi^{ij}(v)v_{i}v_{j} = 0.
\]
Therefore
\begin{equation}
\label{Eq : (II)_2}%
\begin{split}
|(II)_{2}|  &  \le\left|  \sum_{k} \big((I-P_{v})\nu\big)^{T}\partial_{k}%
\phi(v) \big((I-P_{v})\nu\big) \int_{2|v^{\prime}| \le|v|} v_{k}^{\prime}%
\sqrt\mu(v^{\prime}) g(v^{\prime}) dv^{\prime}\right| \\
&  \quad+ 2\left|  \sum_{k} (P_{v}\nu)^{T}\partial_{k}\phi(v) \big((I-P_{v}%
)\nu\big) \int_{2|v^{\prime}| \le|v|} v_{k}^{\prime}\sqrt\mu(v^{\prime})
g(v^{\prime}) dv^{\prime}\right| \\
&  \le C\|g\|_{\infty}(1+|v|)^{-2}\left(  |(I-P_{v})\nu|^{2} + |P_{v}
\nu||(I-P_{v})\nu|\right) \\
&  \le C\|g\|_{\infty}\left(  (1+|v|)^{-3}|P_{v} \nu|^{2} + (1 +
|v|)^{-1}|(I-P_{v})\nu|^{2}\right)  .
\end{split}
\end{equation}
Since $\bar v$ is in $(v,v-v^{\prime})$ and $2|v^{\prime}| \le|v|$, we have
$|v|/2 \le|\bar v| \le3|v|/2$. Therefore, $\partial_{kl}\phi^{ij}(\bar v) \le
C^{\prime} |v|^{-3}$. Thus we have
\begin{equation}
\label{Eq : (II)_3}|(II)_{3}| \le C^{\prime}\|g\|_{\infty}(1+|v|)^{-3}%
|\nu|^{2}.
\end{equation}
Combining \eqref{Eq : (I)} - \eqref{Eq : (II)_3}, we have \eqref{Eq : D_(mu^1/2 g)}.

Now we can compute $D_{G}(\nu;v)$. Since $\varepsilon>0$ is a given small enough
constant, from \eqref{Eq : D_mu}, \eqref{Eq : D_(mu^1/2 g)}, we have
\[%
\begin{split}
D_{G}(\nu;v)  &  \ge\frac{1}{2C} \left(  (1+|v|)^{-3}|P_{v} \nu|^{2} +
(1+|v|)^{-1}|(I-P_{v})\nu|^{2}\right)  ,\\
D_{G}(\nu;v)  &  \le2C \left(  (1+|v|)^{-3}|P_{v} \nu|^{2} +
(1+|v|)^{-1}|(I-P_{v})\nu|^{2}\right)  .
\end{split}
\]
Therefore,
\[
\frac{1}{2C} (1+|v|)^{-3}\le\lambda\le2C(1+|v|)^{-1}.
\]

\end{proof}

\begin{lemma}
[Corollary 1 in \cite{guo2002landau}]\label{Lemma : Guo2002 corollary1} There
exists $c=c_{\vartheta}>0$, such that
\[%
\begin{split}
|g|_{\sigma,\vartheta}^2  &  \ge c\left\{  \left|  w^{\vartheta}[1+|v|]^{-3
/2}\left\{  P_{v} \partial_{i} g\right\}  \right|  _{2}^{2} + \left|
w^{\vartheta}[1+|v|]^{-1/2}\left\{  [I-P_{v}] \partial_{i} g\right\}
\right|  _{2}^{2} \right. \\
&  \quad\quad\quad+ \left.  \left|  w^{\vartheta}[1+|v|]^{-1/2}g\right|  _{2}^{2}\right\}.
\end{split}
\]

\end{lemma}

\begin{lemma}
[Lemma 5 in \cite{guo2002landau}]\label{Lemma : Guo2002 lemma5} 
Let $L$, $K$, and $\sigma^{i}$ be defined as in \eqref{Eq : L}, \eqref{Eq : sigma and sigma^i}, and \eqref{Eq : K}.
Let
$\vartheta\in\mathbb{R}$. For any $m>1$, there is $0<C(m)<\infty$, such that
\[%
\begin{split}
&  |\langle w^{2 \vartheta}\partial_{i} \sigma^{i} g_{1}, g_{2} \rangle| +
|\langle w^{2 \vartheta} Kg_{1}, g_{2} \rangle|\\
&  \quad\le\frac{C}{m}|g_{1}|_{\sigma,\vartheta}|g_{2}|_{\sigma,\vartheta} +
C(m) \left\{  \int_{|v|\le C(m)} |w^{\vartheta}g_{1}|^{2} dv \right\}  ^{1/2}
\left\{  \int_{|v|\le C(m)} |w^{\vartheta}g_{2}|^{2} dv\right\}  ^{1/2}.
\end{split}
\]
Moreover, there is $\delta>0$, such that
\[
\langle Lg,g \rangle\ge\delta|(I-\mathbf{P}) g|_{\sigma}^{2}.
\]

\end{lemma}

\begin{lemma}
[Lemma 6 in \cite{guo2002landau}]\label{Lemma : Guo2002 lemma6}
Let $L$, $A$, and $K$ be defined as in \eqref{Eq : L}, \eqref{Eq : A}, and \eqref{Eq : K}.
Let $\vartheta\in\mathbb{R}$ and $|\beta|\ge0$. For small $\delta>0$, there exists
$C_{\delta} = C_{\delta}(\vartheta)>0$ such that
\[
-\langle w^{2 \vartheta} Ag, g \rangle\ge|g|_{\sigma, \vartheta}^{2} -
\delta|g|_{\sigma, \vartheta}^{2} - C_{\delta}|\mu g|_{2}^{2},
\]
\[
|\langle w^{2 \vartheta} K g_{1}, g_{2} \rangle| \le\left\{  \delta
|g_{1}|_{\sigma,\vartheta} +C_{\delta}|\mu g_{1}|_{2} \right\}  |g_{2}%
|_{\sigma, \vartheta}.
\]
Thus we have
\[
\frac{1}{2}|g|_{\sigma,\vartheta}^{2} - C_{\vartheta}|g|_{\sigma}^{2}
\le\langle w^{2 \vartheta} Lg, g \rangle\le\frac{3}{2}|g|_{\sigma,\vartheta
}^{2} + C_{\vartheta}|g|_{\sigma}^{2}.%
\]

\end{lemma}

For the nonlinear estimate in Theorem 3 \cite{guo2002landau}, they estimated
\[
\left(  w^{2 \vartheta}\Gamma[g_{1},g_{2}], g_{3} \right)
\]
in terms of $\|g_{i}\|_{2, \vartheta}$ and $\|g_{i}\|_{\sigma, \vartheta}$ for
$i=1,2,3$ and $\vartheta\ge0$. To get such a $L^{2}$ estimate, they need a
higher-order regularity like $\|D_{\beta}^{\alpha} g_{i}\|_{2, \vartheta}$ and
$\|D_{\beta}^{\alpha} g_{i}\|_{\sigma, \vartheta}$ for $i=1,2,3$,
$\vartheta\ge0$, $|\alpha| + |\beta| \le N$ and $N\ge8$.

The following lemma is a refinement of Theorem 3 in \cite{guo2002landau}.
First, the range of $\vartheta$ is extended to $\mathbb{R}$. Second, we
estimate the nonlinear term in terms of $\|\cdot\|_{\infty}$, $\|\cdot\|_{2,
\vartheta}$, and $\|\cdot\|_{\sigma, \vartheta}$ without a higher-order regularity.

\begin{theorem}
\label{Thm : modified Guo2002 theorem3}
Let $\Gamma$ be defined as in \eqref{Eq : Gamma}.
\begin{enumerate}
\item For every $\vartheta\in\mathbb{R}$, there exists $C_{\vartheta}$ such
that
\begin{equation}
\label{Eq : nonlinear estimate for v}
	\langle w^{2 \vartheta} \Gamma[g_{1},g_{2}], g_{3} \rangle|\le C_{\vartheta}|g_{1}|_{\infty}|g_{2}|_{\sigma, \vartheta}
	|g_{3}|_{\sigma, \vartheta},
\end{equation}
and
\begin{equation}
\label{Eq : nonlinear estimate 1}%
	\left|  \left(  w^{2 \vartheta} \Gamma[g_{1},g_{2}], g_{3} \right)  \right| \le C_{\vartheta}\|g_{1}\|_{\infty}\|g_{2}\|_{\sigma, \vartheta}\|g_{3}\|_{\sigma, \vartheta}.
\end{equation}

\item There exists $\bar \vartheta<0$ such that for any $\vartheta
\le \bar \vartheta$,
\begin{equation}
\label{Eq : nonlinear estimate 2}
	\left|  \left(  w^{2 \vartheta} \Gamma[g_{1},g_{2}], g_{3} \right)  \right| \le C_{\vartheta}\min\{ \|g_{1}\|_{2, \vartheta}, \|g_{1}\|_{\sigma,\vartheta} \}(\|g_{2}\|_{\infty}+ \|D_{v} g_{2}\|_{\infty})\|g_{3}\|_{\sigma,\vartheta}.
\end{equation}

\end{enumerate}
\end{theorem}

\begin{proof}

\begin{enumerate}
\item By the integration by parts, we have
\begin{equation}
\label{Eq : gamma integration by parts 1}%
\begin{split}
|\langle w^{2 \vartheta} \Gamma[g_{1},g_{2}],g_{3} \rangle|  &  \le|
\langle\partial_{i} w^{2 \vartheta}\{\phi^{ij}*[\sqrt\mu g_{1}]\}\partial_{j}
g_{2}, g_{3} \rangle|\\
&  \quad+ | \langle w^{2 \vartheta}\{\partial_{j} \phi^{ij}*[\sqrt\mu g_{1}]\}
g_{2}, \partial_{i} g_{3} \rangle|\\
&  \quad+ | \langle\partial_{i} w^{2 \vartheta}\{\partial_{j}\phi^{ij}%
*[\sqrt\mu g_{1}]\} g_{2}, g_{3} \rangle|\\
&  \quad+ | \langle\partial_{i} w^{2 \vartheta}\{\phi^{ij}*[v_{j} \sqrt\mu
g_{1}]\} g_{2}, g_{3} \rangle|\\
&  \quad+ | \langle w^{2 \vartheta}\{\partial_{j} \phi^{ij}*[v_{i} \sqrt\mu
g_{1}]\} g_{2}, g_{3} \rangle|\\
&  \quad+ | \langle w^{2 \vartheta}\{\phi^{ij}*[ \partial_{j} \{v_{i} \sqrt
\mu\} g_{1}]\} g_{2}, g_{3} \rangle|\\
&  \quad+ | \langle w^{2 \vartheta}\{\phi^{ij}* [v_{i}\sqrt\mu g_{1}%
]\}\partial_{j} g_{2}, g_{3} \rangle|\\
&  \quad+ | \langle w^{2 \vartheta}\{\phi^{ij}*[ v_{j} \sqrt\mu g_{1}]\}
g_{2}, \partial_{i} g_{3} \rangle|\\
&  \quad+ | \langle w^{2 \vartheta}\{\phi^{ij}*[\sqrt\mu g_{1}]\}\partial_{j}
g_{2}, \partial_{i} g_{3} \rangle|\\
&  = (I) + (II) + \dots+ (IX),
\end{split}
\end{equation}
where $\phi$ is the matrix defined as in \eqref{Eq : phi}.
Clearly, $|\partial_{i} w^{2 \vartheta}| \le C_{\vartheta} (1+|v|)^{-1}w^{2
\vartheta}$ and by Lemma \ref{Lemma : Guo2002 lemma2}, we have
\[
|\phi^{ij} * [\sqrt\mu g_{1}] | + |\phi^{ij} * [v_{i} \sqrt\mu g_{1}] | +
|\phi^{ij} * [v_{j} \sqrt\mu g_{1}] | + |\phi^{ij} * [\partial_{j} \{v_{i}
\sqrt\mu\} g_{1}] | \le C(1+|v|)^{-1}\|g_{1}\|_{\infty},
\]
\[
|\partial_{j} \phi^{ij} * [\sqrt\mu g_{1}] | + |\partial_{j} \phi^{ij} *
[v_{i} \sqrt\mu g_{1}]| \le C(1+|v|)^{-2}\|g_{1}\|_{\infty}.
\]
Therefore, by Lemma \ref{Lemma : Guo2002 corollary1} and the H\"older
inequality,
\[%
\begin{split}
(I)  &  \le C_{\vartheta} |g_{1}|_{\infty}| \langle w^{\vartheta}
(1+|v|)^{-3/2}\partial_{j} g_{2}, w^{\vartheta} (1+|v|)^{-1/2} g_{3} \rangle|\\
&  \le C_{\vartheta} |g_{1}|_{\infty}|g_{2}|_{\sigma,\vartheta} |g_{3}%
|_{\sigma, \vartheta},
\end{split}
\]
\[%
\begin{split}
(II)  &  \le C |g_{1}|_{\infty}| \langle w^{\vartheta} (1+|v|)^{-1/2} g_{2}, w^{\vartheta} (1+|v|)^{-3/2} \partial_{i}
g_{3} \rangle|\\
&  \le C |g_{1}|_{\infty}|g_{2}|_{\sigma,\vartheta} |g_{3}|_{\sigma,
\vartheta},
\end{split}
\]
\[%
\begin{split}
(III) + (IV) + (V) + (VI)  &  \le C_{\vartheta}|g_{1}|_{\infty}| \langle
w^{\vartheta} (1+|v|)^{-1/2} g_{2}, w^{\vartheta} (1+|v|)^{-1/2} g_{3} \rangle|\\
&  \le C_{\vartheta} |g_{1}|_{\infty}|g_{2}|_{\sigma,\vartheta} |g_{3}%
|_{\sigma, \vartheta}.
\end{split}
\]
By \eqref{Eq : D_(mu^1/2 g)} and the H\"older inequality,
\[%
\begin{split}
(VII)  &  \le C|g_{1}|_{\infty}\int w^{2 \vartheta}\left|  (1+|v|)^{-3/2}|P_{v} \partial_{j} g_{2}| + (1+|v|)^{-1/2}%
|(I-P_{v}) \partial_{j} g_{2}| \right|  (1+|v|)^{-1/2} |g_{3}|
dv\\
&  \le C|g_{1}|_{\infty}|g_{2}|_{\sigma, \vartheta} |g_{3}|_{\sigma,
\vartheta},
\end{split}
\]
\[%
\begin{split}
(VIII)  &  \le C|g_{1}|_{\infty}\int w^{2 \vartheta} (1+|v|)^{-1/2} |g_{2}| \left|  (1+|v|)^{-3/2}|P_{v} \partial_{j} g_{3}| +
(1+|v|)^{-1/2}|(I-P_{v}) \partial_{j} g_{3}| \right|  dv\\
&  \le C|g_{1}|_{\infty}|g_{2}|_{\sigma, \vartheta} |g_{3}|_{\sigma,
\vartheta},
\end{split}
\]
and
\begin{multline*}
	(IX)  \le C|g_{1}|_{\infty}\int w^{2 \vartheta} \left|  (1+|v|)^{-3/2}|P_{v} \partial_{j} g_{2}| + (1+|v|)^{-1/2}%
	|(I-P_{v}) \partial_{j} g_{2}| \right| \\
	\times\left|  (1+|v|)^{-3/2}|P_{v} \partial_{j} g_{3}| + (1+|v|)^{-1/2}|(I-P_{v})\partial_{j} g_{3}| \right|  dv\\
	\le C|g_{1}|_{\infty}|g_{2}|_{\sigma, \vartheta} |g_{3}|_{\sigma,\vartheta}.
\end{multline*}
\[%
\begin{split}
\end{split}
\]
Thus we obtain \eqref{Eq : nonlinear estimate for v}. By applying the H\"older
inequality to \eqref{Eq : nonlinear estimate for v},
\begin{equation*}
	\begin{split}
		\left|  \left(  w^{2 \vartheta} \Gamma[g_{1},g_{2}], g_{3} \right)  \right|& = \int|\langle w^{2 \vartheta} \Gamma[g_{1},g_{2}], g_{3} \rangle|dx\\
		&  \le\int C_{\vartheta}|g_{1}|_{\infty}|g_{2}|_{\sigma, \vartheta}|g_{3}|_{\sigma, \vartheta} dx\\
		& \le C_{\vartheta}\|g_{1}\|_{\infty}\|g_{2}\|_{\sigma, \vartheta}\|g_{3}\|_{\sigma, \vartheta}.
	\end{split}
\end{equation*}

Thus we have \eqref{Eq : nonlinear estimate 1}.

\item By the integration by parts again, we have
\begin{equation}
\label{Eq : gamma integration by parts 2}%
\begin{split}
|\langle w^{2 \vartheta} \Gamma[g_{1},g_{2}], g_{3} \rangle|  &  := | \langle
w^{2 \vartheta} \{\phi^{ij}*[\mu^{1/2}g_{1}]\}\partial_{j} g_{2}, \partial_{i}
g_{3} \rangle|\\
&  \quad+|\langle w^{2 \vartheta} \{\phi^{ij}*[v_{i} \mu^{1/2} g_{1}]
\}\partial_{j} g_{2} , g_{3} \rangle|\\
&  \quad+ | \langle w^{2 \vartheta} \{\phi^{ij}*[\mu^{1/2}\partial_{j} g_{1}]
\} g_{2}, \partial_{i} g_{3} \rangle|\\
&  \quad+ |\langle w^{2 \vartheta} \{\phi^{ij}*[v_{i} \mu^{1/2} \partial_{j}
g_{1}] \}g_{2}, g_{3} \rangle|\\
&  \quad+ |\langle\partial_{i} w^{2 \vartheta} \{\phi^{ij}*[\mu^{1/2} g_{1}]
\} \partial_{j} g_{2}, g_{3} \rangle|\\
&  \quad+ | \langle\partial_{i} w^{2 \vartheta} \{\phi^{ij}*[\mu^{1/2}%
\partial_{j} g_{1}] \} g_{2}, g_{3} \rangle|\\
&  = (i) + (ii) + \cdots+ (vi).
\end{split}
\end{equation}
By the H\"older inequality and the integration by parts, we have
\[%
\begin{split}
&  |\langle\{\phi^{ij}*[\mu^{1/2}g_{1}]\} \rangle| + |\langle\{\phi
^{ij}*[v_{i} \mu^{1/2} g_{1}] \} \rangle| + |\langle\{\phi^{ij}*[\mu
^{1/2}\partial_{j} g_{1}] \} \rangle| + |\langle\{\phi^{ij}*[v_{i} \mu^{1/2}
\partial_{j} g_{1}] \} \rangle|\\
&  \quad\le C_{\vartheta} (1+|v|)^{-1}\min\{ |g_{1}|_{2, \vartheta
},|g_{1}|_{\sigma, \vartheta}\}.
\end{split}
\]
Let $\bar \vartheta := -2$, then by applying the H\"older inequality to $(i)$ and Lemma \ref{Lemma : Guo2002 corollary1}, we have
\[%
\begin{split}
(i)  &  \le C_{\vartheta} \min\{ |g_{1}|_{2, \vartheta},|g_{1}|_{\sigma,
\vartheta}\} \left|  \int_{\mathbb{R}^{3}} w^{2\vartheta} (1+|v|)^{-1}
\partial_{j} g_{2}(v) \partial_{i} g_{3}(v) dv\right| \\
&  \le C_{\vartheta} \min\{ |g_{1}|_{2, \vartheta},|g_{1}|_{\sigma, \vartheta
}\} |\partial_{j} g_{2}|_{\infty} \left(  \int_{\mathbb{R}^{3}} (1+|v|)^{2
\vartheta+ 1} dv\right)  ^{1/2} \left(  \int_{\mathbb{R}^{3}}
w^{2\vartheta} (1+|v|)^{-3} |\partial_{i} g_{3}|^{2} dv \right)  ^{1/2}\\
&  \le C_{\vartheta} \min\{ |g_{1}|_{2, \vartheta},|g_{1}|_{\sigma, \vartheta
}\} |D_{v} g_{2}|_{\infty}|g_{3}|_{\sigma, \vartheta}.
\end{split}
\]
Similarly,
\[
(ii) + (iii) + \cdots+ (vi) \le C_{\vartheta} \min\{ |g_{1}|_{2, \vartheta
},|g_{1}|_{\sigma, \vartheta}\}(|g_{2}|_{\infty}+ |D_{v} g_{2}|_{\infty
})|g_{3}|_{\sigma, \vartheta}.%
\]
Therefore, by the H\"older inequality again, we have
\begin{equation*}
	\begin{split}
		\left|  \left(  w^{2 \vartheta} \Gamma[g_{1}, g_{2}], g_{3} \right)	\right| & = \int| \langle w^{2 \vartheta} \Gamma[g_{1}, g_{2}], g_{3} \rangle|	dx\\
		&  \le\int C_{\vartheta} \min\{ |g_{1}|_{2, \vartheta},|g_{1}|_{\sigma,\vartheta}\}(|g_{2}|_{\infty} + |D_{v} g_{2}|_{\infty}) |g_{3}|_{\sigma,\vartheta} dx\\
		&  \le C_{\vartheta} \min\{ \|g_{1}\|_{2, \vartheta},\|g_{1}\|_{\sigma,\vartheta}\}(\|g_{2}\|_{\infty} + \|D_{v} g_{2}\|_{\infty}) \|g_{3}\|_{\sigma,\vartheta}.
	\end{split}
\end{equation*}

\end{enumerate}
\end{proof}

\begin{lemma}
\label{lemma : Kf estimate} 
Let $\bar K_g^\vartheta$ be defined as in \eqref{Eq : bar K_g^theta}.
Suppose that $g$ satisfies the assumption in Lemma
\ref{lemma : eigenvalue estimate}. Then there exists $C=C_{\vartheta}>0$ such that for every
$N,M>0$,
\begin{equation}
\label{Eq : Kf L^infty norm}\|\bar K_{g}^{\vartheta}f\|_{L^{\infty}%
(\mathbb{T}^{3} \times\mathbb{R}^{3})}\le C\|f^{\vartheta}\|_{L^{\infty
}(\mathbb{T}^{3} \times\mathbb{R}^{3})},
\end{equation}
\begin{equation}
\label{Eq : Kf L^infty norm unbdd}\|\bar K_{g}^{\vartheta}1_{|v|>M}%
f\|_{L^{\infty}(\mathbb{T}^{3} \times\mathbb{R}^{3})}\le C(1+M)^{-1}\|f^{\vartheta}\|_{L^{\infty}(\mathbb{T}^{3} \times\mathbb{R}^{3})},
\end{equation}
and
\begin{equation}
\label{Eq : Kf L2 norm}\|\bar K_{g}^{\vartheta}f\|_{L^{2}(\mathbb{T}^{3}
\times\mathbb{R}^{3})} \le C N^{2}\|f^{\vartheta}%
\|_{L^{2}(\mathbb{T}^{3} \times\mathbb{R}^{3})} + \frac{C}%
{N}\|f^{\vartheta}\|_{L^{\infty}(\mathbb{T}^{3}
\times\mathbb{R}^{3})}.
\end{equation}
\end{lemma}

\begin{proof}
Since
\[%
\begin{split}
\bar K^{\vartheta}_{g} f  &  =w^{\vartheta}\bar K_{g} f + \left(  2
\frac{\partial_{i} w^{\vartheta}\partial_{j} w^{\vartheta}}{w^{2 \vartheta}%
}\sigma_{G}^{ij} -\frac{\partial_{ij}w^{\vartheta}}{w^{\vartheta}}\sigma
_{G}^{ij} - \frac{\partial_{j}w^{\vartheta}}{w^{\vartheta}}\partial_{i}%
\sigma_{G}^{ij} - \frac{\partial_{i} w^{\vartheta}}{w^{\vartheta}}a_{g}^{i}
\right)  f^{\vartheta}\\
&  = w^{\vartheta}Kf + \partial_{i} \sigma^{i} f^{\vartheta}-v\cdot \sigma v f^{\vartheta}-\partial
_{i}\left\{  \phi^{ij}*[\mu^{1/2}\partial_{j} g] \right\}  f^{\vartheta}+
\left\{  \phi^{ij}*[v_{i} \mu^{1/2} \partial_{j} g] \right\}  f^{\vartheta}\\
&  \quad+\left(  2 \frac{\partial_{i} w^{\vartheta}\partial_{j} w^{\vartheta}%
}{w^{2 \vartheta}}\sigma_{G}^{ij} -\frac{\partial_{ij}w^{\vartheta}%
}{w^{\vartheta}}\sigma_{G}^{ij} - \frac{\partial_{j}w^{\vartheta}%
}{w^{\vartheta}}\partial_{i}\sigma_{G}^{ij} - \frac{\partial_{i} w^{\vartheta
}}{w^{\vartheta}}a_{g}^{i} \right)  f^{\vartheta},%
\end{split}
\]
where $\phi$, $K$, $\sigma_G$, $\sigma^i$ are defined as in \eqref{Eq : phi}, \eqref{Eq : K}, \eqref{Eq : sigma_u}, and \eqref{Eq : sigma and sigma^i} with $G = \mu + \mu^{1/2}g$,
by Lemma \ref{Lemma : Guo2002 lemma2} and \ref{Lemma : Guo2002 lemma3}, we
have
\[
|\partial_{i} \sigma^{i}(v)| + |v\cdot \sigma v| + \left|  \partial_{i}\left\{  \phi^{ij}%
*[\mu^{1/2}\partial_{j} g](v) \right\}  \right|  + \left|  \left\{  \phi
^{ij}*[v_{i} \mu^{1/2} \partial_{j} g](v) \right\}  \right|   \le
C(1+|v|)^{-1},
\]
\[
\left|  2 \frac{\partial_{i} w^{\vartheta}\partial_{j} w^{\vartheta}}{w^{2
\vartheta}}\sigma_{G}^{ij}\right|  +\left|  \frac{\partial_{ij}w^{\vartheta}%
}{w^{\vartheta}}\sigma_{G}^{ij}\right|  + \left|  \frac{\partial
_{j}w^{\vartheta}}{w^{\vartheta}}\partial_{i}\sigma_{G}^{ij}\right|  + \left|
\frac{\partial_{i} w^{\vartheta}}{w^{\vartheta}}a_{g}^{i} \right|  \le
C(1+|v|)^{-1}.%
\]
Thus it is sufficient to show that $w^{\vartheta}Kf$ also satisfies
\eqref{Eq : Kf L^infty norm} - \eqref{Eq : Kf L2 norm}.

After the integration by parts, we have
\[%
\begin{split}
w^{\vartheta}Kf  &  = - w^{\vartheta}\mu^{-1/2}\partial_{i} \left\{
\mu\left[  \phi^{ij}* \left\{  \mu^{1/2}[\partial_{j} f + v_{j} f] \right\}
\right]  \right\} \\
&  = 2w^{\vartheta}v_{i} \mu\left[  \phi^{ij}* \left\{  \mu^{1/2}[\partial_{j}
f + v_{j} f] \right\}  \right]  - w^{\vartheta}\mu^{1/2}\left[  \partial_{i}
\phi^{ij}* \left\{  \mu^{1/2}[\partial_{j} f + v_{j} f] \right\}  \right] \\
&  =2w^{\vartheta}v_{i} \mu\left[  \phi^{ij}*(v_{j} \mu^{1/2} f)\right]  -
2w^{\vartheta}v_{i} \mu\left[  \phi^{ij}*(\partial_{j}\mu^{1/2} f)\right]  +
2w^{\vartheta}v_{i} \mu\left[  \partial_{j} \phi^{ij}*(\mu^{1/2} f)\right] \\
&  \quad-w^{\vartheta}\mu^{1/2}\left[  \partial_{i} \phi^{ij}* (v_{j}\mu
^{1/2}f)\right]  + w^{\vartheta}\mu^{1/2}\left[  \partial_{i} \phi^{ij}*
(\partial_{j}\mu^{1/2}f)\right]  -w^{\vartheta}\mu^{1/2}\left[  \partial_{ij}
\phi^{ij}* (\mu^{1/2}f)\right] \\
&  =4w^{\vartheta}v_{i} \mu\left[  \phi^{ij}*(v_{j} \mu^{1/2} f)\right]  +
2w^{\vartheta}v_{i} \mu\left[  \partial_{j} \phi^{ij}*(\mu^{1/2} f)\right] \\
&  \quad-w^{\vartheta}2\mu^{1/2}\left[  \partial_{i} \phi^{ij}* (v_{j}%
\mu^{1/2}f)\right]  - w^{\vartheta}\mu^{1/2}\left[  \partial_{ij} \phi^{ij}*
(\mu^{1/2}f)\right] \\
&  =4w^{\vartheta}v_{i} \mu\left[  \phi^{ij}*(v_{j} w^{-\vartheta}\mu^{1/2}
f^{\vartheta})\right]  + 2w^{\vartheta}v_{i} \mu\left[  \partial_{j} \phi
^{ij}*(w^{-\vartheta}\mu^{1/2} f^{\vartheta})\right] \\
&  \quad-w^{\vartheta}2\mu^{1/2}\left[  \partial_{i} \phi^{ij}* (v_{j}
w^{-\vartheta}\mu^{1/2}f^{\vartheta})\right]  - w^{\vartheta}\mu^{1/2}\left[
\partial_{ij} \phi^{ij}* (w^{-\vartheta}\mu^{1/2}f^{\vartheta})\right]\\
&=(I) + (II) + (III) + (VI).
\end{split}
\]
Applying Lemma \ref{Lemma : Guo2002 lemma2} to $(I) + (II) + (III)$, we have $(I) + (II) + (III) \le C\|f^{\vartheta}\|_{L^{\infty}(\T^3 \times \R^3)}$.
Note that $\partial_{ij}\phi^{ij}*(\mu^{1/2}f) = -8 \pi\mu^{1/2}f$.
Thus we also have $(VI) \le C\|f^{\vartheta}\|_{L^{\infty}(\T^3 \times \R^3)}$.
Therefore, we have \eqref{Eq : Kf L^infty norm}.

Since every convolution term of $Kf$ contains $\mu^{1/2}$, we have \eqref{Eq : Kf L^infty norm unbdd}.

For \eqref{Eq : Kf L2 norm}, clearly we have
\begin{equation}	\label{Eq : unbdd term}
	\|1_{|v|\ge N}w^{\vartheta}K f\|_{2} \le\frac{C}{N}
\|w^{\vartheta}f\|_{\infty}.
\end{equation}
Now we will estimate $\|1_{|v|<N}w^{\vartheta}K f\|_{2}$.
First, consider $ \left\|1_{|v|<N} v_{i} w^{\vartheta}\mu\left[  \partial_{j} \phi^{ij}*w^{-\vartheta}\mu^{1/2} f^{\vartheta})\right]  \right\|_{L^{2}(\T^3 \times \R^3)}$. 
\begin{align*}
	&\left\|1_{|v|<N} v_{i} w^{\vartheta}\mu\left[  \partial_{j} \phi^{ij}*w^{-\vartheta}\mu^{1/2} f^{\vartheta})\right]  \right\|_{2}^{2}\\
	&  = \int_{\mathbb{T}^{3}}\int_{|v|<N} \left(  \int\partial_{j} \phi^{ij}(v-v^{\prime})v_{i}w^{\vartheta}(v)\mu(v) \mu^{1/2}(v^{\prime})w^{- \vartheta}(v^{\prime})f^{\vartheta}(v^{\prime}) dv^{\prime}\right)  ^{2} dv dx\\
	&  \le C \Bigg(\int_{\mathbb{T}^{3}} \int_{|v|<N} \left(  \int_{1/N<|v-v^{\prime}|<2N}dv^{\prime}\right)  ^{2} dvdx\\
	&  \quad+ \int_{\mathbb{T}^{3}}\int_{|v|<N} \left(  \int_{|v-v^{\prime}|>2N}dv^{\prime}\right)  ^{2} dv dx\\
	&  \quad+ \int_{\mathbb{T}^{3}}\int_{|v|<N} \left(  \int_{|v-v^{\prime}|<1/N}dv^{\prime}\right)  ^{2} dv dx\Bigg)\\
	&  =(i) + (ii) + (iii).
\end{align*}
Since $|\partial_{j} \phi^{ij}(v-v^{\prime})|\le C|v-v^{\prime}|^{-2}$, by
the Minkowski and H\"older inequality,
\begin{align*}
	(i) & \le\int_{\mathbb{T}^{3}}\left(  \int_{|v^{\prime}|<3N}\left(\int_{1/N<|v-v^{\prime}|<2N}|v-v^{\prime}|^{-4} w^{2 \vartheta}(v)v_{i}^2\mu^{2}(v)w^{-2 \vartheta}(v^{\prime})\mu(v^{\prime})(f^{\vartheta})^{2}(v^{\prime})dv\right)^{1/2}dv^{\prime}\right)  ^{2} dx\\
	&  = \int_{\mathbb{T}^{3}}\left(  \int_{|v^{\prime}|<3N}w^{- \vartheta}(v^{\prime})\mu^{1/2}(v^{\prime})f^{\vartheta}(v^{\prime})\left(  \int_{1/N<|v-v^{\prime}|<2N}|v-v^{\prime}|^{-4} w^{2 \vartheta}(v)v_{i}^2\mu^{2}(v) dv\right)  ^{1/2}dv^{\prime}\right)  ^{2}dx\\
	&  \le C\int_{\mathbb{T}^{3}}\left(  \int_{|v^{\prime}|<3N} \mu^{1/2}(v^{\prime})(f^{\vartheta})^{2}(v^{\prime})dv^{\prime}\right)  \left(\int_{|v^{\prime}|<3N}\int_{1/N<|v-v^{\prime}|<2N}|v-v^{\prime}|^{-4}w^{2\vartheta}(v) v_{i}^{2}\mu^{2}(v) dvdv^{\prime}\right)  dx\\
	&  \le C N^{4}\|f^{\vartheta}\|_{2}^{2}.
\end{align*}
Note that if $|v|<N$ and $|v-v^{\prime}|>2N$, then $|v^{\prime}|>N$. Since the
integrand of $(ii)$ contains a Maxwellian and $|v^{\prime}|>N$, for every
$\beta>0$ we have
\[
(ii) \le\frac{C_{\beta}}{N^{2\beta}}\|f^{\vartheta}\|_{\infty}^{2}.
\]
Finally,
\[
(iii) \le C \|f^{\vartheta}\|_{\infty}^{2} \iint_{\mathbb{T}^{3}
\times[0,\infty)}\int_{\mathbb{R}^{3}} \mu(v)\left(  \int_{|v-v^{\prime}%
|<1/N}|v-v^{\prime}|^{-2}dv^{\prime}\right)  ^{2}dv dxdt \le C\frac{1}{N^{2}}\|f^{\vartheta}\|_{\infty}^{2}.
\]
So we have
\begin{equation}		\label{Eq : bdd 3rd term}
\left\|  1_{|v|<N} w^{\vartheta}\mu^{1/2} \left[  \partial_{ij} \phi^{ij}*
w^{-\vartheta}\mu^{1/2} f^{\vartheta}\right]  \right\|  _{2} \le CN^{2}\|f^{\vartheta}\|_{2} + \frac{C}{N}\|f^{\vartheta}\|_{\infty}.
\end{equation}
In a similar manner,
\begin{equation}		\label{Eq : bdd 2nd term}
 	\left\|1_{|v|<N} w^{\vartheta}\mu^{1/2}\left[  \partial_{i} \phi^{ij}* (v_{j} w^{-\vartheta}\mu^{1/2}f^{\vartheta})\right] \right\|  _{2} \le CN^{2}\|f^{\vartheta}\|_{2} + \frac{C}{N}\|f^{\vartheta}\|_{\infty}
\end{equation}
and
\begin{equation} 	\label{Eq : bdd 1st term}
	\left\|  1_{|v|<N} v_{i} w^{\vartheta}\mu\left[  \phi^{ij}*(v_{j} w^{-\vartheta}\mu^{1/2} f^{\vartheta})\right]  \right\|  _{2}\le CN\|f^{\vartheta}\|_{2} + \frac{C}{N^{2}}\|f^{\vartheta}\|_{\infty}.
\end{equation}
Note that 
\begin{align*}
	w^{\vartheta}\mu^{1/2}\left[\partial_{ij}\phi^{ij}*(w^{-\vartheta}\mu^{1/2}f^{\vartheta})\right]& = w^{\vartheta}\mu^{1/2}\left[-8 \pi w^{-\vartheta}\mu^{1/2}f^{\vartheta} \right]\\
	&=-8 \pi \mu f^{\vartheta}.
\end{align*}
Thus,
\begin{equation}		\label{Eq : 4th term}
	\left\| 1_{|v|<N} w^{\vartheta}\mu^{1/2}\left[\partial_{ij}\phi^{ij}*(w^{-\vartheta}\mu^{1/2}f^{\vartheta})\right]\right\|_{2} \le C \|f^{\vartheta}\|_{2}.
\end{equation}
From \eqref{Eq : unbdd term} - \eqref{Eq : 4th term}, we have \eqref{Eq : Kf L2 norm}

So the proof is complete.
\end{proof}

\section{Maximum Principle} % (fold)
\label{sec:maximum_principle}
In this section, we first define a weak solution for \eqref{Eq : equation for h} and obtain the well-posedness and the maximum principle of the weak solution for \eqref{Eq : equation for h}.
Due to the lack of regularity, we cannot use a direct contradiction argument for the weak solution as in the case of strong solutions.
Therefore, we first construct a smooth approximated solution and then pass to the limit to obtain the maximum principle for the weak solution.
\begin{definition}
\label{Def : weak sol for h} Let $h(t,x,v) \in L^{\infty}((0,\infty)\times
\mathbb{T}^{3} \times\mathbb{R}^{3},w^{\vartheta}(v)dtdxdv)$ be a periodic function in
$x\in\mathbb{T}^{3} = [-\pi,\pi]^{3}$ satisfying
\begin{equation*}
	\int_{0}^{t}\iint_{\T^3 \times \R^3} \big(\sigma^{ij}\partial_{i}h \partial_{j}h\big)(s,x,v) dxdvds< \infty,
\end{equation*} 
where $\sigma$ is defined as in \eqref{Eq : sigma and sigma^i}.
We say that $h$ is a weak solution of
\eqref{Eq : equation for h}, with $h(0) = h_0$ on $(0,\infty
)\times\mathbb{T}^{3} \times\mathbb{R}^{3}$ if for all $t\in(0,\infty)$ and
all $\varphi\in C^{1,1,1}_{t,x,v}\left(  (0,\infty)\times\mathbb{T}^{3}
\times\mathbb{R}^{3}\right)  $ such that $\varphi(t,x,v)$ is a periodic
function in $x\in\mathbb{T}^{3} = [-\pi,\pi]^{3}$ and $\varphi(t,x,\cdot)$ is
compactly supported in $\mathbb{R}^{3}$, it satisfies
\begin{multline}
\label{Eq : weak sol for h}
\iint_{\mathbb{T}^{3} \times\mathbb{R}^{3}} h(t,x,v) \varphi(t,x,v) dxdv -
\iint_{\mathbb{T}^{3} \times\mathbb{R}^{3}} h_{0}(x,v) \varphi(0,x,v) dxdv\\
= \iiint_{(0,t)\times\mathbb{T}^{3} \times\mathbb{R}^{3}}
h(s,x,v)\bigg(\partial_{s} \varphi + v \cdot\nabla_{x}\varphi - \left(a_g + 2 \frac{\nabla w^{\vartheta}}{w^{\vartheta}} \sigma_{G}\right)\cdot \nabla_v \varphi\bigg)(s,x,v)\\
- \nabla_{v}h(s,x,v) \cdot (\sigma_{G} \nabla_{v}\varphi)(s,x,v) dsdxdv,
\end{multline}
where $\sigma_G$ is defined as in \eqref{Eq : sigma_u} with $G = \mu + \mu^{1/2} g$.
\end{definition}
\begin{lemma}		\label{Lemma : energy estimate approximated bar A^theta} 
	Assume \eqref{Eq : condition for g}.
	Let $\sigma_G$ be the matrix defined as in \eqref{Eq : sigma_u} with $G=\mu + \mu^{1/2}g$.
	Let $\vartheta\in \N \cup \{0\}$, $\delta\ge 0$, and $h$ be a classical solution of \eqref{Eq : equation for h}.
	Then there exist $C = C(\vartheta)$, $0<\varepsilon\ll 1$ such that if $\|g\|_{\infty} < \varepsilon$, then
	\begin{equation}		\label{Eq : energy estimate h}
		\sup_{0\le s\le t}\|h(s)\|_{L^{2}}^{2}+\int_{0}^{t}\iint_{\T^3 \times \R^3} \big(\sigma^{ij}\partial_{i}h \partial_{j}h\big)(s,x,v) dxdvds\leq C(t)\|h(0)\|_{L^{2}}^{2}.
	\end{equation}
\end{lemma}
\begin{proof}
	Multiplying \eqref{Eq : equation for h} by $h$ and integrating both sides of the resulting equation, we have
	\begin{equation}		\label{Eq : energy estimate h-1}
		\iint_{\T^3 \times \R^3} \frac{1}{2} \left( h^2(t,x,v) - h^2(0,x,v) \right) dxdv = \int_0^t \iint_{\T^3 \times \R^3} (\bar A_{g}^{\vartheta} h(s,x,v) )h(s,x,v) dxdvds.
	\end{equation}
	By Lemma \ref{lemma : eigenvalue estimate}, we have
	\begin{equation}		\label{Eq : energy estimate h-2}
		C^{-1} \sigma^{ij}\partial_{i}h \partial_{j}h \le \sigma_{G}^{ij}\partial_{i}h \partial_{j}h \le C\sigma^{ij}\partial_{i}h \partial_{j}h,
	\end{equation}
	for some $C$.
	By Lemma \ref{lemma : eigenvalue estimate} and the Young inequality, we have
	\begin{equation}		\label{Eq : energy estimate h-3}
		\begin{split}
			\left| \frac{\partial_{i}w^{\vartheta}}{w^{\vartheta}}\sigma_{G}^{ij} (\partial_{j}h) h\right|& \le C (1+|v|)^{-1} (\sigma_{G}^{ij}\partial_{i}h\partial_{j}h)^{1/2}(\sigma_{G}^{ij}h^2)^{1/2}\\
			&\le \varepsilon \sigma_{G}^{ij}\partial_{i}h\partial_{j}h + C_{\varepsilon} (1+|v|)^{-1}\sigma_{G}^{ij}h^2\\
			&\le \varepsilon \sigma^{ij}\partial_{i}h\partial_{j}h + C_{\varepsilon} h^2.
		\end{split}
	\end{equation}
	In a similar manner, by \eqref{Eq : D_(mu^1/2 g)} and the Young inequality, we have
	\begin{equation}		\label{Eq : energy estimate h-4}
		\begin{split}
			\left| \left\{\phi * [v_i \mu^{1/2}g] \right\} (\partial_{i}h) h \right| & \le C\|g\|_{\infty} D_{\mu}(\nabla_{v} h ; v)^{1/2}(\sigma^{ij}h^2)^{1/2}\\
			&\le \varepsilon \sigma^{ij}\partial_{i}h\partial_{j}h + \varepsilon h^2
		\end{split}
	\end{equation}
	and
	\begin{equation}		\label{Eq : energy estimate h-5}
		\left| \left\{  \phi^{ij}\ast\lbrack\mu^{1/2}\partial_{j}g]\right\}	\partial_{i}h h\right| \le \varepsilon \sigma^{ij}\partial_{i}h\partial_{j}h + \varepsilon h^2.
	\end{equation}
	Thus from \eqref{Eq : energy estimate h-1} - \eqref{Eq : energy estimate h-5}, we have
	\begin{multline*}
		\iint_{\T^3 \times \R^3} h^2(t,x,v)dxdv +  \int_{0}^{t}\iint_{\T^3 \times \R^3} \big(\sigma^{ij}\partial_{i}h \partial_{j}h\big)(s,x,v) dxdvds \\
		\le \iint_{\T^3 \times \R^3} h^2(0,x,v)dxdv + \varepsilon\int_{0}^{t}\iint_{\T^3 \times \R^3} \big(\sigma^{ij}\partial_{i}h \partial_{j}h\big)(s,x,v) dxdvds + C_{\varepsilon} \int_{0}^{t}\iint_{\T^3 \times \R^3} h^2(s,x,v) dxdvds.
	\end{multline*}
	Absorbing the second term of the RHS to the LHS and applying the Gronwall inequality to the resulting equation, we have \eqref{Eq : energy estimate h}.
\end{proof}
\begin{lemma}
\label{lemma : existence of weak sol h}
	Assume \eqref{Eq : condition for g}. Then there exists a unique weak solution to \eqref{Eq : equation for h} which satisifes \eqref{Eq : energy estimate h}
\end{lemma}
\begin{proof}
	We approximate $g$ by $g^{\delta} \in C^{\infty}$ and $h_{0}$ by $h^{\delta}_{0}\in C^\infty$ such that $\|g^{\delta}\|_{\infty }\leq \|g\|_{\infty }$, $\|h^{\delta}_{0}\|_{\infty }\leq \|h_{0}\|_{\infty }$ and 
	\begin{equation*}
		\lim \|g^{\delta}-g\|_{\infty }=0,\quad \lim \|h^{\delta}_{0}-h_{0}\|_{1}=0.
	\end{equation*}
	Consider
	\begin{equation}		\label{Eq : h delta}
		\partial_{t} h^{\delta}+v\cdot\nabla_{x}h^{\delta}=\bar{A}_{g^{\delta} }^{\vartheta}h^{\delta}
	\end{equation}
	\begin{equation*}
		h^{\delta}(0,x,v) = h^{\delta}_{0}(x,v).
	\end{equation*}
	By Lemma \ref{lemma : eigenvalue estimate}, $\sigma_G \ge 0$.
	Since $\sigma_G \ge 0$, it is rather standard (for instance, by adding regularization $\varepsilon(\nabla_{x,v})^{2m}$, for some large integer $m$, then letting $\varepsilon \rightarrow 0$, if necessary) that there exists a solution $h^{\delta}$ to the linear equation \eqref{Eq : h delta}.
	Since $g^{\delta}$ and $h^{\delta}_{0}$ are smooth, we can derive a similar energy estimate for the derivatives of $h^{\delta}$ by taking derivatives of the above equation and multiplying by the derivatives of $h^{\delta}$ and integrating both sides of the resulting equation as in \cite{guo2002landau}.
	For more details, see \cite{guo2002landau}.
	Therefore, $h^{\delta}$ is smooth.

	By \eqref{Eq : energy estimate h}, $\|h^{\delta}(s)\|_{L^2}^2$ is uniformly bounded on $0\le s\le t$.
	Therefore, there exists $h$ such that $h^{\delta}\rightarrow h$ weakly in $L^2$.
	Multiplying \eqref{Eq : h delta} by a test function $\varphi$, integrating both sides of the resulting equation, and taking the integration by parts, we have
	\begin{equation}
	\label{Eq : h delta weak form}
		\begin{split}
			&  \iint_{\mathbb{T}^{3} \times\mathbb{R}^{3}} h^{\delta}(t,x,v) \varphi(t,x,v) dxdv -
			\iint_{\mathbb{T}^{3} \times\mathbb{R}^{3}} h^{\delta}_{0}(x,v) \varphi(0,x,v) dxdv\\
			& = \iiint_{(0,t)\times\mathbb{T}^{3} \times\mathbb{R}^{3}}
			h^{\delta}(s,x,v)\bigg(\partial_{s} \varphi(s,x,v) + v \cdot\nabla_{x}\varphi(s,x,v)- \left(a_{g^{\delta}} + 2 \frac{\nabla w^{\vartheta}}{w^{\vartheta}} \sigma_{G^{\delta}}\right)\cdot \nabla_v \varphi
			\\
			& \quad \quad \quad \quad \quad \quad + \nabla_v \cdot (\sigma_{G^{\delta}} \nabla_v \varphi) dsdxdv,
		\end{split}
	\end{equation}
	where $G^{\delta} = \mu + \sqrt \mu g^{\delta}$.
	Since $h^{\delta} \rightarrow h$ weakly in $L^2$, taking $\delta \rightarrow 0$ in \eqref{Eq : h delta weak form} we have \eqref{Eq : weak sol for h}.
	Therefore $h$ is a weak solution of \eqref{Eq : equation for h}.
	The second assertion is an analogue of Lemma \ref{Lemma : energy estimate approximated bar A^theta}.
	Let $h$ and $\tilde h$ be weak solutions to \eqref{Eq : equation for h}.
	Then $h-\tilde h$ is also a weak solution to \eqref{Eq : equation for h} with zero initial data.
	Therefore, we have $\sup_{0\le s \le t} \|(h-\tilde h)(s)\|_{L^2}^2 = 0$.
	Thus we obtain the uniqueness.
\end{proof}

Before we derive the maximum principle for weak solutions, we establish the maximum principle for strong solutions.
We first derive the maximum principle for strong solutions in bounded domains. The following
technique is similar to that in \cite{hwang2014fokker}.

\begin{lemma}
\label{lemma : max. principle bdd domain}
Assume \eqref{Eq : condition for g}. 
Let $h\in C^{1,1,2}_{t,x,v}\left(
[0,T] \times\mathbb{T}^{3} \times B(0;M) \right)  $ be a periodic function
satisfying $\mathcal{M}_{g}^{\vartheta}h \le0$. Then $h$ attains its maximum
only at $t=0$ or $|v|=M$.
\end{lemma}

\begin{proof}
Let us assume that $\max_{(t,x,v)\in[0,T] \times\mathbb{T}^{3} \times B(0;M)}
h(t,x,v) > 0$ and $\mathcal{M}_{g}^{\vartheta}h <0$. Suppose that $h$ attains
its maximum at an interior point $(t,x,v) \in[0,T] \times\mathbb{T}^{3} \times
B(0;M)$ or at $(T,x,v)$ with $(x,v)$ lying in the interior. Since $\sigma_G \ge 0$ by Lemma \ref{lemma : eigenvalue estimate}, we have $\partial_{t}
h \ge0$, $\nabla_{x} h =0$, and $\nabla_{v} h =0$ while $\sigma_{G}^{ij}\partial_{ij} h \le0$ and $h(t,x,v)>0$. Thus $\mathcal{M}_{g}^{\vartheta
}h(t,x,v) \ge0$ and this gives a contradiction. Suppose $h$ attains its
maximum at $|x| = \pi$. Since $h$ is periodic in $x$, we can assume that $h$
attains its maximum at $x = \pi$, $v\ge0$ or $x= -\pi$, $v\le0$. Then 
$\partial_{t} h = 0$, $v\cdot\nabla_{x} h \ge0$, and $\nabla_{v} h = 0$
while $\sigma_{G}^{ij}\partial_{ij} h \le0$ and $h(t,x,v)>0$. Thus
$\mathcal{M}_{g}^{\vartheta}h(t,x,v) \ge0$ which makes a contradiction too.

In the case of $\mathcal{M}_{g}^{\vartheta}h \le0$, define $h^{k} := h-kt$ for
$k>0$, then $\mathcal{M}_{g}^{\vartheta}h^{k} <0$. Thus we have
\[
\sup_{(t,x,v)\in[0,T] \times\mathbb{T}^{3} \times B(0;M)} h^{k}(t,x,v) =
\sup_{t=0 \text{ or }|v|=M} h^{k}(t,x,v).
\]
Taking $k \rightarrow0$, we complete the proof.
\end{proof}

\begin{lemma}
\label{lemma : barrier function}
Assume \eqref{Eq : condition for g}.
There exists $\varphi\in C^{1,1,2}\left(
[0,T]\times\mathbb{T}^{3} \times\mathbb{R}^{3}\right)  $ with $\varphi\ge0$,
which satisfies $\mathcal{M}_{g}^{\vartheta}\varphi\ge0$ and $\varphi
\rightarrow\infty$ as $|v| \rightarrow\infty$ uniformly in $t \in[0,T]$.
\end{lemma}

\begin{proof}
Define,
\begin{equation}
\label{Eq : barrier}\varphi(t,x,v):= \varphi(t,v) = \alpha_{1}(t) + \alpha
_{2}(t) |v|^{2}.
\end{equation}
Then
\[%
\begin{split}
\mathcal{M}_{g}^{\vartheta}\varphi &  = \alpha_{1}^{\prime}(t) + \alpha
_{2}^{\prime }(t)|v|^{2}- 2 \alpha_{2}(t) \nabla_{v}\cdot(\sigma_{G} v)\\
&  \quad - 2 \alpha_{2}(t)
a_{g}\cdot v -2 \alpha_{2}(t) \frac{\partial_{i} w^{\vartheta}}{w^{\vartheta}%
}\sigma^{ij}v_{j}\\
&  = \alpha_{1}^{\prime}(t) + \alpha_{2}^{\prime }(t)|v|^2- 2 \alpha_{2}(t)
\partial_{i} \sigma_{G}^{ij}v_{j} - 2 \alpha_{2}(t) \sigma_{G}^{ii}\\
&  \quad - 2
\alpha_{2}(t)a_{g}\cdot v-2 \alpha_{2}(t) \frac{\partial_{i} w^{\vartheta}%
}{w^{\vartheta}}\sigma^{ij}v_{j}.
\end{split}
\]
Note that
\[
|\partial_{i} \sigma_{G}^{ij}| \le C \|g\|_{\infty}(1+|v|)^{-2},
\quad|\sigma_{G}^{ii}|\le C \|g\|_{\infty}(1+|v|)^{-1},
\]
\[
|a_{g}| \le C \|g\|_{\infty}(1+|v|)^{-1},\text{ and }
\frac{\partial_{i} w^{\vartheta}}{w^{\vartheta}}\sigma^{ij}v_{j} \le C
\|g\|_{\infty}(1+|v|)^{-3}.
\]
Choose $\alpha_{1}(t) = \alpha_{2}(t) := \exp(k t)$. Then $\varphi\ge0$.
Moreover,
\begin{equation*}
	\mathcal{M}_{g}^{\vartheta}\varphi  \ge e^{kt} \left(  k(1+|v|)^{2} - C(1+|v|)^{-1} - C(1+|v|)^{-1}(1+|v|) \right) \ge0
\end{equation*}
for a sufficiently large $k$.
\end{proof}

\begin{lemma}
\label{lemma : strong maximum principle}
Assume \eqref{Eq : condition for g}.
Let $h\in C^{1,1,2}_{t,x,v}\left(  [0,T]
\times\mathbb{T}^{3} \times\mathbb{R}^{3} \right)  $ be a periodic function
satisfying $\mathcal{M}_{g}^{\vartheta}h \le0$. Then $h$ attains its maximum
only at $t=0$.
\end{lemma}

\begin{proof}
Fix $\lambda>0$. Let $\varphi$ be a barrier function obtained in Lemma
\ref{lemma : barrier function}. Define $\eta^{\lambda}(t,x,v) := h(t,x,v) -
\lambda\varphi(t,x,v)$, then $\mathcal{M}_{g}^{\vartheta}\eta^{\lambda} \le0$.
Thus we can apply Lemma \ref{lemma : max. principle bdd domain} on the domain
$[0,T]\times\mathbb{T}^{3} \times B(0;M)$. Then we have
\[
\eta^{\lambda}(t,x,v) \le\sup_{t=0 \text{ or } |v|=M} \eta^{\lambda}(t,x,v),
\quad\text{for }(t,x,v)\in[0,T]\times\mathbb{T}^{3} \times B(0;M).
\]
Note that
\[
\eta^{\lambda}(0,x,v) = h(0,x,v) - \lambda\varphi(0,x,v) \le h(0,x,v) \le
\sup_{x,v}h(0,x,v).
\]
For a sufficiently large $M$, we have
\[
\eta^{\lambda}(t,x,v) = h(t,x,v) - \lambda\varphi(t,x,v) = h(t,x,v) -
\lambda(\alpha_{1}(t) + \alpha_{2}(t)M^{2}) \le\sup_{x,v} h(0,x,v)
\]
for $|v|=M$. Thus
\[
\eta^{\lambda}(t,x,v) \le\sup_{x,v}h(0,x,v), \quad\text{for }(t,x,v)\in
[0,T]\times\mathbb{T}^{3} \times B(0;M).
\]
Since $M$ is an arbitrary large enough constant, we can take $M \rightarrow
\infty$. Then we have
\[
\eta^{\lambda}(t,x,v) \le\sup_{x,v}h(0,x,v), \quad\text{for }(t,x,v)\in
[0,T]\times\mathbb{T}^{3} \times\mathbb{R}^{3}.
\]
Taking $\lambda\rightarrow0$, we have
\[
h(t,x,v) \le\sup_{x,v}h(0,x,v).
\]
Thus we complete the proof.
\end{proof}

Now we will derive the maximum principle for weak solutions.
\begin{lemma}
\label{lemma : Maximum principle}
Assume \eqref{Eq : condition for g} and $g\in C^{0}$ and $||h_{0}||_{\infty }<\infty .$ Then the weak
solution to \eqref{Eq : equation for h} satisfies 
\begin{equation}
\label{Eq : maximum principle}
\sup_{t}\|h(t)\|_{\infty }\le \|h_{0}\|_{\infty }
\end{equation}
\end{lemma}
\begin{proof}
	Approximating $g$ by $g^{\delta} \in C^{\infty}$ and $h_{0}$ by $h^{\delta}_{0}\in C^\infty$ as in Lemma \ref{lemma : existence of weak sol h}, we can obtain a smooth solution $h^{\delta}$ to \eqref{Eq : equation for h}.
	Thus by Lemma \ref{lemma : strong maximum principle}, we have
	\begin{equation*}
		\sup_{t} \|h^{\delta}(t)\|_{\infty} \le \|h^{\delta}_{0}\|_{\infty} \le \|h_{0}\|_{\infty}.
	\end{equation*}
	In a similar manner to Lemma \ref{Lemma : energy estimate approximated bar A^theta} we can derive an energy estimate for $h^{\delta} - h^{\delta'}$ and we can show that $h^{\delta}$ is a Cauchy sequence in $L^2$.
	Therefore there exists $h$ such that $\|h^{\delta} - h\|_{L^2} \rightarrow 0$.
	In a similar manner to Lemma \ref{lemma : existence of weak sol h}, we can show that $h$ is a weak solution.
	Since $\sup_{t} \|h^{\delta}(t)\|_{\infty} \le \|h_{0}\|_{\infty}$ and $\|h^{\delta} - h\|_{L^2} \rightarrow 0$, we can obtain \eqref{Eq : maximum principle}.
\end{proof}

% section maximum_principle (end)
% section basic_estimates (end)

\section{$L^{2}$ decay}

\label{sec:l2 decay}

In this section, we will establish a weighted $L^{2}$ estimate for
\eqref{linear landau}. We will adapt techniques in \cite{esposito2013non}, \cite{guo2002landau}, and \cite{strain2006almost}.

As a starting point, we prove that \eqref{linear landau} has a unique weak solution globally in time.
\begin{definition}\label{Def : linear weak sol}
	 Let $f(t,x,v) \in L^{\infty}((0,\infty)\times
\mathbb{T}^{3} \times\mathbb{R}^{3},w^{\vartheta}(v)dtdxdv)$ be a periodic function in
$x\in\mathbb{T}^{3} = [-\pi,\pi]^{3}$ satisfying
\begin{equation}		\label{Eq : sigma norm bdd condition}
	\int_{0}^{t}\|f(s)\|_{\sigma, \vartheta}^2 ds < \infty.
\end{equation}
We say that $f$ is a weak solution of
the Landau equation \eqref{initial}, \eqref{linear landau} on $(0,\infty
)\times\mathbb{T}^{3} \times\mathbb{R}^{3}$ if for all $t\in(0,\infty)$ and
all $\varphi\in C^{1,1,1}_{t,x,v}\left(  (0,\infty)\times\mathbb{T}^{3}
\times\mathbb{R}^{3}\right)  $ such that $\varphi(t,x,v)$ is a periodic
function in $x\in\mathbb{T}^{3} = [-\pi,\pi]^{3}$ and $\varphi(t,x,\cdot)$ is
compactly supported in $\mathbb{R}^{3}$, it satisfies
\begin{equation}
\label{Eq : linear weak sol}%
\begin{split}
&  \iint_{\mathbb{T}^{3} \times\mathbb{R}^{3}} f(t,x,v) \varphi(t,x,v) dxdv -
\iint_{\mathbb{T}^{3} \times\mathbb{R}^{3}} f_{0}(x,v) \varphi(0,x,v) dxdv\\
& = - (f,\varphi)_{\sigma} + \iiint_{(0,t)\times\mathbb{T}^{3} \times\mathbb{R}^{3}}
f(s,x,v)\bigg(\partial_{s} \varphi(s,x,v) + v \cdot\nabla_{x}\varphi(s,x,v)+a_{g}(s,x,v)\cdot \nabla_{v} \varphi(s,x,v)\\
&\quad \quad + K \varphi(s,x,v) + \partial_{i}\sigma^{i}(s,x,v) \varphi(s,x,v) - \partial_{i}\left\{\phi^{ij}*[\mu^{1/2}\partial_{j}g]\right\}(s,x,v)\varphi(s,x,v)\\
&\quad \quad + \left\{\phi^{ij}*[v_{i}\mu^{1/2}\partial_{j}g]\right\}(s,x,v)\varphi(s,x,v) \bigg) - \sigma_{\mu^{1/2}g}(s,x,v) \partial_{i}f(s,x,v) \partial_{j} \varphi(s,x,v) dsdxdv.
\end{split}
\end{equation}
\end{definition}

Let $f(t) = U(t,s)f_0$ be a solution of the following equation
\begin{equation}
f_{t}+v\cdot\partial_{x}f=\bar A_{g} f, \label{bar Af}%
\end{equation}
\begin{equation*}
	f(s) = U(s,s)f_0 = f_0,
\end{equation*}
where $\bar A_{g}$ is defined as in \eqref{Eq : bar A}.
Then by the Duhamel principle, the solution of \eqref{linear landau} is
\begin{equation}
\label{Eq : Duhamel principle}
	f(t) = U(t,0)f_{0} + \int_{0}^{t} U(t, \tau) \bar K_{g} f(\tau) d \tau.
\end{equation}
\begin{lemma}
\label{lemma : existence of linear weak sol}
	Assume \eqref{Eq : condition for g}. Then there exists a unique weak solution f to \eqref{linear landau} in the sense of Definition \ref{Def : linear weak sol} with $f(0) = f_{0}$, which satisfies
	\begin{equation*}
		\sup_{0\le s\le t}\|f(s)\|_\infty \le C(t) \|f_0\|_\infty.
	\end{equation*}
\end{lemma}
\begin{proof}[Sketch of proof]
	It is clear from \eqref{Eq : Duhamel principle} and by the Gronwall inequality.
\end{proof}

For any real-valued function $f(v)$, we define the projection onto the
$\text{span} \{\sqrt\mu, v \sqrt\mu, |v|^{2} \sqrt\mu\}$ in $L^{2}%
(\mathbb{R}^{3})$ as
\begin{equation}
\label{Eq : P_0}\textbf{P} f := \left(  a_{f}(t,x) + v \cdot b_{f} (t,x) +
\frac{(|v|^{2} -3)}{2}c_{f}(t,x)\right)  \sqrt\mu,
\end{equation}
where
\[%
\begin{split}
a_{f}  &  := \frac{\langle f, \sqrt\mu\rangle}{|\langle\sqrt\mu, \sqrt
\mu\rangle|^{2}},\\
b_{f}^{i}  &  := \frac{\langle f, v_{i} \sqrt\mu\rangle}{|\langle v_{i}
\sqrt\mu, v_{i} \sqrt\mu\rangle|^{2}},\\
c_{f}  &  := \frac{\langle f, (|v|^{2} -3) \sqrt\mu/2 \rangle}{|\langle
(|v|^{2} -3)\sqrt\mu/2, (|v|^{2} -3)\sqrt\mu/2 \rangle|^{2}}.%
\end{split}
\]

%section preliminary (end)

We will prove the positivity of $L$. By Lemma
\ref{Lemma : Guo2002 lemma5}, $L$ is only semi-positive;
\[
\left(  Lf, f \right)  \ge C \|(I-\mathbf{P}) f\|_{\sigma}^{2}.%
\]
Now we will estimate $\mathbf{P} f$ in terms of $(I-\mathbf{P})f$. The following
lemma is an adaptation of Lemma 6.1 in \cite{esposito2013non}.

\begin{lemma}
[Lemma 6.1 in \cite{esposito2013non}]\label{Lemma : Espotiso 2013 lemma 6.1}
Assume \eqref{Eq : condition for g}.
Let $f$ be a weak solution of \eqref{initial}, \eqref{conservation laws},
\eqref{linear landau}. Then there exist $C$ and a function $\eta(t) \le C\|f(t)\|_{2}^{2}$, such
that 
\[
\int_{s}^{t} \|\mathbf{P} f(\tau)\|_{\sigma}^{2} \le\eta(t)-\eta(s) + C \int
_{s}^{t} \|(I-\mathbf{P})f(\tau)\|_{\sigma}^{2}.
\]

\end{lemma}

\begin{proof}
For every periodic test fundtion $\psi$, $f$ satisfies,
\begin{multline}
\label{Eq : weak formulation}%
	-\int_{s}^{t}\iint_{\mathbb{T}^{3} \times\mathbb{R}^{3}} v\cdot\nabla_{x}\psi f - \int_{s}^{t}\iint_{\mathbb{T}^{3} \times\mathbb{R}^{3}} \partial_{t}\psi f\\
	= -\iint_{\mathbb{T}^{3} \times\mathbb{R}^{3}} \psi f(t) +\iint_{\mathbb{T}^{3} \times\mathbb{R}^{3}}\psi f(s) + \int_{s}^{t}\iint_{\mathbb{T}^{3} \times\mathbb{R}^{3}}-\psi L(I-\mathbf{P})f + \psi\Gamma(f,f).
\end{multline}
By convention, we denote $a(t,x) = a_{f}(t,x)$, $b(t,x)=b_{f}(t,x)$, and
$c(t,x)= c_{f}(t,x)$, where $a_{f}$, $b_{f}$, and $c_{f}$ are defined as in
\eqref{Eq : P_0}. We note that, with such choices $\eta(t) = -\int
\int_{\mathbb{T}^{3}\times\mathbb{R}^{3}} \psi f(t) dxdv$, and $\psi=
p(v)\phi(t,x)$ for some $|p(v)|\le\exp(-|v|^{2}/4)$ and $\|\phi(t)\|_{2} \le C
(\|a(t)\|_{2} + \|b(t)\|_{2} + \|c(t)\|_{2})$. Thus,
\[%
\begin{split}
|\eta(t)|  &  \le\|f(t)\|_{2} \left(  \int_{\mathbb{T}^{3} \times
\mathbb{R}^{3}} p(v) |\phi(t,x)|^{2} dxdv\right)  ^{1/2}\\
&  = C\|f(t)\|_{\infty}\left(  \int_{\mathbb{T}^{3}} |\phi(t,x)|^{2}
dx\right)  ^{1/2}\\
&  \le C\|f(t)\|_{2} \|\phi(t)\|_{2}\\
&  \le C\|f(t)\|_{2}^{2}.
\end{split}
\]
Without loss of generality, we can take $s=0$.

\textit{Step 1. Estimate of $\nabla_{x} \Delta^{-1} \partial_{t} a =
\nabla_{x} \partial_{t} \phi_{a}$.} Choosing a test function $\psi= \phi\sqrt
\mu$ with $\phi$ dependent only on $x$, we have (Note that
${\displaystyle \int\sqrt\mu Lf = \int\sqrt\mu\Gamma(f,f) = 0}$)
\[
\sqrt{2 \pi}\int_{\mathbb{T}^{3}} [a(t+ \varepsilon) - a(t)]\phi(x) = 2
\pi\sqrt{2 \pi} \int_{t}^{t+ \varepsilon}\int_{\mathbb{T}^{3}}(b\cdot
\nabla_{x})\phi(x).
\]
Therefore,
\[
\int_{\mathbb{T}^{3}} \phi\partial_{t} a = \sqrt{2 \pi} \int_{\mathbb{T}^{3}%
}(b \cdot\nabla_{x})\phi.
\]
First, take $\phi=1$. Then, we have $\int_{\mathbb{T}^{3}} \partial_{t} a(t)
dx =0$ for all $t>0$. On the other hand, for all $\phi(x)\in H^{1}%
(\mathbb{T}^{3})$, we have
\[
\left|  \int_{\mathbb{T}^{3}} \phi(x) \partial_{t} a dx \right|
\lesssim\|b\|_{2} \|\phi\|_{H^{1}}.
\]
Therefore, for all $t>0$, $\|\partial_{t} a(t)\|_{(H^{1})^{*}} \lesssim
\|b(t)\|_{2}$. Since $\int_{\mathbb{T}^{3}} \partial_{t} a dx =0$ for all
$t>0$, we can find a solution of the Poisson equation with the Neumann
boundary condition $- \Delta\Phi_{a} = \partial_{t} a(t)$, $\frac{\partial
\Phi_{a}}{\partial n}=0$ at $\partial\mathbb{T}^{3}$. Let $\phi_{a}$ be a
solution of the Poisson equation with the Neumann boundary condition $-
\Delta\phi_{a} = a(t)$, $\frac{\partial\phi_{a}}{\partial n}=0$ at
$\partial\mathbb{T}^{3}$. Then $\Phi_{a} = \partial_{t} \phi_{a}$. Moreover,
we have
\begin{equation}
\label{Eq : grad. d_t phi_a}\|\nabla_{x} \partial_{t} \phi_{a} \|_{2} =
\|\Phi_{a}\|_{H^{1}} \lesssim\|\partial_{t} a(t)\|_{(H^{1})^{*}}%
\lesssim\|b(t)\|_{2}.
\end{equation}

\textit{Step 2. Estimate of $\nabla_{x} \Delta^{-1} \partial_{t} b^{j} =
\nabla_{x} \partial_{t} \phi_{b}^{j}$.} Choosing a test function $\psi=
\phi(x)v_{i} \sqrt\mu$, we have
\[%
\begin{split}
&  2 \pi\sqrt{2 \pi} \int_{\mathbb{T}^{3}} [b_{i}(t+ \varepsilon
)-b_{i}(t)]\phi\\
&  \quad= 2 \pi\sqrt{2 \pi}\int_{t}^{t+\varepsilon}\int_{\mathbb{T}^{3}}
\partial_{i} \phi[a+c] + \int_{t}^{t+\varepsilon}\iint_{\mathbb{T}^{3}%
\times\mathbb{R}^{3}}\sum_{j=1}^{d} v_{j}v_{i} \sqrt\mu\partial_{j}
\phi(I-\mathbf{P}) f\\
&  \quad\quad+ \int_{t}^{t+ \varepsilon}\iint_{\mathbb{T}^{3} \times
\mathbb{R}^{3}}\phi v_{i} \Gamma(f,f) \sqrt\mu.
\end{split}
\]
Therefore,
\begin{multline*}
	\int_{\mathbb{T}^{3}}\partial_{t} b_{i}(t) \phi = \int_{\mathbb{T}^{3}}\partial_{i} \phi[a(t)+c(t)] + \frac{1}{2 \pi\sqrt{2 \pi}} \Bigg\{\iint_{\mathbb{T}^{3} \times\mathbb{R}^{3}}\sum_{j=1}^{d} v_{i}v_{j} \sqrt{\mu}\partial_{j} \phi(I-\mathbf{P})f(t)\\
	+ \iint_{\mathbb{T}^{3} \times\mathbb{R}^{3}}\phi v_{i} \Gamma(f,f)(t)\sqrt{\mu} \Bigg\} .
\end{multline*}
By the H\"older inequality and Theorem \ref{Thm : modified Guo2002 theorem3}
\[%
\begin{split}
&  \iint_{\mathbb{T}^{3} \times\mathbb{R}^{3}}\sum_{j=1}^{d} v_{i}v_{j}
\sqrt{\mu}\partial_{j} \phi(I-\mathbf{P})f(t)\\
&  \quad\le C\|(1+|v|)^{-1/2}(I-\mathbf{P})f(t)\|_{2}^{2}\left\|
(1+|v|)^{-\frac{-1}{2}}\sum_{j=1}^{d} v_{i}v_{j} \sqrt{\mu}\partial_{j}
\phi\right\|  _{2}\\
&  \quad\le C\|(I-\mathbf{P})\|_{\sigma}\|\phi\|_{2}
\end{split}
\]
and
\[
\left\|  \iint_{\mathbb{T}^{3} \times\mathbb{R}^{3}}\phi v_{i} \Gamma(f,f)(t)
\sqrt{\mu}\right\|  \le C \|f\|_{\infty}\|f\|_{\sigma}\|\phi v_{i}
\mu\|_{\sigma} \le\|f\|_{\infty}\|f\|_{\sigma}\|\phi\|_{2}.
\]
For fixed $t>0$, we choose $\phi= \Phi_{b}^{i}$, where $\Phi_{b}^{i}$ is a
solution of the Poisson equation with the Dirichlet boundary condition $-
\Delta\Phi_{b}^{i} = \partial_{t} b_{i}(t)$, $\Phi_{b}^{i}|_{\partial
\mathbb{T}^{3}} = 0$. Let $\phi_{b}^{i}$ be a solution of the Poisson equation
with the Dirichlet boundary condition $- \Delta\phi_{b}^{i} = b_{i}(t)$,
$\phi_{b}^{i}(t)|_{\partial\mathbb{T}^{3}}=0$. Then $\Phi_{b}^{i} =
\partial_{t} \phi_{b}^{i}$. By the Poincar\`{e} inequality,
\[%
\begin{split}
\int_{\mathbb{T}^{3}} |\nabla_{x} \partial_{t} \phi_{b}^{i}(t)|^{2} dx  &  =
\int_{\mathbb{T}^{3}} |\nabla_{x} \Phi_{b}^{i}|^{2} dx = -\int_{T^{3}}
\Delta_{x} \Phi_{b}^{i} \Phi_{b}^{i} dx\\
&  \le\varepsilon\{\|\nabla_{x} \Phi_{b}^{i}\|_{2}^{2} + \| \Phi_{b}^{i}%
\|_{2}^{2}\}\\
&  \quad+ C_{\varepsilon} \{ \|a(t)\|_{2}^{2} + \|c(t)\|_{2}^{2} +
\|(I-\mathbf{P})f(t)\|_{\sigma}^{2} + \|f(t)\|_{\infty}^{2}\|f(t)\|_{\sigma}^{2}
\}\\
&  \le C\varepsilon\|\nabla_{x} \Phi_{b}^{i}\|_{2}^{2} + C_{\varepsilon} \{
\|a(t)\|_{2}^{2} + \|c(t)\|_{2}^{2} + \|(I-\mathbf{P})f(t)\|_{\sigma}^{2} +
\|f(t)\|_{\infty}^{2}\|f(t)\|_{\sigma}^{2} \},
\end{split}
\]
for every $\varepsilon>0$. Now, we choose small $\varepsilon$, such that $C
\varepsilon\le1/4$. Then we can absorb the first term in RHS to the LHS. Then
we have for all $t>0$,
\begin{equation}
\label{Eq : grad. d_t phi_b^i}\|\nabla_{x} \partial_{t} \phi_{b}^{i}(t)\|_{2}
\le C_{\varepsilon}\{ \|a(t)\|_{2} + \|c(t)\|_{2} + \|(I-\mathbf{P})f(t)\|_{\sigma}
+ \|f(t)\|_{\infty}\|f(t)\|_{\sigma} \}.
\end{equation}

\textit{Step 3. Estimate of $\nabla_{x} \Delta^{-1}\partial_{t} c = \nabla_{x}
\partial_{t} c$.} Choosing a test function $\psi= \phi(x) \left(  \frac
{|v|^{2}-3}{2}\right)  \sqrt\mu$, we have
\begin{align*}
	&  3 \pi\sqrt{2 \pi} \int_{\mathbb{T}^{3}} \phi(x)[c(t+ \varepsilon)-c(t)]\\
	&  \quad= 2 \pi\sqrt{2 \pi} \int_{t}^{t+ \varepsilon}\int_{\mathbb{T}^{3}}b\cdot\nabla_{x} \phi- \int_{t}^{t + \varepsilon}\iint_{\mathbb{T}^{3}\times\mathbb{R}^{3}}(I-\mathbf{P})f \left(  \frac{|v|^{2}-3}{2}\right)  \sqrt{\mu}(v\cdot\nabla_{x})\phi\\
	&  \quad\quad+ \int_{t}^{t+\varepsilon}\iint_{\mathbb{T}^{3} \times\mathbb{R}^{3}} \phi\Gamma(f,f) \left(  \frac{|v|^{2} -3}{2}\right)  \sqrt \mu.
\end{align*}
Therefore,
\[%
\begin{split}
&  \int_{\mathbb{T}^{3}} \phi(x) \partial_{t} c(t)\\
&  \quad= \frac{2}{3}\int_{\mathbb{T}^{3}} b(t)\cdot\nabla_{x} \phi+ \frac
{1}{3 \pi\sqrt{2 \pi}}\iint_{\mathbb{T}^{3} \times\mathbb{R}^{3}}%
(I-\mathbf{P})f(t)\left(  \frac{|v|^{2} -3}{2}\right)  \sqrt{\mu}(v \cdot\nabla
_{x})\phi\\
&  \quad\quad+ \frac{1}{3 \pi\sqrt{2 \pi}}\iint_{\mathbb{T}^{3} \times
\mathbb{R}^{3}}\phi\Gamma(f,f)(t) \left(  \frac{|v|^{2} -3}{2}\right)
\sqrt{\mu}.
\end{split}
\]
Similarly to step 2,
\[
\iint_{\mathbb{T}^{3} \times\mathbb{R}^{3}} (I-\mathbf{P})f(t)\left(  \frac{|v|^{2}
-3}{2}\right)  \sqrt{\mu}(v \cdot\nabla_{x})\phi\le C\|(I-\mathbf{P})\|_{\sigma
}\|\nabla_{x} \phi\|_{2}
\]
and
\[
\left\|  \iint_{\mathbb{T}^{3} \times\mathbb{R}^{3}}\phi\Gamma(f,f)(t) \left(
\frac{|v|^{2} -3}{2}\right)  \sqrt{\mu}\right\|  \le\|f\|_{\infty
}\|f\|_{\sigma}\|\phi\|_{2}.
\]
For fixed $t>0$, we choose $\phi= \Phi_{c}$, where $\Phi_{c}$ is a solution of
the Poisson equation with the Dirichlet boundary condition $- \Delta\Phi_{c} =
\partial_{t} c(t)$, $\Phi_{c}|_{\partial\mathbb{T}^{3}} = 0$. Let $\phi_{c}$
be a solution of the Poisson equation with the Dirichlet boundary condition $-
\Delta\phi_{c} = c(t)$, $\phi_{c}(t)|_{\partial\mathbb{T}^{3}}=0$. Then
$\Phi_{c} = \partial_{t} \phi_{c}$. By the Poincar\`{e} inequality,
\[%
\begin{split}
\int_{\mathbb{T}^{3}} |\nabla_{x} \partial_{t} \phi_{c}(t)|^{2} dx  &  =
\int_{\mathbb{T}^{3}} |\nabla_{x} \Phi_{c}|^{2} dx = -\int_{\mathbb{T}^{3}}
\Delta_{x} \Phi_{c} \Phi_{c} dx\\
&  \le\varepsilon\{ \|\nabla_{x} \Phi_{c}\|_{2}^{2} + \| \Phi_{c}\|_{2}%
^{2}\}\\
&  \quad+ C_{\varepsilon}\{ \|b(t)\|_{2}^{2} + \|(I-\mathbf{P})f(t)\|_{\sigma}^{2}
+ \|f(t)\|_{\infty}^{2}\|f(t)\|_{\sigma}^{2} \}\\
&  \le C\varepsilon\|\nabla_{x} \Phi_{c}\|_{2}^{2} + C_{\varepsilon}\{
\|b(t)\|_{2}^{2} + \|(I-\mathbf{P})f(t)\|_{\sigma}^{2} + \|f(t)\|_{\infty}%
^{2}\|f(t)\|_{\sigma}^{2} \}.
\end{split}
\]
Therefore, for all $t>0$,
\begin{equation}
\label{Eq : grad. d_t phi_c}\|\nabla_{x} \partial_{t} \phi_{c}(t)\|_{2} \le
C_{\varepsilon}\{ \|b(t)\|_{2} + \|(I-\mathbf{P})f(t)\|_{\sigma} + \|f(t)\|_{\infty
}\|f(t)\|_{\sigma}\}.
\end{equation}

\textit{Step 4. Estimate of $c$.} Choosing a test function $\psi= (|v|^{2} - 5)
\sqrt{\mu} v\cdot\nabla_{x} \phi_{c}$, we have
\[%
\begin{split}
-10 \pi\sqrt{2 \pi}\int_{0}^{t} \int_{\mathbb{T}^{3}} \Delta_{x} \phi_{c} c
&  = -\iint_{\mathbb{T}^{3} \times\mathbb{R}^{3}} \psi f(t) + \iint
_{\mathbb{T}^{3} \times\mathbb{R}^{3}} \psi f(0)\\
&  \quad+ \sum_{i=1}^{d}\int_{0}^{t} \iint_{\mathbb{T}^{3} \times
\mathbb{R}^{3}} (|v|^{2} - 5) v_{i} \sqrt{\mu} \partial_{t} \partial_{i}
\phi_{c} f\\
&  \quad- \int_{0}^{t} \iint_{\mathbb{T}^{3} \times\mathbb{R}^{3}} \psi
L(I-\mathbf{P})f + \int_{0}^{t} \iint_{\mathbb{T}^{3} \times\mathbb{R}^{3}}
\psi\Gamma(f,f).\\
&  = -\iint_{\mathbb{T}^{3} \times\mathbb{R}^{3}} \psi f(t) + \iint
_{\mathbb{T}^{3} \times\mathbb{R}^{3}} \psi f(0)\\
&  \quad+ \sum_{i,j=1}^{d}\int_{0}^{t} \iint_{\mathbb{T}^{3} \times
\mathbb{R}^{3}} (|v|^{2} - 5) v_{i} v_{j} \mu\partial_{t} \partial_{i}
\phi_{c} b_{j}\\
&  \quad+ \sum_{i=1}^{d}\int_{0}^{t} \iint_{\mathbb{T}^{3} \times
\mathbb{R}^{3}} (|v|^{2} - 5) v_{i} \sqrt{\mu} \partial_{t} \partial_{i}
\phi_{c} (I-\mathbf{P}) f\\
&  \quad- \int_{0}^{t} \iint_{\mathbb{T}^{3} \times\mathbb{R}^{3}} \psi
L(I-\mathbf{P})f + \int_{0}^{t} \iint_{\mathbb{T}^{3} \times\mathbb{R}^{3}}
\psi\Gamma(f,f).
\end{split}
\]
Note that ${\displaystyle \int(|v|^{2}-5)v_{i}v_{j} \mu=0}$. Therefore, the
third term of RHS is zero. Moreover,
\[%
\begin{split}
&  \sum_{i=1}^{d} \iint_{\mathbb{T}^{3} \times\mathbb{R}^{3}} (|v|^{2} - 5)
v_{i} \sqrt{\mu} \partial_{t} \partial_{i} \phi_{c} (I-\mathbf{P}) f\\
&  \quad\le C\|\nabla_{x} \partial_{t} \phi_{c}\|_{2} \|(I-\mathbf{P})f\|_{\sigma
}\\
&  \quad\le C(C_{\varepsilon}\{ \|b\|_{2} + \|(I-\mathbf{P})f\|_{\sigma} +
\|f\|_{\infty}\|f\|_{\sigma}\})\|(I-\mathbf{P})f\|_{\sigma}\\
&  \quad\le\varepsilon\|b\|_{2}^{2} + C_{\varepsilon}\{\|(I-\mathbf{P})f\|_{\sigma
}^{2} + \|f\|_{\infty}^{2}\|f\|_{\sigma}^{2} \},
\end{split}
\]
\[%
\begin{split}
\iint_{\mathbb{T}^{3} \times\mathbb{R}^{3}} \psi L(I-\mathbf{P})f  &  =
\iint_{\mathbb{T}^{3} \times\mathbb{R}^{3}} L \psi(I-\mathbf{P})f\\
&  \le C\|\nabla_{x} \phi_{c}\|_{2}\|(I-\mathbf{P})f\|_{\sigma}\\
&  \le C\|c\|_{2} \|(I-\mathbf{P})f\|_{\sigma}\\
&  \le\varepsilon\|c\|_{2}^{2} + C_{\varepsilon}\|(I-\mathbf{P})\|_{\sigma}^{2}
\end{split}
\]
and
\[
\iint_{\mathbb{T}^{3} \times\mathbb{R}^{3}} \psi\Gamma(f,f) \le C\|f\|_{\infty
}\|f\|_{\sigma}\|c\|_{2}\le\varepsilon\|c\|_{2}^{2} + C_{ \varepsilon
}\|f\|_{\infty}^{2}\|f\|_{\sigma}^{2}.
\]
For a small $\varepsilon>0$, we can absorb $\|c\|_{2}^{2}$ on the RHS to the LHS.
By \eqref{Eq : grad. d_t phi_c}, we have
\begin{equation}
\label{Eq : estimate c}%
	\int_{0}^{t} \|c(s)\|^{2} ds\le C( \eta(t)-\eta(0)) + \int_{0}^{t} C_{\varepsilon}\left\{\|(I-\mathbf{P})f\|_{\sigma}^{2} + \|f\|_{\infty}^{2}\|f\|_{\sigma}^{2}\right\}  +\varepsilon\|b\|_{2}^{2} ds.
\end{equation}

\textit{Step 5. Estimate of $b$.} We will estimate $(\partial_{ij} \phi^{j}%
_{b})b_{i}$ for all $i,j = 1,...,d,$ and $(\partial_{jj}\phi^{i}_{b})b_{i}$
for $i \neq j$.

We first estimate $(\partial_{ij} \phi^{j}_{b})b_{i}$. Choosing a test function
$\psi= [(v_{i})^{2} - 1]\sqrt\mu\partial_{j} \phi_{b}^{j}$, we have
\begin{equation}
\label{Eq : test b1}%
	\begin{split}
		&  -\sum_{l}\int_{0}^{t} \iint_{\mathbb{T}^{3} \times\mathbb{R}^{3}} v_{*}[(v_{i})^{2} -1]\sqrt\mu\partial_{lj} \phi^{j}_{b} f - \int_{0}^{t}\iint_{\mathbb{T}^{3} \times\mathbb{R}^{3}}(v_{i}^{2} -1)\sqrt\mu\partial_{t}\partial_{j} \phi^{j}_{b} f\\
		& \quad = -\iint_{\mathbb{T}^{3} \times\mathbb{R}^{3}} \psi f(t) + \iint_{\mathbb{T}^{3} \times\mathbb{R}^{3}} \psi f(0)\\
		& \quad \quad- \int_{0}^{t} \iint_{\mathbb{T}^{3} \times\mathbb{R}^{3}} \psi L(I-\mathbf{P})f + \int_{0}^{t} \iint_{\mathbb{T}^{3} \times\mathbb{R}^{3}}\psi\Gamma(f,f).
	\end{split}
\end{equation}
Note that for $i\neq k$
\[
\int[(v_{i})^{2}-1]\mu= \int[(v_{i})^{2} - 1](v_{k})^{2} \mu= 0,
\]
and
\[
\int[(v_{i})^{2} - 1](v_{i})^{2} \mu= 2 \sqrt{2 \pi}.
\]
Therefore,
\begin{equation}
\label{Eq : test b1-1}%
	\begin{split}
		&  \sum_{l}\int_{0}^{t} \iint_{\mathbb{T}^{3} \times\mathbb{R}^{3}} v_{*}[(v_{i})^{2} -1]\sqrt\mu\partial_{lj} \phi^{j}_{b} f\\
		& \quad = \sum_{l}\int_{0}^{t} \iint_{\mathbb{T}^{3} \times\mathbb{R}^{3}}(v_{*})^{2} [(v_{i})^{2} -1] \mu\partial_{lj} \phi^{j}_{b} b_{l}\\
		& \quad \quad+ \sum_{l}\int_{0}^{t} \iint_{\mathbb{T}^{3} \times\mathbb{R}^{3}}v_{*} [(v_{i})^{2} -1]\sqrt\mu\partial_{lj} \phi^{j}_{b} (I-\mathbf{P})f\\
		& \quad = 2 \sqrt{2 \pi} \int_{\mathbb{T}^{3}} \partial_{ij} \phi^{j}_{b} b_{i} +\sum_{l}\int_{0}^{t} \iint_{\mathbb{T}^{3} \times\mathbb{R}^{3}} v_{*}[(v_{i})^{2} -1]\sqrt\mu\partial_{lj} \phi^{j}_{b} (I-\mathbf{P})f,
	\end{split}
\end{equation}
and
\begin{equation}
\label{Eq : test b1-2}%
\begin{split}
\left|  \sum_{l}\int_{0}^{t} \iint_{\mathbb{T}^{3} \times\mathbb{R}^{3}} v_{*}
[(v_{i})^{2} -1]\sqrt\mu\partial_{lj} \phi^{j}_{b} (I-\mathbf{P})f\right|   &  \le
C \int_{0}^{t} \|b\|_{2}\|(I-\mathbf{P})f\|_{\sigma}\\
&  \le\varepsilon\|b\|_{2}^{2} + C_{\varepsilon}\|(I-\mathbf{P})f\|_{\sigma}^{2}.
\end{split}
\end{equation}
Moreover
\[%
\begin{split}
\int_{0}^{t} \iint_{\mathbb{T}^{3} \times\mathbb{R}^{3}}(v_{i}^{2} -1)\sqrt
\mu\partial_{t} \partial_{j} \phi^{j}_{b} f  &  = \int_{0}^{t} \iint
_{\mathbb{T}^{3} \times\mathbb{R}^{3}}(v_{i}^{2} -1) \mu\partial_{t}
\partial_{j} \phi^{j}_{b} \frac{|v|^{2}-3}{2} c\\
&  \quad+ \int_{0}^{t} \iint_{\mathbb{T}^{3} \times\mathbb{R}^{3}}(v_{i}^{2}
-1)\sqrt\mu\partial_{t} \partial_{j} \phi^{j}_{b} (I-\mathbf{P})f.
\end{split}
\]
By \eqref{Eq : grad. d_t phi_b^i},
\begin{equation}
\label{Eq : test b1-3}%
\begin{split}
\left|  \int_{0}^{t} \iint_{\mathbb{T}^{3} \times\mathbb{R}^{3}}(v_{i}^{2}
-1)\sqrt\mu\partial_{t} \partial_{j} \phi^{j}_{b} f\right|   &  \le\int
_{0}^{t} C_{\varepsilon} \left\{  \|a\|_{2} + \|c\|_{2} + \|(I-\mathbf{P}%
)f\|_{\sigma} + \|f\|_{\sigma} \|f\|_{\sigma} \right\} \\
&  \quad\quad\quad\times\left\{  \|c\|_{2} + C_{\vartheta}\|(I-\mathbf{P}%
)f\|_{\sigma}\right\} \\
&  \le\int_{0}^{t} C_{\varepsilon} \left\{  \|(I-\mathbf{P})f\|_{\sigma}^{2} +
\|f\|_{\infty}^{2}\|f\|_{\sigma}^{2} + \|c\|_{2}^{2} \right\}  +
\varepsilon\|a\|_{2}^{2}.
\end{split}
\end{equation}
In a similar way to step 4,
\begin{equation}
\label{Eq : test b1-4}\iint_{\mathbb{T}^{3} \times\mathbb{R}^{3}} \psi
L(I-\mathbf{P})f \le\varepsilon\|b\|_{2}^{2} + C_{\varepsilon}\|(I-\mathbf{P}%
)\|_{\sigma}^{2}.
\end{equation}
and
\begin{equation}
\label{Eq : test b1-5}\iint_{\mathbb{T}^{3} \times\mathbb{R}^{3}} \psi
\Gamma(f,f) \le\varepsilon\|b\|_{2}^{2} + C_{\varepsilon}\|f\|_{\infty}%
^{2}\|f\|_{\sigma}^{2}.
\end{equation}
Combining \eqref{Eq : test b1} - \eqref{Eq : test b1-5}, 
\begin{equation}
\label{Eq : pa_ij phi^j_b b_i}%
\begin{split}
&  \int\partial_{ij} \phi^{j}_{b} b_{i}\\
&  \quad\le C(\eta(t)-\eta(0)) + \int_{0}^{t} C_{\varepsilon}\left\{
\|(I-\mathbf{P})f\|_{\sigma}^{2} + \|f\|_{\infty}^{2}\|f\|_{\sigma}^{2} +
\|c\|_{2}^{2}\right\}  + \varepsilon\{ \|a\|_{2}^{2} + \|b\|_{2}^{2}\}.
\end{split}
\end{equation}

Now we estimate $(\partial_{jj}\phi^{i}_{b})b_{i}$. Choose test function
$\psi= |v|^{2} v_{i} v_{j} \sqrt\mu\partial_{j} \phi^{i}_{b}$ for $i\neq j$.
Then
\begin{equation}
\label{Eq : test b2}%
	\begin{split}
		&  -\sum_{l}\int_{0}^{t} \iint_{\mathbb{T}^{3} \times\mathbb{R}^{3}} v_{*}|v|^{2} v_{i} v_{j} \sqrt\mu\partial_{lj}\phi^{i}_{b} f - \int_{0}^{t}\iint_{\mathbb{T}^{3} \times\mathbb{R}^{3}}|v|^{2} v_{i} v_{j} \sqrt\mu\partial_{t} \partial_{j} \phi^{i}_{b} f\\
		& \quad = -\iint_{\mathbb{T}^{3} \times\mathbb{R}^{3}} \psi f(t) + \iint_{\mathbb{T}^{3} \times\mathbb{R}^{3}} \psi f(0)- \int_{0}^{t} \iint_{\mathbb{T}^{3} \times\mathbb{R}^{3}} \psi L(I-\mathbf{P})f + \int_{0}^{t} \iint_{\mathbb{T}^{3} \times\mathbb{R}^{3}}\psi\Gamma(f,f).
	\end{split}
\end{equation}
Note that
\begin{equation}
\label{Eq : test b2-1}%
	\begin{split}
		&  \sum_{l}\int_{0}^{t} \iint_{\mathbb{T}^{3} \times\mathbb{R}^{3}} v_{*}|v|^{2} v_{i} v_{j} \sqrt\mu\partial_{lj}\phi^{i}_{b} f\\
		& \quad = \int_{0}^{t} \iint_{\mathbb{T}^{3} \times\mathbb{R}^{3}} |v|^{2}(v_{i})^{2} (v_{j})^{2} \sqrt\mu[\partial_{ij}\phi^{i}_{b} b_{j} +\partial_{jj} \phi^{i}_{b} b_{i}]\\
		& \quad \quad+ \sum_{l}\int_{0}^{t} \iint_{\mathbb{T}^{3} \times\mathbb{R}^{3}}v_{*} |v|^{2} v_{i} v_{j} \sqrt\mu\partial_{lj}\phi^{i}_{b} (I-\mathbf{P})f.
	\end{split}
\end{equation}
From \eqref{Eq : pa_ij phi^j_b b_i},
\begin{equation}
\label{Eq : test b2-2}%
\begin{split}
&  \left|  \int_{0}^{t} \iint_{\mathbb{T}^{3} \times\mathbb{R}^{3}} |v|^{2}
(v_{i})^{2} (v_{j})^{2} \sqrt\mu\partial_{ij}\phi^{i}_{b} b_{j}\right| \\
&  \quad\le C(\eta(t)-\eta(0)) + \int_{0}^{t} C_{\varepsilon}\left\{
\|(I-\mathbf{P})f\|_{\sigma, \vartheta}^{2} + \|f\|_{\infty}^{2}\|f\|_{\sigma,
\vartheta}^{2} + \|c\|_{2}^{2} \right\}  + \varepsilon\{ \|a\|_{2}^{2} +
\|b\|_{2}^{2}\},
\end{split}
\end{equation}
and
\begin{equation}
\label{Eq : test b2-3}%
\begin{split}
\left|  \sum_{l}\int_{0}^{t} \iint_{\mathbb{T}^{3} \times\mathbb{R}^{3}} v_{*}
|v|^{2} v_{i} v_{j} \sqrt\mu\partial_{lj}\phi^{i}_{b} (I-\mathbf{P})f\right|   &
\le\int_{0}^{t} C\|b\|_{2} \|(I-\mathbf{P})f\|_{\sigma}\\
&  \le\int_{0}^{t} \varepsilon\|b\|_{2}^{2} + C_{\varepsilon}\|(I-\mathbf{P}%
)f\|_{\sigma}^{2}.
\end{split}
\end{equation}
Moreover, by \eqref{Eq : grad. d_t phi_b^i}
\begin{equation}
\label{Eq : test b2-4}%
\begin{split}
\left|  \int_{0}^{t} \iint_{\mathbb{T}^{3} \times\mathbb{R}^{3}}|v|^{2} v_{i}
v_{j} \sqrt\mu\partial_{t} \partial_{j} \phi^{i}_{b} f\right|   &  = \left|
\int_{0}^{t} \iint_{\mathbb{T}^{3} \times\mathbb{R}^{3}}|v|^{2} v_{i} v_{j}
\sqrt\mu\partial_{t} \partial_{j} \phi^{i}_{b} (I-\mathbf{P})f\right| \\
&  \le\int_{0}^{t} C_{\varepsilon}\left\{  \|a\|_{2} + \|c\|_{2} +
\|(I-\mathbf{P})f\|_{\sigma} + \|f\|_{\infty}\|f\|_{\sigma}\right\}  \|(I-\mathbf{P}%
)f\|_{\sigma}\\
&  \le\int_{0}^{t} C_{\varepsilon}\left\{  \|(I-\mathbf{P})f\|_{\sigma}^{2} +
\|f\|_{\infty}^{2}\|f\|_{\sigma}^{2}\right\}  + \varepsilon\left\{
\|a\|_{2}^{2} + \|c\|_{2}^{2} \right\}  .
\end{split}
\end{equation}
Similarly to \eqref{Eq : test b1-4} and \eqref{Eq : test b1-5},
\begin{equation}
\label{Eq : test b2-5}\iint_{\mathbb{T}^{3} \times\mathbb{R}^{3}} \psi
L(I-\mathbf{P})f \le\varepsilon\|b\|_{2}^{2} + C_{\varepsilon}\|(I-\mathbf{P}%
)\|_{\sigma}^{2}.
\end{equation}
and
\begin{equation}
\label{Eq : test b2-6}\iint_{\mathbb{T}^{3} \times\mathbb{R}^{3}} \psi
\Gamma(f,f) \le\varepsilon\|b\|_{2}^{2} + C_{\varepsilon}\|f\|_{\infty}%
^{2}\|f\|_{\sigma}^{2}.
\end{equation}
Combining \eqref{Eq : test b2} - \eqref{Eq : test b2-6} yields
\begin{equation}
\label{Eq : pa_jj phi^i_b b_i}%
\begin{split}
&  \int\partial_{jj} \phi^{i}_{b} b_{i}\\
&  \quad\le C(\eta(t)-\eta(0)) + \int_{0}^{t} C_{\varepsilon} \left\{
\|(I-\mathbf{P})f\|_{\sigma}^{2} + \|f\|_{\infty}^{2}\|f\|_{\sigma}^{2} +
\|c\|_{2}^{2} \right\}  + \varepsilon\{ \|a\|_{2}^{2} + \|b\|_{2}^{2}\}.
\end{split}
\end{equation}
From \eqref{Eq : pa_ij phi^j_b b_i} and \eqref{Eq : pa_jj phi^i_b b_i} for
small $\varepsilon$, we can absorb $\|b\|_{2}^{2}$ term on RHS to the LHS.
Then we can conclude that
\begin{equation}
\label{Eq : estimate b}\int_{0}^{t} \|b(s)\|^{2} ds \le C(\eta(t)-\eta(0)) +
\int_{0}^{t} C_{\varepsilon}\left\{  \|(I-\mathbf{P})f\|_{\sigma}^{2} +
\|f\|_{\infty}^{2}\|f\|_{\sigma}^{2} + \|c\|_{2}^{2} \right\}  +
\varepsilon\|a\|_{2}^{2} ds.
\end{equation}

\textit{Step 6. Estimate of $a$.} Choosing a test function
\[
\psi= (|v|^{2} - 10) v\cdot\nabla_{x} \phi_{a} \sqrt\mu,
\]
we have
\begin{equation}
\label{Eq : test a}%
\begin{split}
&  -\int_{0}^{t} \iint_{\mathbb{T}^{3}\times\mathbb{R}^{3}} (|v|^{2}%
-10)v_{i}v_{j} \partial_{ij} \phi_{a} \sqrt\mu f -\int_{0}^{t} \iint
_{\mathbb{T}^{3}\times\mathbb{R}^{3}} (|v|^{2}-10)v_{i} \partial_{t}%
\partial_{i} \phi_{a} \sqrt\mu f\\
&  = -\iint_{\mathbb{T}^{3} \times\mathbb{R}^{3}} \psi f(t) + \iint
_{\mathbb{T}^{3} \times\mathbb{R}^{3}} \psi f(0)- \int_{0}^{t} \iint_{\mathbb{T}^{3} \times\mathbb{R}^{3}} \psi
L(I-\mathbf{P})f + \int_{0}^{t} \iint_{\mathbb{T}^{3} \times\mathbb{R}^{3}}
\psi\Gamma(f,f).
\end{split}
\end{equation}
Note that
\[
\int(|v|^{2} -10)\frac{|v|^{2} - 3}{2}(v_{i})^{2} \mu= 0.
\]
Therefore,
\begin{equation}
\label{Eq : test a-1}%
\begin{split}
&  \int_{0}^{t} \iint_{\mathbb{T}^{3}\times\mathbb{R}^{3}} (|v|^{2}%
-10)v_{i}v_{j} \partial_{ij} \phi_{a} \sqrt\mu f\\
&  =\int_{0}^{t} \iint_{\mathbb{T}^{3}\times\mathbb{R}^{3}} (|v|^{2}%
-10)(v_{i})^{2} \mu\partial_{ii} \phi_{a} a+ \int_{0}^{t} \iint_{\mathbb{T}^{3}\times\mathbb{R}^{3}}
(|v|^{2}-10)v_{i}v_{j} \partial_{ij} \phi_{a} \sqrt\mu(I-\mathbf{P})f
\end{split}
\end{equation}
and
\begin{equation}
\label{Eq : test a-2}%
\begin{split}
\left|  \int_{0}^{t} \iint_{\mathbb{T}^{3}\times\mathbb{R}^{3}} (|v|^{2}%
-10)v_{i}v_{j} \partial_{ij} \phi_{a} \sqrt\mu(I-\mathbf{P})f\right|   &  \le
\int_{0}^{t} C\|a\|_{2} \|(I-\mathbf{P})f\|_{\sigma}\\
&  \le\int_{0}^{t} \varepsilon\|a\|_{2}^{2} + C_{\varepsilon}\|(I-\mathbf{P}%
)f\|_{\sigma}^{2}.
\end{split}
\end{equation}
Moreover, by \eqref{Eq : grad. d_t phi_a}
\begin{equation}
\label{Eq : test a-3}%
\begin{split}
\left|  \int_{0}^{t} \iint_{\mathbb{T}^{3}\times\mathbb{R}^{3}} (|v|^{2}%
-10)v_{i} \partial_{t}\partial_{i} \phi_{a} \sqrt\mu f\right|   &  \le\left|
\int_{0}^{t} \iint_{\mathbb{T}^{3}\times\mathbb{R}^{3}} (|v|^{2}
-10)(v_{i})^{2} \mu\partial_{t} \partial_{i} \phi_{a} b_{i}\right| \\
&  \quad+ \left|  \int_{0}^{t} \iint_{\mathbb{T}^{3}\times\mathbb{R}^{3}}
(|v|^{2}-10)v_{i} \partial_{t}\partial_{i} \phi_{a} \sqrt\mu(I-\mathbf{P}) f\right|
\\
&  \le\int_{0}^{t} C\|b\|_{2}\left\{  \|b\|_{2} + C\|(I-\mathbf{P})f\|_{\sigma
}\right\} \\
&  \le\int_{0}^{t} C\{ \|b\|_{2}^{2} + \|(I-\mathbf{P})f\|_{\sigma}^{2}\}.
\end{split}
\end{equation}
Similarly to step 4 and 5, we have
\begin{equation}
\label{Eq : test a-4}\iint_{\mathbb{T}^{3} \times\mathbb{R}^{3}} \psi
L(I-\mathbf{P})f \le\varepsilon\|a\|_{2}^{2} + C_{\varepsilon}\|(I-\mathbf{P}%
)\|_{\sigma}^{2}
\end{equation}
and
\begin{equation}
\label{Eq : test a-5}\iint_{\mathbb{T}^{3} \times\mathbb{R}^{3}} \psi
\Gamma(f,f) \le\varepsilon\|a\|_{2}^{2} + C_{\varepsilon}\|f\|_{\infty}%
^{2}\|f\|_{\sigma}^{2}.
\end{equation}
Similarly, from \eqref{Eq : test a} - \eqref{Eq : test a-5} for a small
$\varepsilon$, we can absorb $\|a\|_{2}^{2}$ on the RHS to the LHS. Then we
have
\begin{equation}
\label{Eq : estimate a}\int_{0}^{t} \|a(s)\|^{2} ds \le C(\eta(t)-\eta(0)) +
\int_{0}^{t} C_{\varepsilon} \left\{  \|(I-\mathbf{P})f(s)\|_{\sigma}^{2} +
\|f(s)\|_{\infty}^{2}\|f(s)\|_{\sigma}^{2} + \|b(s)\|_{2}^{2}\right\}  ds.
\end{equation}

Combining \eqref{Eq : estimate c}, \eqref{Eq : estimate b}, and
\eqref{Eq : estimate a}, we have
\[%
\begin{split}
\int_{0}^{t} \|\mathbf{P} f\|_{\sigma}^{2}  &  \le C(\eta(t)-\eta(0) )+ \int
_{0}^{t} C_{\varepsilon}\left\{  \|(I-\mathbf{P})f(s)\|_{\sigma}^{2} +
\|f(s)\|_{\infty}^{2} \|f(s)\|_{\sigma}^{2} \right\}  ds\\
&  \quad+ \int_{0}^{t} \varepsilon\|\mathbf{P} f(s)\|_{\sigma}^{2}ds\\
&  \le C(\eta(t)-\eta(0)) + \int_{0}^{t} C_{\varepsilon}\left\{
\|(I-\mathbf{P})f(s)\|_{\sigma}^{2} + \|f(s)\|_{\infty}^{2} \|(I-\mathbf{P}%
)f(s)\|_{\sigma}^{2}\right\}  ds\\
&  \quad+ \int_{0}^{t} (C_{\varepsilon}\|f(s)\|_{\infty}^{2} + \varepsilon
)\|\mathbf{P} f(s)\|_{\sigma}^{2} ds.
\end{split}
\]
Note that $C_{\varepsilon} = C \varepsilon^{-1}$. Choosing $\varepsilon_{0} =
\varepsilon$, we have $C_{\varepsilon} \|f(s)\|_{\infty}^{2} \le C_{\varepsilon}
\varepsilon^{2} = C \varepsilon\le1/4$ so that $\|\mathbf{P}f\|_{\sigma}^{2}$ term
on the RHS can be absorbed to the LHS. Thus we complete the proof.
\end{proof}

\begin{corollary}
\label{Coro : coercivity}
Assume \eqref{Eq : condition for g}.
Let $f(t,x,v)$ be a weak solution of \eqref{Landau}
- \eqref{conservation laws} in the sense of Definition \ref{Def : linear weak sol}. Then there exist a constant $0< \delta^{\prime
}\le1/4$ and a function $0\le\eta(t) \le C\|f(t)\|_{2}^{2}$, such that
\begin{equation}
\label{Eq : coercivity}\int_{s}^{t} (L[f(\tau)],f(\tau))d \tau\ge
\delta^{\prime}\left(  \int_{s}^{t} \|f(\tau)\|_{\sigma}^{2} d \tau-
\{\eta(t)-\eta(s)\} \right)  .
\end{equation}

\end{corollary}

\begin{proof}
By Lemma \ref{Lemma : Guo2002 lemma5} and Lemma
\ref{Lemma : Espotiso 2013 lemma 6.1},
\[%
\begin{split}
\int_{s}^{t} (L[f(\tau)],f(\tau))d \tau &  \ge\delta\int_{s}^{t}
\|(I-\mathbf{P})f(\tau)\|_{\sigma}^{2} d \tau\\
&  \ge\delta\frac{C}{1+C} \int_{s}^{t} \|(I-\mathbf{P})f(\tau)\|_{\sigma}^{2} d
\tau+ \delta\frac{1}{1+C} \int_{s}^{t} \|(I-\mathbf{P})f(\tau)\|_{\sigma}^{2} d
\tau\\
&  \ge\delta\frac{C}{1+C} \int_{s}^{t} \|(I-\mathbf{P})f(\tau)\|_{\sigma}^{2} d
\tau+ \delta\frac{1}{1+C} C \left(  \int_{s}^{t} \|\mathbf{P} f(\tau)\|_{\sigma
}^{2} d \tau- \{\eta(t) - \eta(s)\}\right) \\
&  = \frac{C \delta}{1+C} \left(  \int_{0}^{t} \|f(\tau)\|_{\sigma}^{2} d
\tau- \{\eta(t) - \eta(s)\} \right)  .
\end{split}
\]

\end{proof}

\begin{remark}
\label{Rmk : choice of delta'} Note that in Lemma \ref{Lemma : Guo2002 lemma5}%
, we can take $\delta>0$ sufficiently small. Therefore we can also take $\delta
^{\prime}$ small enough.
\end{remark}

Now we will prove Theorem \ref{Thm : energy estimate}.
The proof is a modification of Theorem 5.1 in \cite{strain2006almost}.

\begin{proof}[Proof of Theorem \ref{Thm : energy estimate}]
We will prove
\begin{equation}
\label{Eq : energy induction}%
\begin{split}
&  \sum_{0\le\bar\vartheta\le 2\vartheta}\left(  \frac{C_{\bar\vartheta}^{*}}%
{2}\{ \|f(t)\|_{2, \bar\vartheta/2}^{2} - \|f(s)\|_{2, \bar\vartheta/2}^{2} \} +
\delta_{\bar\vartheta,2\vartheta}\int_{s}^{t} \|f(\tau)\|_{\sigma,
\bar\vartheta/2}^{2} d \tau\right)  - \delta^{\prime}\{\eta(t) - \eta(s)\}\\
&  \quad\le C_{\vartheta} \int_{s}^{t} \|g(\tau)\|_{\infty} \|f(\tau
)\|_{\sigma,\vartheta}^{2} d \tau
\end{split}
\end{equation}
by the induction on $\vartheta$.

\textit{Basis step($\vartheta=0$).} Multiplying \eqref{linear landau} by $f$,
integrating both sides of the resulting equation, by Theorem
\ref{Thm : modified Guo2002 theorem3} and Corollary \ref{Coro : coercivity}, we
have
\[
\frac{1}{2}\{ \|f(t)\|_{2}^{2} - \|f(s)\|_{2}^{2}\} + \delta^{\prime}\left(
\int_{s}^{t} \|f(\tau)\|_{\sigma}^{2} d \tau- \{\eta(t) - \eta(s)\}\right)
\le C\int_{s}^{t} \|g(\tau)\|_{\infty} \|f(\tau)\|_{\sigma}^{2} d \tau.
\]

\textit{Inductive step.} Suppose that \eqref{Eq : energy induction} holds for
$\vartheta-1/2$. Multiplying \eqref{linear landau} by $w^{2\vartheta}f$,
integrating both sides of the resulting equation, by Lemma
\ref{Lemma : Guo2002 lemma6} and Theorem
\ref{Thm : modified Guo2002 theorem3}, we have
\begin{equation}
\label{Eq : energy induction 2}\frac{1}{2}\{ \|f(t)\|_{2,\vartheta}^{2} -
\|f(s)\|_{2,\vartheta}^{2} \} + \int_{s}^{t}\left(  \frac{1}{2} \|f(\tau
)\|_{\sigma,\vartheta}^{2} - C_{\vartheta} \|f(\tau)\|_{\sigma}^{2} d
\tau\right)  \le C_{\vartheta}\int_{s}^{t} \|g(\tau)\|_{\infty}\|f(\tau
)\|_{\sigma,\vartheta}^{2}d \tau.
\end{equation}
Multiply \eqref{Eq : energy induction 2} by $\frac{\delta_{0,2\vartheta-1}%
}{2C_{\vartheta}}$ and add it to \eqref{Eq : energy induction}. Then we have
\[%
\begin{split}
&  \sum_{0\le\bar\vartheta\le 2\vartheta-1}\left(  \frac{C_{\bar\vartheta}^{*}%
}{2}\{ \|f(t)\|_{2, \bar\vartheta/2}^{2} - \|f(s)\|_{2, \bar\vartheta/2}^{2} \} +
\delta_{\bar\vartheta,2\vartheta-1}\int_{s}^{t} \|f(\tau)\|_{\sigma,
\bar\vartheta/2}^{2} d \tau\right)  - \delta^{\prime}\{\eta(t) - \eta(s)\}\\
&  \quad+ \frac{\delta_{0,2\vartheta-1}}{2C_{\vartheta}} \left[  \frac{1}%
{2}\{\|f(t)\|_{2,\vartheta}^{2} - \|f(s)\|_{2,\vartheta}^{2}\} + \int_{s}%
^{t}\left(  \frac{1}{2} \|f(\tau)\|_{\sigma,\vartheta}^{2} - C_{\vartheta}
\|f(\tau)\|_{\sigma}^{2} d \tau\right)  \right] \\
&  \qquad\le C_{\vartheta-1/2} \int_{s}^{t} \|g(\tau)\|_{\infty} \|f(\tau
)\|_{\sigma,\vartheta-1/2}^{2} d \tau+ \frac{\delta_{0,2\vartheta-1}}{2} \int
_{s}^{t} \|g(\tau)\|_{\infty}\|f(\tau)\|_{\sigma,\vartheta}^{2} d \tau.
\end{split}
\]
Note that $\|\cdot\|_{2,\vartheta-1/2} \le\|\cdot\|_{2,\vartheta}$,
$\|\cdot\|_{\sigma,\vartheta-1/2} \le\|\cdot\|_{\sigma,\vartheta}$. Choosing
sequences of $C_{2\vartheta}^{*}$, $\delta_{\bar\vartheta, 2\vartheta}$, and
$C_{\vartheta}$ such that
\[
C_{0}^{*} = 1,\quad\delta_{0,0} = \delta^{\prime}, \quad C_{0} = C,
\]
\begin{equation}
\label{Eq : induction constants 1}C_{\vartheta}^{*} = \frac{\delta
_{0,2\vartheta-1}}{2 C_{\vartheta}},
\end{equation}
\begin{equation}
\label{Eq : induction constants 2}\delta_{\bar\vartheta, 2\vartheta} =
\begin{cases}
\frac{\delta_{0, 2\vartheta-1}}{2}, & \text{ if }\bar\vartheta=0\\
\delta_{\bar\vartheta, 2\vartheta-1}, & \text{ if }\bar2\vartheta= 1,\cdots,
\vartheta-1,\\
\frac{\delta_{0, 2\vartheta-1}}{4C_{\vartheta}}, & \text{ if }\bar\vartheta= 2 \vartheta,
\end{cases}
\end{equation}
and
\begin{equation}
\label{Eq : induction constants 3}C_{\vartheta} = C_{\vartheta-1/2} +
\frac{\delta_{0,2\vartheta-1}}{2},
\end{equation}
we have \eqref{Eq : energy induction} for all $\vartheta$.

Note that from \eqref{Eq : induction constants 1} -
\eqref{Eq : induction constants 3}, we have
\[
\delta_{0,k} = \frac{\delta^{\prime}}{2^{k}}, \text{ for } k = 1,2, \cdots, 2 \vartheta.
\]
\[
C <C_{\vartheta} = C+\sum_{0 \le\bar\vartheta\le 2\vartheta-1} \frac
{\delta_{0,\bar\vartheta}}{2} < C+ \delta^{\prime}< C+1,
\]
\[
\frac{\delta^{\prime}}{2^{2\vartheta}(C+1)}< C_{\vartheta}^{*} < \frac
{\delta^{\prime}}{2^{2\vartheta} C},
\]
and
\[
\delta_{2\vartheta, 2\vartheta} = \frac{\delta_{0,2\vartheta-1}}{4C_{\vartheta}}
= \frac{C_{\vartheta}^{*}}{2}.
\]
Let $\varepsilon= \frac{\delta_{2\vartheta,2\vartheta}}{2 C_{\vartheta}}$. By
Remark \ref{Rmk : choice of delta'}, we can choose $\delta
^{\prime}$ small enough such that
\begin{equation}
\label{Eq : choice of delta'}\delta^{\prime}\eta(t) \le\frac{C_{0}^{*}}%
{4}\|f(t)\|_{2}^{2} = \frac{1}{4}\|f(t)\|_{2}^{2}.
\end{equation}
From \eqref{Eq : energy induction}, we have
\begin{equation}
\label{Eq : energy estimate 2}%
\begin{split}
\frac{C_{\vartheta}^{*}}{2}\|f(t)\|_{2,\vartheta}^{2} + \frac{\delta
_{2\vartheta,2\vartheta}}{2}\int_{s}^{t} \|f(\tau)\|_{\sigma,\vartheta}^{2} d
\tau &  \le\frac{1}{4}\|f(s)\|_{2}^{2} + \sum_{1\le\bar\vartheta\le 2\vartheta}
\frac{C_{\bar\vartheta}^{*}}{2}\|f(s)\|_{2,\bar\vartheta/2}^{2}\\
\quad &  \le\left(  \frac{1}{4} + \frac{\delta^{\prime}}{C} \right)
\|f(s)\|_{2,\vartheta}^{2}\\
\quad &  \le\frac{1}{2}\|f(s)\|_{2,\vartheta}^{2}.%
\end{split}
\end{equation}
Taking $s=0$ and dividing by $\frac{\delta_{2\vartheta,2\vartheta}}{2}$ both sides
of \eqref{Eq : energy estimate 2}, we have
\[
2\|f(t)\|_{2,\vartheta}^{2} + \int_{0}^{t} \|f(\tau)\|_{\sigma,\vartheta}^{2}
d \tau\le\frac{2}{C_{\vartheta}^{*}}\|f(0)\|_{2,\vartheta}^{2}. \le C
2^{2\vartheta}\|f(0)\|_{2,\vartheta}^{2}.%
\]
Therefore, we have \eqref{Eq : energy estimate linear}.

Fix $\vartheta,k \ge0$, by the H\"older inequality and
\eqref{Eq : energy estimate linear},
\begin{equation}
\label{Eq : weight interpolation}%
\begin{split}
	\|f\|_{2,\vartheta}^{2}  &  = \int w^{2\vartheta} f^{2}\\
	&  = \int\left(  w^{2 \left( \vartheta - \frac{1}{2}\right)}f^{2}\right)  ^{\frac{k}{k+1}}\left(w^{2 \left( \vartheta + \frac{k}{2}\right)}f^{2}\right)  ^{\frac{1}{k+1}}\\
	&  \le\left(  \int w^{2 \left( \vartheta - \frac{1}{2}\right)}f^{2}\right)  ^{\frac{k}{k+1}}\left(\int w^{2 \left( \vartheta + \frac{k}{2}\right)}f^{2}\right)  ^{\frac{1}{k+1}}\\
	&  \le\|f\|_{2,\vartheta-1/2}^{2\frac{k}{k+1}} \left(  C 2^{2\vartheta+k}\mathcal{E}_{\vartheta+k/2}(0) \right)  ^{\frac{1}{k+1}}.
\end{split}
\end{equation}
By Lemma \ref{Lemma : Guo2002 corollary1},
\begin{equation}
\label{Eq : sigma norm estimate}\|f\|_{\sigma,\vartheta} \ge\|(1+|v|)^{-1/2} f\|_{2,\vartheta} =
\|f\|_{2,\vartheta-1/2}.
\end{equation}
Combining \eqref{Eq : energy induction}, \eqref{Eq : weight interpolation},
and \eqref{Eq : sigma norm estimate}, we have
\begin{equation}
\label{Eq : decay estimate 2}%
\begin{split}
&  \sum_{1 \le\bar\vartheta\le2\vartheta} \frac{C_{\bar\vartheta}^{*}}{2}
\left(  \|f(t)\|_{2,\bar\vartheta/2}^{2} - \|f(s)\|_{2,\bar\vartheta/2}%
^{2}\right)  + \left\{  \frac{1}{2}\|f(t)\|_{2}^{2} - \delta^{\prime}%
\eta(t)\right\}  - \left\{  \frac{1}{2}\|f(s)\|_{2}^{2} - \delta^{\prime}%
\eta(s)\right\} \\
&  \quad\le-\frac{\delta_{2\vartheta,2\vartheta}}{2} \int_{s}^{t} \|f(\tau
)\|_{\sigma,\vartheta}^{2} d \tau\\
&  \quad\le-\frac{\delta_{2\vartheta,2\vartheta}}{2} \int_{s}^{t} \|f(\tau
)\|_{2,\vartheta-1/2}^{2} d \tau\\
&  \quad\le-\frac{\delta_{2\vartheta,2\vartheta}}{2} \int_{s}^{t} \left(
C2^{2\vartheta+k}\mathcal{E}_{\vartheta+k/2}(0)\right)  ^{-\frac{1}{k}}%
\|f(\tau)\|_{2,\vartheta}^{2\frac{k+1}{k}} d \tau.
\end{split}
\end{equation}
Let
\[
y(t) := \left\{  \frac{1}{2}\|f(t)\|_{2}^{2} - \delta^{\prime}\eta(t)\right\}
+ \sum_{1 \le\bar\vartheta\le2\vartheta} \frac{C_{\bar\vartheta}^{*}}{2}
\|f(t)\|_{2,\bar\vartheta/2}^{2}.
\]
Then
\begin{equation}
\label{Eq : y(t)}\frac{C_{\vartheta}^{*}}{2}\|f(t)\|_{2,\vartheta}^{2} \le
y(t) \le\left(  \frac{1}{2} + \sum_{\bar\vartheta=1}^{2\vartheta} \frac
{C_{\bar\vartheta}^{*}}{2}\right)  \|f(t)\|_{2,\vartheta}^{2} \le
\|f(t)\|_{2,\vartheta}^{2}.
\end{equation}
Combining \eqref{Eq : decay estimate 2}, \eqref{Eq : y(t)}, we have
\[
y(t) - y(s) \le- \frac{\delta_{2\vartheta,2\vartheta}}{2} \int_{s}^{t} \left(
C2^{2\vartheta+k}\mathcal{E}_{\vartheta+k/2}(0)\right)  ^{-\frac{1}{k}}\left(
y(\tau)\right)  ^{\frac{k+1}{k}} d \tau.
\]
Therefore, we have
\begin{equation}
\label{Eq : y'(t)}y^{\prime}(t) \le-\frac{1}{2} \delta_{2\vartheta,2\vartheta
}\left(  C2^{2\vartheta+k}\mathcal{E}_{\vartheta+k/2}(0)\right)  ^{-\frac{1}{k}}
y(t)^{\frac{k+1}{k}}\le- \frac{1}{2^{2\vartheta}C} \left(  C2^{2\vartheta
+k}\mathcal{E}_{\vartheta+k/2}(0)\right)  ^{-\frac{1}{k}} y(t)^{\frac{k+1}{k}} .
\end{equation}
Multiplying \eqref{Eq : y'(t)} by $-\frac{1}{k}y^{-\frac{k+1}{k}}$, we have
\[
\partial_{t} \left(  y(t)^{-\frac{1}{k}}\right)  \ge\frac{1}{2^{2\vartheta}Ck}
\left(  C2^{2\vartheta+k}\mathcal{E}_{\vartheta+k/2}(0)\right)  ^{-\frac{1}{k}}.
\]
Integrating above over $[0,t]$ yields
\[%
\begin{split}
y(t)^{-\frac{1}{k}}  &  \ge\frac{t}{2^{2\vartheta}Ck} \left(  C2^{2\vartheta
+k}\mathcal{E}_{\vartheta+k/2}(0)\right)  ^{-\frac{1}{k}} + y(0)^{-\frac{1}{k}%
}\\
&  \ge\frac{t}{2^{2\vartheta}Ck} \left(  C2^{2\vartheta+k}\mathcal{E}%
_{\vartheta+k/2}(0)\right)  ^{-\frac{1}{k}} + \left(  \|f(0)\|_{2,\vartheta}%
^{2}\right)  ^{-\frac{1}{k}}\\
&  \ge\frac{\left(  C2^{2\vartheta+k}\mathcal{E}_{\vartheta+k/2}(0)\right)
^{-\frac{1}{k}}}{2^{2\vartheta}C} \left(  \frac{t}{k} + 1\right).
\end{split}
\]
Therefore,
\[
\|f(t)\|_{2,\vartheta}^{2} \le\frac{2}{C_{\vartheta}^{*}} y(t) \le
C_{\vartheta,k} \mathcal{E}_{\vartheta+k/2}(0)\left(  1+ \frac{t}{k}\right)
^{-k},
\]
where we use \eqref{Eq : y(t)} in the first inequality. Thus we complete the proof.
\end{proof}

\begin{theorem}
\label{Thm : energy estimate tilde Af}
Assume \eqref{Eq : condition for g}.
Let $\vartheta\in 2^{-1}\mathbb{N} \cup\{0\}$ and $f$ be a classical solution of
\eqref{conservation laws}, \eqref{bar Af}. Then there exist $C,
\varepsilon(\vartheta)>0$ such that if $\|g\|_{\infty}< \varepsilon$, then
\begin{equation}
\label{Eq : energy estimate for tilde Af}\sup_{0 \le s< \infty}\mathcal{E}%
_{\vartheta}(f(s)) \le C 2^{2\vartheta} \mathcal{E}_{\vartheta}(0),
\end{equation}
and for any $t>0$, $k\in\mathbb{N}$,
\begin{equation}
\label{Eq : decay estimate for tilde Af}\|f(t)\|_{2,\vartheta} \le
C_{\vartheta,k} \mathcal{E}_{\vartheta+k/2}(0)\left(  1+ \frac{t}{k}\right)
^{-k/2}.
\end{equation}

\end{theorem}

\begin{proof}[Sketch of proof]
	The proof can be done by choosing $\Gamma = 0$ in Theorem \ref{Thm : energy estimate}.
\end{proof}

\section{$L^{2}-L^{\infty}$ estimate}

\label{sec:L2-Linfty estimate}

\subsection{Local $L^2-L^\infty$ estimate} % (fold)
\label{sub:local_L2-Linfty estimate}

% subsection local_ (end)
%section section_name (end)

In this subsection we will derive a local $L^{\infty}$ estimate for $h$.
\begin{equation}
\label{Eq : M_g^theta}\mathcal{M}_{g}^{\vartheta}h:= (\partial_{t} +
v\cdot\nabla_{x} - \bar A_{g}^{\vartheta})h,
\end{equation}
where $\bar A_{g}^{\vartheta}$ is defined as in \eqref{Eq : bar A_g^theta}.

Here we will refine the results about the $L^{2}$-$L^{\infty}$ estimate in
\cite{mouhot2015holder}. Comparing with \cite{mouhot2015holder}, we have an
additional term; $a_{g}\cdot \nabla_{v} f-2 \frac{\partial_{i} w^{\vartheta}%
}{w^{\vartheta}}\sigma_{G}^{ij} \partial_{j} f$ and a diffusion matrix of $\mathcal{M}_{g}^{\vartheta}$ is not uniformly elliptic. Moreover,
to get a $L^{2}$-$L^{\infty}$ estimate for $\mathbb{T}^{3} \times
\mathbb{R}^{3}$, we need to know the local $L^{2}$-$L^{\infty}$ estimate more explicitly.

Define $Q_{n} := [-t_{n},0]\times\mathbb{T}^{3} \times B(0;R_{n})$, for $t_{n}
\ge t_{n+1}$ and $R_{n} \ge R_{n+1}$. The following estimates are refinements of Lemma 4 - 6 and Theorem 2 and Theorem 7 in
\cite{mouhot2015holder}.

\begin{lemma}
[Lemma 4 in \cite{mouhot2015holder}]%
\label{lemma : gain of integrability w.r.t. v and t}
Assume \eqref{Eq : condition for g}.
Let $h$ be a nonnegative
periodic function in $x$ satisfying $\mathcal{M}_{g}^{\vartheta}h \le0$. Then
$h$ satisfies
\begin{align}
\int_{Q_{1}}|\nabla_{v} h|^{2} \le C \int_{Q_{0}}h^{2}%
\label{Eq : gain of regularity w.r.t. v}\\
\|h\|_{L^{2}_{t}L^{2}_{x}L^{q}_{v}(Q_{1})}^{2} \le C \int_{Q_{0}}%
h^{2}\label{Eq : gain of integrability w.r.t. v}\\
\|h\|_{L^{\infty}_{t}L^{2}_{x}L^{2}_{v}(Q_{1})}^{2} \le C \int_{Q_{0}}h^{2} .
\label{Eq : gain of integrability w.r.t. t}%
\end{align}
for some $q>2$ and $C = \bar C(R_{0})\left(  1+ \frac{1}{t_{0} -
t_{1}} + \frac{1}{R_{0} - R_{1}} + \frac{1}{(R_{0} - R_{1})^{2}} \right)  $.
\end{lemma}

\begin{proof}
Consider a test function $\Phi\in C^{\infty}(\mathbb{R} \times\mathbb{T}^{3}
\times\mathbb{R}^{3})$, periodic with respect to $x$ and $\Phi(t,x,v)=0$ for
$|v|> R_{0}$. Multiplying \eqref{Eq : M_g^theta} by $2h \Phi^{2}$ and
integrating the resulting equation over $\mathcal{R }:= [-t_{0},s]
\times\mathbb{T}^{3} \times B(0;R_{0})$ for some $s\in[-t_{1},0]$, then
\begin{align*}
	&  \int_{\mathcal{R}} \partial_{t}(h^{2})\Phi^{2} + \int_{\mathcal{R}}v\cdot\nabla_{x}(h^{2})\Phi^{2}\\
	&  \quad\le2\int_{\mathcal{R}} \nabla_{v}\cdot(\sigma_{G} \nabla_{v} h) h\Phi^{2} + \int_{\mathcal{R}} a_{g} \cdot\nabla_{v}(h^{2})\Phi^{2} -2\int_{\mathcal{R}} \frac{\nabla_{v} (w^{\vartheta})}{w^{\vartheta}}\cdot\sigma_{G} \nabla_{v}(h^{2})\Phi^{2},
\end{align*}
where $\sigma_G$ is defined as in \eqref{Eq : sigma_u} with $G = \mu + \mu^{1/2}g$.
Using the integration by parts and the positivity of $\sigma_G$, we have
\[%
\begin{split}
&  \int_{\mathcal{R}} \partial_{t}(h^{2} \Phi^{2}) + 2\int_{\mathcal{R}%
}(\nabla_{v} h \cdot\sigma_{G} \nabla_{v} h) \Phi^{2}\\
&  \quad\le\int_{\mathcal{R}} h^{2}\left(  \partial_{t}(\Phi^{2}) +
v\cdot\nabla_{x}(\Phi^{2}) - \nabla_{v}\cdot(\Phi^{2} a_{g}) + 2 \nabla_{v}
\cdot\left(  \Phi^{2} \sigma_{G} \frac{\nabla_{v} (w^{\vartheta}%
)}{w^{\vartheta}}\right)  \right)  - 4 \int_{\mathcal{R}} h \Phi\nabla_{v}
\Phi\cdot\sigma_{G} \nabla_{v} h\\
&  \quad\le\int_{\mathcal{R}} h^{2}\left(  \partial_{t}(\Phi^{2}) +
v\cdot\nabla_{x}(\Phi^{2}) - \nabla_{v}\cdot(\Phi^{2} a_{g}) + 2 \nabla_{v}
\cdot\left(  \Phi^{2} \sigma_{G} \frac{\nabla_{v} (w^{\vartheta}%
)}{w^{\vartheta}}\right)  \right) \\
&  \quad+ \int_{\mathcal{R}}(\nabla_{v} h \cdot\sigma_{G} \nabla_{v} h)
\Phi^{2} + C\int_{\mathcal{R}}(\nabla_{v} \Phi\cdot\sigma_{G} \nabla_{v} \Phi)
h^{2}.
\end{split}
\]
Thus we have
\[%
\begin{split}
&  \int_{\mathcal{R}} \partial_{t}(h^{2} \Phi^{2}) + \min(1,(1+R_{0})^{-3
})\int_{\mathcal{R}}|\nabla_{v} h|^{2} \Phi^{2}\\
&  \quad\le\bar C(d) \max(1, (1+R_{0})^{-1})\bigg(\|\partial_{t}
\Phi\|_{\infty}\|\Phi\|_{\infty}+ R_{0}\|\nabla_{x} \Phi\|_{\infty}%
\|\Phi\|_{\infty}\\
&  \quad\quad+ \|\Phi\|_{\infty}\|a_{g}\|_{\infty}\|\nabla_{v} \Phi\|_{\infty
}+ \|\Phi\|_{\infty}^{2}\|\nabla_{v}\cdot a_{g}\|_{\infty}\\
&  \quad\quad+ \|\nabla_{v}\Phi\|_{\infty}^{2} + \|\Phi\|_{\infty}\|\nabla_{v}
\Phi\|_{\infty}\left\|  \sigma\frac{\nabla_{v} (w^{\vartheta})}{w^{\vartheta}%
}\right\|  _{\infty}+ \|\Phi\|_{\infty}^{2}\left\|  \nabla_{v} \cdot\left(
\sigma\frac{\nabla_{v} (w^{\vartheta})}{w^{\vartheta}}\right)  \right\|
_{\infty}\bigg)\int_{\mathcal{R}} h^{2}.
\end{split}
\]
Choosing $\Phi$ such that $\Phi(-t_{0},x,v)=0$ and $\Phi= 1$ in $Q_{1}$, we
have
\begin{equation}
\label{Eq : gain of integrability w.r.t. v and t-1}
	\int_{\mathbb{T}^{3} \times B(0;R_{1})}h^{2}(s)dxdv +\int_{\mathcal{R}} |\nabla_{v} h|^{2}\le\bar C(R_{0})\left(  1+ \frac{1}{t_{0} - t_{1}} + \frac{1}{R_{0}- R_{1}} + \frac{1}{(R_{0} - R_{1})^{2}} \right)  \int_{\mathcal{R}}h^{2}.
\end{equation}
Especially,
\[
\sup_{s\in[-t_{1},0]}\int_{\mathbb{T}^{3} \times B(0;R_{1})}h^{2}(s)dxdv
\le\bar C(R_{0})\left(  1+ \frac{1}{t_{0} - t_{1}} + \frac{1}{R_{0}
- R_{1}} + \frac{1}{(R_{0} - R_{1})^{2}} \right)  \int_{Q_{0}}h^{2}.
\]
Therefore, we prove \eqref{Eq : gain of integrability w.r.t. t}. Choosing
$s=0$ in \eqref{Eq : gain of integrability w.r.t. v and t-1}, we have
\[
\int_{Q_{1}} |\nabla_{v} h|^{2} \le\bar C(R_{0})\left(  1+ \frac
{1}{t_{0} - t_{1}} + \frac{1}{R_{0} - R_{1}} + \frac{1}{(R_{0} - R_{1})^{2}}
\right)  \int_{Q_{0}}h^{2},
\]
so we obtain \eqref{Eq : gain of regularity w.r.t. v}. Moreover, the Sobolev
inequality implies \eqref{Eq : gain of integrability w.r.t. v}
\end{proof}

\begin{lemma}
[Lemma 5 in \cite{mouhot2015holder}]%
\label{lemma : Gain of regularity w.r.t. x,t}
Assume \eqref{Eq : condition for g}.
If $h$ is a weak solution of
\eqref{Eq : equation for h}, then
\begin{equation}
\label{Eq : Gain of regularity w.r.t. x,t}%
\begin{split}
\|D_{x}^{1/3}h\|_{L^{2}(Q_{1})}^{2}  &  \le C\|h\|_{L^{2}(Q_{0})}^{2}\\
\|D_{t}^{1/3}h\|_{L^{2}(Q_{1})}^{2}  &  \le C\|h\|_{L^{2}(Q_{0})}^{2}%
\end{split}
\end{equation}
for some $C = \bar C(R_{0})\left(  1+ \frac{1}{t_{0} - t_{1}} +
\frac{1}{R_{0} - R_{1}} + \frac{1}{(R_{0} - R_{1})^{2}} \right)  $.
\end{lemma}

\begin{proof}
Let $R_{\frac{1}{2}} = \frac{R_{1} + R_{2}}{2}$ and $Q_{\frac{1}{2}} =
Q_{R_{\frac{1}{2}}}$. Define truncation functions $\chi_{1}$ and $\chi_{1/2}$
such that
\[%
\begin{split}
\chi_{1}  &  =
\begin{cases}
1, & \text{ if }(t,x,v)\in Q_{1}\\
0, & \text{ if }(t,x,v)\in Q_{\frac{1}{2}}^{c},
\end{cases}
\\
\chi_{\frac{1}{2}}  &  =
\begin{cases}
1, & \text{ if }(t,x,v)\in Q_{\frac{1}{2}}\\
0, & \text{ if }(t,x,v)\in Q_{0}^{c}.
\end{cases}
\end{split}
\]
Let $h_{i} = h \chi_{i}$, for $i = 1, \frac{1}{2}$. Then we have
\[
(\partial_{t} + v \cdot\nabla_{x})h_{1} = \nabla_{v} \cdot H_{1} + H_{0}
\quad\text{ in }(-\infty,0]\times\mathbb{R}^{6},
\]
\[%
\begin{split}
H_{1}  &  = \chi_{1} \sigma_{G}\nabla_{v} h_{\frac{1}{2}},\\
H_{0}  &  = - \nabla_{v} \chi_{1} \cdot\sigma_{G}\nabla_{v} h_{\frac{1}{2}} +
\alpha_{1} h_{\frac{1}{2}} + \chi_{1} a_{g}
\cdot\nabla_{v} h_{1/2} -2 \chi_{1} \frac{\nabla_{v}(w^{\vartheta}%
)}{w^{\vartheta}}\cdot\sigma_{G}\nabla_{v} h_{1/2},\\
\alpha_{1}  &  = (\partial_{t} + v\cdot\nabla_{x})\chi_{1},
\end{split}
\]
where $\sigma_G$ is defined as in \eqref{Eq : sigma_u} with $G=\mu + \mu^{1/2}g$.
By Lemma \ref{lemma : gain of integrability w.r.t. v and t},
\[
\|H_{0}\|_{L^{2}(\mathbb{R}^{7})} + \|H_{1}\|_{L^{2}(\mathbb{R}^{7})} \le
C\|h\|_{L^{2}(Q_{0})}%
\]
with $C$ as in the statement. Applying Theorem 1.3 in \cite{bouchut02} with
$p=2$, $r=0$, $\beta=1$, $m=1$, $\kappa=1$ and $\Omega= 1$ yields \eqref{Eq : Gain of regularity w.r.t. x,t}.
\end{proof}

\begin{lemma}
[Lemma 6 in \cite{mouhot2015holder}]%
\label{lemma : gain of integrability w.r.t. x} Under the assumptions of Lemma
\ref{lemma : gain of integrability w.r.t. v and t}, there exists $p>2$ such
that
\begin{equation}
\label{Eq : gain of integrability w.r.t. x}\|h\|_{L^{2}_{t}L^{p}_{x}L^{2}%
_{v}(Q_{1})}^{2} \le C \|h\|_{L^{2}(Q_{0})}^{2}%
\end{equation}
with the same $C$ as in Lemma \ref{lemma : gain of integrability w.r.t. v and t}.
\end{lemma}

\begin{proof}
The proof is exactly the same as in the proof of Lemma 6 in \cite{mouhot2015holder}. We omit the proof.
\end{proof}
The following Lemma is a consequence of Lemma \ref{lemma : gain of integrability w.r.t. v and t}, \ref{lemma : Gain of regularity w.r.t. x,t}.
We omit the proof.
\begin{lemma}
\label{Lemma : Gain of regularity of singed weak solution} Under the
assumptions of Lemma \ref{lemma : Gain of regularity w.r.t. x,t}, we have
\[
\|h\|_{H^{s}_{x,v,t}(Q_{1})} \le C\|h\|_{L^{2}(Q_{0})}%
\]
with the same $C$ as in Lemma \ref{lemma : gain of integrability w.r.t. v and t}
and $s=1/3$.
\end{lemma}

\begin{lemma}
[Theorem 2 in \cite{mouhot2015holder}]\label{lemma : gain of integrability}
Under the assumptions of Lemma
\ref{lemma : gain of integrability w.r.t. v and t}, there exists $q>2$ such
that
\begin{equation}
\label{Eq : gain of integrability}\|h\|_{L^{q}(Q_{1})}^{2} \le C
\|h\|_{L^{2}(Q_{0})}^{2}%
\end{equation}
with the same $C$ as in Lemma \ref{lemma : gain of integrability w.r.t. v and t}.
\end{lemma}

\begin{proof}
The proof is exactly the same as in the proof of Theorem 2 in \cite{mouhot2015holder}. We omit the proof.
\end{proof}

\begin{lemma}
[Theorem 7 in \cite{mouhot2015holder}]
Assume \eqref{Eq : condition for g}.
Let $h$ be a nonnegative
periodic function in $x$ satisfying $\mathcal{M}_{g}^{\vartheta}h \le0$.
Then, there exists $m>1$ such that
\[
\|h\|_{L^{\infty}(Q_{\infty})} \le\bar C (R_{0})^{m} \left(  1+
\frac{1}{\min(t_{0}-t_{\infty}, (R_{0}-R_{\infty})^{2})}\right)  ^{m}
\|h\|_{L^{2}(Q_{0})},
\]
where $Q_{0} = [-t_{0},0]\times\mathbb{T}^{3} \times[-R_{0},R_{0}]$ and
$Q_{\infty}= [-t_{\infty},0] \times\mathbb{T}^{3} \times[-R_{\infty},
R_{\infty}]$.
\end{lemma}

\begin{proof}
Let $\kappa:= q/2>1$. Since $|h|^{q_{n}}$, $q_{n}>1$, is also a sub-solution of
\eqref{Eq : equation for h}, by Lemma \ref{lemma : gain of integrability}
\[
\||h|^{q_{n}}\|_{L^{q}(Q_{n+1})}^{2} \le C_{n} \||h|^{q_{n}}\|_{L^{2}(Q_{n}%
)}^{2},
\]
where $C_{n} = \bar C(R_{0})\left(  1+ \frac{1}{t_{n} - t_{n+1}} +
\frac{1}{R_{n} - R_{n+1}} + \frac{1}{(R_{n} - R_{n+1})^{2}} \right)  $.
Changing $\|\cdot\|_{q}$ to $\|\cdot\|_{2}$ yields
\[
\||h|^{\kappa q_{n}}\|_{L^{2}(Q_{n+1})}^{2} \le C_{n}^{\kappa}\||h|^{q_{n}%
}\|_{L^{2}(Q_{n})}^{2 \kappa}.
\]
Let $q_{n} := \kappa^{n}$, then after iteration we have
\[
\||h|^{q_{n}}\|_{L^{2}(Q_{n})}^{2} \le\prod_{j=1}^{n} C_{n-j}^{\kappa^{j}}
\|h\|_{L^{2}(Q_{0})}^{2\kappa^{n}}.
\]
Changing $\|\cdot\|_{2}$ to $\|\cdot\|_{2q_{n}}$, we have
\[
\|h\|_{L^{2 q_{n}}(Q_{n})}^{2}\le\prod_{j=1}^{n} C_{n-j}^{\kappa^{j-n}}
\|h\|_{L^{2}(Q_{0})}^{2} = \prod_{j=0}^{n-1} C_{j}^{\kappa^{-j}}
\|h\|_{L^{2}(Q_{0})}^{2}.%
\]
Choosing $t_{n}-t_{n+1} = \alpha(t_{0}-t_{\infty})n^{-4}$ and $R_{n}-R_{n+1} =
\beta(R_{0}-R_{\infty})n^{-2}$, we have
\[
C_{j}^{\kappa^{-j}} \le\bar C^{\kappa^{-j}} \left(  C^{\prime}\left(  1+
\frac{1}{\min(t_{0}-t_{\infty}, (R_{0}-R_{\infty})^{2})}\right)  \right)
^{j^{4} \kappa^{-j}}.
\]
Thus,
\[
\prod_{j=0}^{\infty} C_{j}^{\kappa^{-j}} \le C^{\prime m }\bar C(R_{0})^{m}\left(  1+ \frac{1}{\min(t_{0}-t_{\infty}, (R_{0}-R_{\infty
})^{2})}\right)  ^{m}%
\]
for some $m>1$. So the proof is complete.
\end{proof}

\begin{lemma}
Assume \eqref{Eq : condition for g}.
If $h_{+}=\max\left\{  h ,0\right\}  $, where $h$ is a subsolution of
\eqref{Eq : equation for h}, then $h_{+}$ is a subsolution.
\end{lemma}

\begin{proof}
Approximate a convex function $Q\left(  h\right)  \rightarrow h_{+}$ and then
use the convexity of $Q\left(  h\right)  $ such that $Q^{\prime}\left(
h\right)  >0$ and $Q^{\prime\prime}\left(  h\right)  >0$. Applying $Q\left(
h\right)  $ to the equation \eqref{Eq : equation for h}, we complete the proof.
\end{proof}

Let $h$ be a weak solution. Then since $\left\vert h\right\vert =h_{+}-h_{-}$
and $h_{+}=\max\left\{  h,0\right\}  $ are subsolutions (maximum of two
subsolutions is a subsolution) and $h_{-}=\min\left\{  -h,0\right\}  $ is a
supersolution (minimum of two supersolution is a supersolution), we can apply
Lemma above to both $h_{+}$ and $-h_{-}$. Thus we obtain:

\begin{lemma}
[Theorem 7 in \cite{mouhot2015holder}]\label{lemma : L^2-L^infty estimate} 
Assume \eqref{Eq : condition for g}.
Let $h$ be a sub-solution of \eqref{Eq : equation for h}. Then, there exists $m>1$ such
that
\[
\|h\|_{L^{\infty}(Q_{\infty})} \le\bar C (R_{0})^{m} \left(  1+
\frac{1}{\min(t_{0}-t_{\infty}, (R_{0}-R_{\infty})^{2})}\right)  ^{m}
\|h\|_{L^{2}(Q_{0})},
\]
where $Q_{0} = [-t_{0},0]\times\mathbb{T}^{3} \times[-R_{0},R_{0}]$ and
$Q_{\infty}= [-t_{\infty},0] \times\mathbb{T}^{3} \times[-R_{\infty},
R_{\infty}]$.
\end{lemma}

\subsection{$L^2 - L^\infty$ estimate for \eqref{linear landau}} % (fold)
\label{sub:L2-Linfty estimate for linear landau}

We now consider \eqref{linear landau} and let $f$ be a solution of
\eqref{linear landau}. Then we split $f$ into two parts:%
\[
f=f \mathbf{1} _{\left\{  |v| \le M\right\}  }+f \mathbf{1} _{\left\{  |v| \ge
M\right\}  }=:f_{1}+f_{2}.
\]

Let $U\left(  t,s\right)  h$ be a solution of \eqref{Eq : equation for h} corresponding
to the initial times $s$ with the initial data $h$. Then
\[%
\begin{split}
f_{1}(t,x,v)  &  =\mathbf{1}_{\left\{  |v| \le M\right\}  }U(t,0) f_{0} +
\mathbf{1}_{\left\{  |v| \le M\right\}  }\int_{0}^{t}U(t,\tau) \bar
K_{g}^{\vartheta}f(\tau) d\tau\\
&  =\mathbf{1}_{\left\{  |v| \le M\right\}  }U(t,0) f_{0} + \int_{0}%
^{t}\mathbf{1}_{\left\{  |v| \le M\right\}  }U(t,\tau) \bar K_{g}^{\vartheta
}f(\tau) d\tau.
\end{split}
\]

Then we first obtain the $L^{\infty}$estimates for $f_{1}$:

\begin{lemma}
\label{Lemma : L^infty estimate for bdd v} 
Assume \eqref{Eq : condition for g}.
Let $f$ be a weak
solution of \eqref{initial}, \eqref{conservation laws}, and
\eqref{linear landau} in a periodic box in the sense of Definition \ref{Def : linear weak sol}, then there exist $C$, $\beta>0$
satisfying the following property: for any $Z, s, k>1$, and $\vartheta,
l\in\mathbb{N} \cup\{0\}$, there exists $C_{\vartheta,l}$ such that
\begin{equation}
\label{Eq : L^infty estimate for bdd v}%
\begin{split}
&  \left\|  1_{|v|<Zs^{k}} f^{\vartheta}( s) \right\|  _{\infty}\\
&  \quad\le C_{\vartheta,l}\left(  Zs^{k}\right)  ^{\beta}(1+s)^{-l}%
\|f_{0}\|_{2,\vartheta+ l} + \frac{C}{1+Zs^{k}}\sup_{s^{\prime}\in
(s-1,s)}\|f^{\vartheta}(s^{\prime})\|_{\infty}.
\end{split}
\end{equation}

\end{lemma}

\begin{proof}
By the Duhamel principle,
\[%
\begin{split}
\|1_{|v|<Zs^{k}}f^{\vartheta}(s)\|_{L^{\infty}}  &  \le\|1_{|v|<Zs^{k}%
}U(s,s-1)f^{\vartheta}(s-1)\|_{L^{\infty}}\\
&  \quad+ \int_{0}^{1-\varepsilon} \|1_{|v|\le Zs^{k}} U(s,s-1+ \tau) \bar
K^{\vartheta}_{g} f(s -1 +\tau)\|_{\infty}d \tau\\
&  \quad+ \int_{1-\varepsilon}^{1} \|1_{|v|\le Zs^{k}} U(s,s-1+ \tau) \bar
K^{\vartheta}_{g} f(s -1 +\tau)\|_{\infty}d \tau\\
&  =(i)+(ii)+(iii),
\end{split}
\]
where $\bar K_g^{\vartheta}$ is defined as in \eqref{Eq : bar K_g^theta} and $\varepsilon$ is a constant which will be chosen later. By Lemma
\ref{lemma : L^2-L^infty estimate}, there exists $m>0$ such that
\[
(i) \le C\left(  Zs^{k}\right)  ^{m} \left(  \int_{0}^{1} \|U(s^{\prime},s-1)f^{
\vartheta}(s-1)\|_{2}^{2} ds^{\prime}\right)  ^{1/2}.
\]
By Theorem \ref{Thm : energy estimate} and Lemma
\ref{lemma : existence of weak sol h}, for every integer $l$, there exists
$C_{l}$ such that
\[%
\begin{split}
\|U(s^{\prime},s-1)f^{\vartheta}(s-1)\|_{2}  &  \le C\|f^{\vartheta}(s-1)\|_{2} =
C\|f(s-1)\|_{2,\vartheta}\\
&  \le C_{\vartheta,l}\left(  1+\frac{s-1}{l}\right)  ^{-l}\|f_{0}%
\|_{2,\vartheta+l}.\\
&  \le C_{\vartheta,l}(1+s)^{-l}\|f_{0}\|_{2, \vartheta+l}.%
\end{split}
\]
Thus
\[
(i) \le C_{\vartheta,l} \left(  Zs^{k}\right)  ^{m}(1+s)^{-l}\|f_{0}%
\|_{2,\vartheta+ l}.
\]
By the maximum principle and \eqref{Eq : Kf L^infty norm},
\[
(iii) \le C \varepsilon\sup_{s^{\prime}\in(s-1,s)}\|f^{\vartheta}(s^{\prime
})\|_{\infty}.
\]
By Lemma \ref{lemma : L^2-L^infty estimate},
\begin{multline*}
\|1_{|v|\le Zs^{k}} U(s,s-1+\tau) \bar K^{\vartheta}_{g} f(s-1+\tau)\|_{\infty}\\
\le C \left(  Zs^{k}\right)  ^{m}\left(  1+ \frac{1}{1- \tau}\right)  ^{m}
\left(  \int_{s-1+\tau}^{s}\|1_{|v|<2Zs^{k}} U(s^{\prime},s-1+\tau)\bar K^{\vartheta
}_{g} f(s-1+ \tau)\|_{2}^{2} ds^{\prime}\right)  ^{1/2}.%
\end{multline*}
By \eqref{Eq : Kf L2 norm} and Theorem \ref{Thm : energy estimate}, for any
$N>0$,
\begin{equation*}
	\begin{split}
		(ii)  &  \le C\int_{0}^{1- \varepsilon} \left(  Zs^{k}\right)  ^{m}\left(  1+\frac{1}{1- \tau}\right)  ^{m} \left(  \int_{s-1+\tau}^{s}\|\bar K^{\vartheta}_{f} f(s-1+ \tau)\|_{2}^{2} ds^{\prime}\right)  ^{1/2} d \tau\\
		&  \le C\left(  Zs^{k}\right)  ^{m} \left(  1+ \frac{1}{\varepsilon}\right)^{m} \int_{0}^{1} \left(  N^{2} \|f^{\vartheta}(s-1+\tau)\|_{2} +\frac{1}{N}\|f^{\vartheta}(s-1+ \tau)\|_{\infty}\right)  d\tau\\
		&  \le\left(  Zs^{k}\right)  ^{m} \left(  1+ \frac{1}{\varepsilon}\right)^{m} \left(  C_{\vartheta,l} N^{2}(1+s)^{-l} \|f_{0}\|_{2,\vartheta+l} +\frac{C}{N}\sup_{s^{\prime}\in(s-1,s)}\|f^{\vartheta}(s^{\prime})\|_{\infty}\right).
	\end{split}
\end{equation*}
Choose $\varepsilon^{-1} = 1+Zs^{k}$ and $N = (1+Zs^{k}%
)^{2m+1}$. Then
\[
(i)+(ii)+(iii) \le C_{\vartheta,l}\left(  Zs^{k}\right)  ^{\beta}%
(1+s)^{-l}\|f_{0}\|_{2,\vartheta+l} + \frac{C}{1+Zs^{k}}\sup_{s^{\prime}%
\in(s-1,s)}\|f^{\vartheta}(s^{\prime})\|_{\infty},
\]
where $\beta= 6m+2>0$.
\end{proof}

Based on the above results, we will prove Theorem \ref{Thm : L^infty estimate}.
\begin{proof}[Proof of Theorem \ref{Thm : L^infty estimate}]
Choose $\varepsilon$ as in Lemma \ref{lemma : eigenvalue estimate}.
By Lemma \ref{Lemma : L^infty estimate for bdd v}, there exists $l$ such that
\[
\left\|  1_{|v|<Zs^{k}} f^{\vartheta}( s) \right\|  _{L^{\infty}}\le
C_{\vartheta, l, Z} (1+s)^{-2}\|f_{0}\|_{2,\vartheta+l} + \frac{C}{1+Zs^{k}%
}\sup_{s^{\prime}\in(s-1,s)}\|f^{\vartheta}(s^{\prime})\|_{\infty}.
\]

Therefore by the Duhamel principle, the maximum principle, and
\eqref{Eq : Kf L^infty norm unbdd},
\[%
\begin{split}
\|f^{\vartheta}(n+1)\|_{L^{\infty}}  &  \le\|U(n+1,n)f^{\vartheta}(n)\|_{L^{\infty
}}\\
&  \quad+ \int_{0}^{1} \left\|  U(n+1, n + s)\bar K^{\vartheta}_{g} \left(  1_{|v|<
Z(n+s)^{k}}f(n+s) + 1_{|v|>Z(n+s)^{k}} f(n+s) \right)  \right\|  _{L^{\infty}}
ds\\
&  \le\|f^{\vartheta}(n)\|_{L^{\infty}} + C_{\vartheta,l,Z} \int_{0}^{1}
(1+n+s)^{-2}\|f_{0}\|_{2,\vartheta+l}\\
&  \quad+ \frac{C}{1+Z(n+s)^{k}} \sup_{s^{\prime}\in(n+s-1,n+s)}%
\|f^{\vartheta}(s^{\prime})\|_{\infty}ds\\
&  \quad+ \int_{0}^{1} (1+Z(n+s)^{k})^{-1}\|f^{\vartheta}(n+s)\|_{\infty
}ds\\
&  \le\|f^{\vartheta}(n)\|_{L^{\infty}} + C_{\vartheta,l,Z} (1+n)^{-2}%
\|f_{0}\|_{2,\vartheta+l}+ C (Zn^{k})^{-1}\sup_{s^{\prime}\in[n-1,n+1]} \|f^{\vartheta
}(s^{\prime})\|_{\infty},
\end{split}
\]
where $\bar K_g^{\vartheta}$ is defined as in \eqref{Eq : bar K_g^theta}.
After iteration,
\[%
\begin{split}
\|f^{\vartheta}(n+1)\|_{L^{\infty}}  &  \le\|f^{\vartheta}(1)\|_{L^{\infty}} +
C_{\vartheta,l,Z} \sum_{\bar{n}=1}^{n} (1+\bar{n})^{-2} \|f_{0}\|_{2,\vartheta
+l}\\
&  \quad+ C Z^{-1} \sum_{\bar{n}=1}^{n} \bar{n}^{-k}\sup
_{s\in[0,n+1]} \|f^{\vartheta}(s)\|_{\infty}.
\end{split}
\]
Choose large $k$ and $Z$ such that $-k < -1$ and $CZ^{-1}%
\sum_{\bar{n}=1}^{\infty}\bar{n}^{k(-1)} \le\varepsilon_{0}$, where
$\varepsilon_{0}$ will be determined later. Then
\[
\|f^{\vartheta}(n+1)\|_{L^{\infty}} \le\|f^{\vartheta}(1)\|_{L^{\infty}} +
C_{\vartheta,l,Z}\|f_{0}\|_{2,\vartheta+l} + \varepsilon_{0}\sup_{s\in
[0,n+1]}\|f^{\vartheta}(s)\|_{\infty}.
\]
Since $n$ is an arbitrary integer,
\[
\sup_{s=1,2,... n+1}\|f^{\vartheta}(s)\|_{L^{\infty}} \le\|f^{\vartheta
}(1)\|_{L^{\infty}} + C_{\vartheta,l,Z}\|f_{0}\|_{2,\vartheta,+l} +
\varepsilon_{0}\sup_{s\in[0,n+1]}\|f^{\vartheta}(s)\|_{\infty}.
\]
By the Duhamel principle, the maximum principle, and
\eqref{Eq : Kf L^infty norm unbdd}.
\[%
\begin{split}
\|f^{\vartheta}(n+t)\|_{L^{\infty}}  &  \le\|U(n+t,n)f^{\vartheta}(n)\|_{L^{\infty
}} + \int_{0}^{t} \left\|  U(n+t,n+s)\bar K^{\vartheta}_{g} f(n+s)\right\|
_{L^{\infty}} ds\\
&  \le\|f^{\vartheta}(n)\|_{L^{\infty}} + \int_{0}^{t} \left\|  \bar
K^{\vartheta}_{g} f(n+s)\right\|  _{L^{\infty}} ds\\
&  \le\|f^{\vartheta}(n)\|_{L^{\infty}} + C \int_{0}^{t} \left\|
f^{\vartheta}(n+s)\right\|  _{L^{\infty}} ds,
\end{split}
\]
for $t\in [0,1]$.
By the Gronwall inequality,
\[
\|f^{\vartheta}(n+t)\|_{L^{\infty}} \le C\|f^{\vartheta}(n)\|_{\infty}, \text{
for all }t\in[0,1].
\]
Therefore,
\[
\sup_{s\in[0,n+1]}\|f^{\vartheta}(s)\|_{L^{\infty}} \le C\|f^{\vartheta}%
_{0}\|_{L^{\infty}} + C_{\vartheta,l}\|f_{0}\|_{2,\vartheta+l} +
C\varepsilon_{0}\sup_{s\in[0,n+1]}\|f^{\vartheta}(s)\|_{\infty}.
\]
Now, we choose a small $\varepsilon_{0}$ satisfying $C\varepsilon_{0}< 1/2$, and then
absorb the last term on the RHS to the LHS. Then, we have \eqref{Eq : L^infty estimate linear} in case of $\vartheta_0 = 0$ by taking $l_0(0) = l$.

By \eqref{Eq : L^infty estimate for bdd v}, there exist $C$, $l_{1}(\vartheta_0)$ such
that
\[
\|1_{|v| < (1+t)^{\vartheta_{0}}}f^{\vartheta}(t)\|_{\infty} \le
C(1+t)^{-\vartheta_{0}}(\|f_{0}\|_{2,\vartheta+l_{1}(\vartheta_{0})} + \|f\|_{\infty,
\vartheta}).
\]
Thus, by Proposition \ref{Prop : L2Linfty}, we have
\[%
\begin{split}
	\|f(t)\|_{\infty, \vartheta}  &  \le\|1_{|v|< (1+t)}f(t)\|_{\infty, \vartheta} + \|1_{|v| \ge(1+t)}%
	f(t)\|_{\infty, \vartheta}\\
	&  \le C(1+t)^{-\vartheta_{0}}\left(\|f_{0}\|_{2,\vartheta+l_{1}(\vartheta_{0})} + \sup_{0 \le s \le t}\|f(s)\|_{\infty,\vartheta}\right) + C(1+t)^{-\vartheta_{0}}\|1_{|v| \ge(1+t)}f(t)\|_{\infty, \vartheta+\vartheta_{0}}\\
	&  \le C(1+t)^{-\vartheta_{0}}\left(\|f_{0}\|_{2,\vartheta+l_{1}(\vartheta_{0})} +  \sup_{0 \le s \le t}\|f(s)\|_{\infty,\vartheta + \vartheta_{0}}\right)\\
	&  \le C(1+t)^{-\vartheta_{0}}(\|f_{0}\|_{2,\vartheta+l_{1}(\vartheta_{0})} +
	\|f_{0}\|_{2,\vartheta+\vartheta_{0}+l_{0}(0)} + \|f_{0}\|_{\infty,	\vartheta+\vartheta_{0}})\\
	&  \le C(1+t)^{-\vartheta_{0}}\|f_{0}\|_{\infty,\vartheta+l_{0}(\vartheta_0)},
\end{split}
\]
where $l_{0}(\vartheta_0) = \max\{l_{1}({\vartheta_{0}}), \vartheta_{0} + l_{0}(0)\} + 2$.
\end{proof}

\begin{lemma}
\label{Lemma : weighted f^p norm} 
Assume \eqref{Eq : condition for g}.
Let $f$ be a strong solution of
\eqref{initial}, \eqref{conservation laws}, and \eqref{linear landau} in a
periodic box. Let $\beta>0$ and $p>2$ be given constants. Then there exist
$l\in\mathbb{N}$ and $C_{\beta,l}$ such that
\begin{equation}
\label{Eq : weighted f^p norm}%
	\begin{split}
		\left(  \int_{0}^{t} \|f(s)\|_{p,\beta}^{p} ds \right)  ^{1/p}&  \le C_{\beta,l_{0}} \|f_{0}\|_{2,\beta+l}^{2/p}\left(  \|f_{0}\|_{\infty,\beta}+\|f_{0}\|_{2,\beta+l}\right)  ^{(p-2)/p}\\
		&  \le C_{\beta,l_{0}} \left(  \|f_{0}\|_{\infty,\beta} + \|f_{0}\|_{2,\beta+l}\right)  .
	\end{split}
\end{equation}

\end{lemma}

\begin{proof}
By Theorem \ref{Thm : energy estimate} and Theorem
\ref{Thm : L^infty estimate}, there exist $l\in\mathbb{N}$ and $C_{\beta,l}$ such that
\[%
\begin{split}
\|f(s)\|_{2,\beta}  &  \le C_{\beta,l_{0}} (1+s)^{-l} \|f_{0}\|_{2,\beta+l}\\
\|f(s)\|_{\infty, \beta}  &  \le C_{\beta,l_{0}} \left(  \|f_{0}\|_{\infty, \beta}
+ \|f_{0}|_{2,\beta+l}\right)  .
\end{split}
\]
By the interpolation, we have
\[
\|f(s)\|_{p,\beta}^{p} \le(C_{\beta,l_{0}})^{p} (1+s)^{-2l} \|f_{0}\|_{2,\beta+l}%
^{2}\left(  \|f_{0}\|_{\infty,\beta} + \|f_{0}\|_{2,\beta+l}\right)  ^{p-2}.
\]
Taking the integral over $s\in(0,\infty)$, we have the first inequality
of \eqref{Eq : weighted f^p norm}. The second inequality of
\eqref{Eq : weighted f^p norm} comes from the Young inequality, then we
complete the proof.
\end{proof}

\subsection{$L^2-L^\infty$ estimate for \eqref{bar Af}} % (fold)
\label{sub:l2-linfty estimate for tilde af}
We will derive another type of $L^2-L^\infty$ estimate to obtain a uniform H\"older estimate for a weak solution of \eqref{linear landau} in the sense of Definition \ref{Def : linear weak sol}.
The proof is similar to the case of Section \ref{sub:L2-Linfty estimate for linear landau}.

Let us multiply \eqref{bar Af} by $w^{\vartheta}$, then $h:=
w^{\vartheta}f = f^{\vartheta}$ satisfies
\begin{equation}
\label{Eq : weighted tilde Af}(\partial_{t} + v\cdot\nabla_{x} - \bar
A_{g}^{\vartheta})h = \tilde K_{g}^{\vartheta}h,
\end{equation}
where
\begin{equation}
\label{tilde K_g^theta}\tilde K^{\vartheta}_{g} h = \left(  2 \frac
{\partial_{i} w^{\vartheta}\partial_{j} w^{\vartheta}}{w^{2 \vartheta}}%
\sigma^{ij}_{G} -\frac{\partial_{ij}w^{\vartheta}}{w^{\vartheta}}\sigma
^{ij}_{G} - \frac{\partial_{j}w^{\vartheta}}{w^{\vartheta}}\partial_{i}%
\sigma^{ij}_{G} - \frac{\partial_{i} w^{\vartheta}}{w^{\vartheta}}a_{g}^{i}
\right)  h.
\end{equation}
Similar to Definition \ref{Def : linear weak sol}, we can define a weak solution of \eqref{Eq : weighted tilde Af}.

Then we split $f^{\vartheta}$ into two parts:
\[
f^{\vartheta}=f^{\vartheta}\mathbf{1} _{\left\{  |v| \le M\right\}
}+f^{\vartheta}\mathbf{1} _{\left\{  |v| \ge M\right\}  }=:f_{1}+f_{2}.
\]
Let $U^{\vartheta}(t,s)f_{0}$ be a weak solution of
\eqref{Eq : M_g^theta} in the sense of Definition \ref{Def : weak sol for h} corresponding to the initial data $f_{0}$ with the initial time $t=s$, then we have
\begin{align*}
	f_{1}(t) &= \mathbf{1}_{|v| \le M} U^{\vartheta}(t,0)f_{0}^{\vartheta} + \mathbf{1}_{|v| \le M} \int_{0}^{t}U^{\vartheta}(t, \tau)\tilde K_{g}^{\vartheta} f^{\vartheta}(\tau) d \tau\\
	&=\mathbf{1}_{|v| \le M} U^{\vartheta}(t,0)f_{0}^{\vartheta} +  \int_{0}^{t}\mathbf{1}_{|v| \le M} U^{\vartheta}(t, \tau)\tilde K_{g}^{\vartheta} f^{\vartheta}(\tau) d \tau.
\end{align*}

\begin{lemma}
\label{lemma : Kf estimate tilde M}
Assume \eqref{Eq : condition for g}.
There exists $C=C_{\vartheta}>0$ such that
\begin{equation}
\label{Eq : Kf L^infty norm tilde M}\|\tilde K_{g}^{\vartheta}f^{\vartheta
}\|_{L^{\infty}(\mathbb{T}^{3} \times\mathbb{R}^{3})}\le C\|f^{\vartheta
}\|_{L^{\infty}(\mathbb{T}^{3} \times\mathbb{R}^{3})},
\end{equation}
\begin{equation}
\label{Eq : Kf L^infty norm unbdd tilde M}\|\tilde K_{g}^{\vartheta}%
1_{|v|>M}f^{\vartheta}\|_{L^{\infty}(\mathbb{T}^{3} \times\mathbb{R}^{3})}\le
C(1+M)^{-1}\|f^{\vartheta}\|_{L^{\infty}(\mathbb{T}^{3} \times
\mathbb{R}^{3})},
\end{equation}
and
\begin{equation}
\label{Eq : Kf L2 norm tilde M}
	\|\tilde K_{g}^{\vartheta}f^{\vartheta}\|_{L^{2}(\mathbb{T}^{3} \times\mathbb{R}^{3})} \le C\|f^{\vartheta}\|_{L^{2}(\mathbb{T}^{3} \times\mathbb{R}^{3})}.
\end{equation}
\end{lemma}

\begin{proof}
Since
\[
\tilde K^{\vartheta}_{g} f^{\vartheta} = \left(  2 \frac{\partial_{i}
w^{\vartheta}\partial_{j} w^{\vartheta}}{w^{2 \vartheta}}\sigma^{ij}_{G}
-\frac{\partial_{ij}w^{\vartheta}}{w^{\vartheta}}\sigma^{ij}_{G} -
\frac{\partial_{j}w^{\vartheta}}{w^{\vartheta}}\partial_{i}\sigma^{ij}_{G} -
\frac{\partial_{i} w^{\vartheta}}{w^{\vartheta}}a_{g}^{i} \right)
f^{\vartheta}%
\]
and by Lemma \ref{Lemma : Guo2002 lemma2},
\[
\left|  2 \frac{\partial_{i} w^{\vartheta}\partial_{j} w^{\vartheta}}{w^{2
\vartheta}}\sigma_{G}^{ij}\right|  +\left|  \frac{\partial_{ij}w^{\vartheta}%
}{w^{\vartheta}}\sigma_{G}^{ij}\right|  + \left|  \frac{\partial
_{j}w^{\vartheta}}{w^{\vartheta}}\partial_{i}\sigma_{G}^{ij}\right|  + \left|
\frac{\partial_{i} w^{\vartheta}}{w^{\vartheta}}a_{g}^{i} \right|  \le
C(1+|v|)^{-1}.%
\]
So the proof is complete.
\end{proof}

\begin{lemma}
\label{Lemma : L^infty estimate for bdd v tilde M}
Assume \eqref{Eq : condition for g}.
Let $f$ be a weak
solution of \eqref{bar Af} in a periodic box in the sense of Definition \ref{Def : weak sol for h}, then there exist $C$,
$\beta>0$ satisfying the following property: for any $Z$, $s>1$, $\vartheta$,
and $k>0$, and $l\in\mathbb{N}$, there exists $C_{\vartheta,l}$ such that
\begin{equation}
\label{Eq : L^infty estimate for bdd v tilde M}%
\begin{split}
&  \left\|  1_{|v|<Zs^{k}} f^{\vartheta}( s) \right\|  _{\infty}\\
&  \quad\le C_{\vartheta,l}\left(  Zs^{k}\right)  ^{\beta}(1+s)^{-l}%
\|f_{0}\|_{2,\vartheta+ l} + \frac{C}{1+Zs^{k}}\sup_{s^{\prime}\in
(s-1,s)}\|f^{\vartheta}(s^{\prime})\|_{\infty},
\end{split}
\end{equation}
for any $s\ge1$.
\end{lemma}

\begin{proof}
By the Duhamel principle,
\[%
\begin{split}
\|1_{|v|<Zs^{k}}f^{\vartheta}(s)\|_{L^{\infty}}  &  \le\|1_{|v|<Zs^{k}}
U^{\vartheta}(s, s-1)f^{\vartheta}(s-1)\|_{L^{\infty}}\\
&  \quad+ \int_{0}^{1-\varepsilon} \|1_{|v|\le Zs^{k}}U^{\vartheta}(s, s-1+
\tau) \tilde K^{\vartheta}_{g} f^{\vartheta}(s -1 +\tau)\|_{\infty}d \tau\\
&  \quad+ \int_{1-\varepsilon}^{1} \|1_{|v|\le Zs^{k}}U^{\vartheta}(s,s-1+
\tau) \tilde K^{\vartheta}_{g} f^{\vartheta}(s -1 +\tau)\|_{\infty}d \tau\\
&  =(i)+(ii)+(iii).
\end{split}
\]
By Lemma \ref{lemma : L^2-L^infty estimate}, there exists $m>0$ such
that
\[
(i) \le C\left(  Zs^{k}\right)  ^{m} \left(  \int_{s-1}^{s} \|
U^{\vartheta}(s^{\prime},s-1)f^{\vartheta}(s-1)\|_{2}^{2} ds^{\prime}\right)  ^{1/2}.
\]
By Theorem \ref{Thm : energy estimate tilde Af} and Lemma
\ref{lemma : existence of weak sol h}, for every integer $l$, there
exists $C_{l}$ such that
\[%
\begin{split}
\|U^{\vartheta}(s^{\prime},s-1)f^{\vartheta}(s-1)\|_{2}  &  \le C\|f^{\vartheta}(s-1)\|_{2}
= C\|f(s-1)\|_{2,\vartheta}\\
&  \le C_{\vartheta,l}\left(  1+\frac{s-1}{l}\right)  ^{-l}\|f_{0}%
\|_{2,\vartheta+l}\\
&  \le C_{\vartheta,l}(1+s)^{-l}\|f_{0}\|_{2, \vartheta+l}.%
\end{split}
\]
Thus
\[
(i) \le C_{\vartheta,l} \left(  Zs^{k}\right)  ^{m}(1+s)^{-l}\|f_{0}%
\|_{2,\vartheta+ l}.
\]
By the maximum principle and \eqref{Eq : Kf L^infty norm tilde M},
\[
(iii) \le C \varepsilon\sup_{s^{\prime}\in(s-1,s)}\|f^{\vartheta}(s^{\prime
})\|_{\infty}.
\]
By Lemma \ref{lemma : L^2-L^infty estimate},
\begin{multline*}
\|1_{|v|\le Zs^{k}} U^{\vartheta}(s,s-1+ \tau) \tilde K^{\vartheta}_{g}
f(s-1+\tau)\|_{\infty}\\
\le C \left(  Zs^{k}\right)  ^{m}\left(  1+ \frac{1}{1- \tau}\right)  ^{m}
\left(  \int_{s}^{s-1+ \tau}\|1_{|v|<2Zs^{k}} U^{\vartheta}(s^{\prime
}, s-1+\tau)\tilde K^{\vartheta}_{g} f(s-1+ \tau)\|_{2}^{2} ds^{\prime}\right)  ^{1/2}.%
\end{multline*}
By \eqref{Eq : Kf L2 norm tilde M} and Theorem
\ref{Thm : energy estimate tilde Af}, for any $N>0$,
\[%
\begin{split}
(ii)  &  \le C\int_{0}^{1- \varepsilon} \left(  Zs^{k}\right)  ^{m}\left(  1+
\frac{1}{1- \tau}\right)  ^{m} \left(  \int_{0}^{1- \tau}\|\tilde
K^{\vartheta}_{g} f^{\vartheta}(s-1+ \tau)\|_{2}^{2} ds^{\prime}\right)
^{1/2} d \tau\\
&  \le C\left(  Zs^{k}\right)  ^{m} \left(  1+ \frac{1}{\varepsilon}\right)
^{m} \int_{0}^{1} C \|f^{\vartheta}(s-1+\tau)\|_{2} d \tau\\
&  \le\left(  Zs^{k}\right)  ^{m} \left(  1+ \frac{1}{\varepsilon}\right)
^{m} C_{\vartheta,l}(1+s)^{-l} \|f_{0}\|_{2,\vartheta-l}.%
\end{split}
\]
Choose $\varepsilon^{-1} = 1+Zs^{k}$. Then
\[
(i)+(ii)+(iii) \le C_{\vartheta,l}\left(  Zs^{k}\right)  ^{\beta}%
(1+s)^{-l}\|f_{0}\|_{2,\vartheta-l} + \frac{C}{1+Zs^{k}}\sup_{s^{\prime}%
\in(s-1,s)}\|f^{\vartheta}(s^{\prime})\|_{\infty},
\]
where $\beta= 2m+1$.
\end{proof}

\begin{theorem}
\label{Thm : L^infty estimate tilde M} 
Assume \eqref{Eq : condition for g}.
Let $f$ be a weak solution of \eqref{conservation laws}, \eqref{bar Af} in a
periodic box in the sense of Definition \ref{Def : weak sol for h}. Then there exists $l$ such that for every $\vartheta>0$,
\begin{equation}
\label{Eq : L^infty estimate tilde M}\| f^{\vartheta}( t) \| _{L^{\infty}}\le
C \|f_{0}^{\vartheta}\|_{L^{\infty}} + C_{\vartheta,l}\|f_{0}\|_{2,\vartheta
+l}\le C\|f_0\|_{\infty, \vartheta+ l_{0}},\quad \text{ for any }t>0,
\end{equation}
where $l_0 = l+2$.
\end{theorem}

\begin{proof}
By Lemma \ref{Lemma : L^infty estimate for bdd v tilde M}, there exists
$l_{0}$ such that for $l> l_{0}$,
\[
\left\|  1_{|v|<Zs^{k}} f^{\vartheta}( s) \right\|  _{L^{\infty}}\le
C_{\vartheta, l, Z} (1+s)^{-2}\|f_{0}\|_{2,\vartheta+l} + \frac{C}{1+Zs^{k}%
}\sup_{s^{\prime}\in(s-1,s)}\|f^{\vartheta}(s^{\prime})\|_{\infty}.
\]

Therefore by the Duhamel principle, the maximum principle, and
\eqref{Eq : Kf L^infty norm unbdd tilde M},
\[%
\begin{split}
\|f^{\vartheta}(n+1)\|_{L^{\infty}}  &  \le\|U^{\vartheta}(n+1,n)f^{\vartheta}(n)\|_{L^{\infty}}\\
&  \quad+ \int_{n}^{n+1} \left\|  U^{\vartheta}(n,n-1+s)\tilde K^{\vartheta
}_{g} \left(  1_{|v|< Z(n+s)^{k}}f^{\vartheta}(n+s) + 1_{|v|>Z(n+s)^{k}}
f^{\vartheta}(n+s) \right)  \right\|  _{L^{\infty}} ds\\
&  \le\|f^{\vartheta}(n)\|_{L^{\infty}} + C_{\vartheta,l,Z} \int_{0}^{1}
(1+n+s)^{-2}\|f_{0}\|_{2,\vartheta+l}\\
&  \quad+ \frac{C}{1+Z(n+s)^{k}} \sup_{s^{\prime}\in(n+s-1,n+s)}%
\|f^{\vartheta}(s^{\prime})\|_{\infty}ds\\
&  \quad+ \int_{0}^{1} (1+Z(n+s)^{k})^{-1}\|f^{\vartheta}(n+s)\|_{\infty
}ds\\
&  \le\|f^{\vartheta}(n)\|_{L^{\infty}} + C_{\vartheta,l,Z} (1+n)^{-2}%
\|f_{0}\|_{2,\vartheta+l}\\
&  \quad+ C (Zn^{k})^{-1}\sup_{s^{\prime}\in[n-1,n+1]} \|f^{\vartheta
}(s^{\prime})\|_{\infty}.
\end{split}
\]
After iteration,
\[%
\begin{split}
\|f^{\vartheta}(n+1)\|_{L^{\infty}}  &  \le\|f^{\vartheta}(1)\|_{L^{\infty}} +
C_{\vartheta,l,Z} \sum_{\bar{n}=1}^{n} (1+\bar{n})^{-2} \|f_{0}\|_{2,\vartheta
+l}\\
&  \quad+ C Z^{-1} \sum_{\bar{n}=1}^{n} \bar{n}^{k(-1)}\sup
_{s\in[0,n+1]} \|f^{\vartheta}(s)\|_{\infty}.
\end{split}
\]
Choose large $k$ and $Z$ such that $k(-1) < -1$ and $CZ^{-1}%
\sum_{\bar{n}=1}^{\infty}\bar{n}^{k(-1)} \le\varepsilon$, where
$\varepsilon$ will be determined later. Then
\[
\|f^{\vartheta}(n+1)\|_{L^{\infty}} \le\|f^{\vartheta}(1)\|_{L^{\infty}} +
C_{\vartheta,l,Z}\|f_{0}\|_{2,\vartheta+l} + \varepsilon\sup_{s\in
[0,n+1]}\|f^{\vartheta}(s)\|_{\infty}.
\]
Since $n$ is an arbitrary integer,
\[
\sup_{s=1,2,... n+1}\|f^{\vartheta}(s)\|_{L^{\infty}} \le\|f^{\vartheta
}(1)\|_{L^{\infty}} + C_{\vartheta,l,Z}\|f_{0}\|_{2,\vartheta,+l} +
\varepsilon\sup_{s\in[0,n+1]}\|f^{\vartheta}(s)\|_{\infty}.
\]
By the Duhamel principle, the maximum principle, and
\eqref{Eq : Kf L^infty norm unbdd tilde M}
\[%
\begin{split}
\|f^{\vartheta}(n+t)\|_{L^{\infty}}  &  \le\|U^{\vartheta}
(n+t,n)f^{\vartheta}(n)\|_{L^{\infty}} + \int_{n}^{n+t} \left\|
U^{\vartheta}(n+t,n+s)\tilde K^{\vartheta}_{g} f^{\vartheta}(n+s)\right\|
_{L^{\infty}} ds\\
&  \le\|f^{\vartheta}(n)\|_{L^{\infty}} + \int_{0}^{t} \left\|  \tilde
K^{\vartheta}_{g} f^{\vartheta}(n+s)\right\|  _{L^{\infty}} ds\\
&  \le\|f^{\vartheta}(n)\|_{L^{\infty}} + C \int_{0}^{t} \left\|
f^{\vartheta}(n+s)\right\|  _{L^{\infty}} ds.
\end{split}
\]
By the Gronwall inequality,
\[
\|f^{\vartheta}(n+t)\|_{L^{\infty}} \le C\|f^{\vartheta}(n)\|_{\infty}, \text{
for all }t\in[0,1].
\]
Therefore,
\[
\sup_{s\in[0,n+1]}\|f^{\vartheta}(s)\|_{L^{\infty}} \le C\|f^{\vartheta}%
_{0}\|_{L^{\infty}} + C_{\vartheta,l}\|f_{0}\|_{2,\vartheta+l} +
C\varepsilon\sup_{s\in[0,n+1]}\|f^{\vartheta}(s)\|_{\infty}.
\]
Now, we choose small $\varepsilon$ satisfying $C\varepsilon< 1/2$, and then
absorb the last term on the RHS to the LHS.
Thus, we obtain the first inequality of \eqref{Eq : L^infty estimate tilde M}.
The seconds inequality of \eqref{Eq : L^infty estimate tilde M} is a consequence of Proposition \ref{Prop : L2Linfty}.
\end{proof}

% subsection  (end)

\section{$L^\infty$ to H\"older Estimate}

\label{sec : Linfty to Holder Estimate} 
\subsection{Local H\"older estimate} % (fold)
\label{sub:local_holder_estimate}

In this subsection, we will derive a local H\"older estimate for \eqref{bar Af}.
We redefine $Q_{R}(z_{0}) := (t_{0}-R^{2},t_{0}]\times B(x_{0};R^{3}) \times
B(v_{0};R)$, $z_{0} = (t_{0}, x_{0}, v_{0})$, and $Q_{R} := Q_{R}((0,0,0))$.

Since we consider the local properties of the solution on the interior part,
we can use the technique in \cite{golse2016harnack} for our modified operator
$\tilde{A_{g}}$. 
In this subsection, we assume that $g$ satisfies the conditions in Lemma
\ref{lemma : eigenvalue estimate}.

First, we introduce a De Giorgi-type lemma.
\begin{lemma}[Lemma 13 in \cite{golse2016harnack}]
Assume \eqref{Eq : condition for g}.
Let $\hat{Q}:= Q_{1/4}(0,0,-1)$.
For any (universal) constants $\delta_1 \in (0,1)$ and $\delta_2 \in (0,1)$ there exist $\nu > 0$ and $\vartheta\in (0,1)$ (both universal) such that for any solution $f$ of \eqref{bar Af} in
$Q_{2}$ with $|f|\le1$ and
\[%
\begin{split}
|\{f\ge 1- \vartheta\} \cap Q_{1/4} |  &  \ge\delta_{1}|Q_{1/4}|,\\
|\{f\le0\} \cap \hat{Q} |  &  \ge\delta_{2}|\hat{Q}|,%
\end{split}
\]
we have
\[
|\{0<f<1- \vartheta\}\cap B_{1} \times B_{1} \times (-2,0]| \ge \nu.
\]

\end{lemma}
\begin{proof}
The proof is exactly the same as \cite{golse2016harnack}. We omit the proof.
\end{proof}

\begin{lemma}[Lemma 17 in \cite{golse2016harnack}]
Assume \eqref{Eq : condition for g}.
Let $\hat{Q}:= Q_{1/4}(0,0,-1)$ and $f$ be a weak solution of \eqref{bar Af} in $Q_{2}$ in the sense of Definition \ref{Def : weak sol for h} with $|f| \le1$. If
\[
|\{f\le0\}\cap \hat{Q}| \ge \delta_{2}|\hat{Q}|,
\]
then
\[
\sup_{Q_{1/8}} f \le1- \lambda
\]
for some $\lambda\in(0,1)$, depending only on dimension and the eigenvalue of
$\sigma$.
\end{lemma}

\begin{proof}
The proof is exactly the same as \cite{golse2016harnack}. We omit the proof.
\end{proof}

The following lemma can be derived by the previous lemma.

\begin{lemma}
\label{Lemma : decreasing oscillation Q_2} 
Assume \eqref{Eq : condition for g}.
Let $f$ be a weak solution of
\eqref{bar Af} in $Q_{2}$ in the sense of Definition \ref{Def : weak sol for h} with $|f| \le1$. Then
\[
\sup_{Q_{1/8}} f - \inf_{Q_{1/8}} f \le2- \lambda
\]
for some $\lambda\in(0,2)$, depending only on dimension and the eigenvalue of
$\sigma$.
\end{lemma}

By the scaling argument, $Q_{2}$ and $Q_{1/8}$ can be replaced by $Q_{2r}$ and $Q_{r/8}$.

\begin{lemma}
\label{Lemma : decreasing oscillation Q_2r} 
Assume \eqref{Eq : condition for g}.
Let $f$ be a weak solution of
\eqref{bar Af} in $Q_{2r}$ in the sense of Definition \ref{Def : weak sol for h} with $|f| \le1$. For any subset $Q\subset
\mathbb{R}^{7}$, define
\[
\underset{Q}{\text{Osc}} f := \sup_{(t^{\prime},x^{\prime},v^{\prime})\in Q}
f(t^{\prime},x^{\prime},v^{\prime}) - \inf_{(t^{\prime},x^{\prime},v^{\prime
})\in Q} f(t^{\prime},x^{\prime},v^{\prime}),
\]
Then for every $r\le1$
\[
\underset{Q_{r/8}}{\text{Osc}} f \le\left(  1- \frac{\lambda}{2}\right)
\underset{Q_{2r}}{\text{Osc}} f
\]
for some $\lambda\in(0,2)$, depending only on dimension and the eigenvalue of
$\sigma$.
\end{lemma}

\begin{proof}
Define
\[
\bar F(t,x,v):=\frac{2}{\underset{Q_{2r}}{\text{Osc}} f} \left(  f(r^{2} t, r^{3}
x, r v) - \frac{\sup_{Q_{2r}} f + \inf_{Q_{2r}} f}{2}\right)  .
\]
Then $\bar F$ satisfies
\[
\bar F_{t} + v\cdot\partial_{x} \bar F = \tilde{A_{g}}^{r} \bar F
\]
\[
\tilde{A_{g}}^{r} \bar F(t,x,v) := \nabla_{v}(\sigma_{G}(r^{2} t, r^{3} x,
rv) \nabla_{v} \bar F(t,x,v)) + r a_{g}(r^{2} t, r^{3} x, rv) \cdot
\nabla_{v} \bar F(t,x,v).
\]
and then apply Lemma \ref{Lemma : decreasing oscillation Q_2}.
\end{proof}

Now we establish the H\"older continuity at $v=0$.

\begin{lemma}
[Holder continuity near $v=0$]\label{Lemma : Holder near v=0}
Assume \eqref{Eq : condition for g}.
Let $f$ be a weak solution of \eqref{bar Af} in $\Omega_{R}(t_{0},x_{0},0)$ in the sense of Definition \ref{Def : weak sol for h}. Then there exist
a uniform constant $C>0$ and a constant $\alpha\in(0,1)$ depending only on
dimension and the eigenvalue of $\sigma_{G}$ such that
\[
\|f\|_{C^{\alpha}(Q_{R/128}(t_{0},x_{0},0))} \le\frac{C}{R^{3\alpha}%
}\|f\|_{L^{\infty}(Q_{R}(t_{0},x_{0},0))},
\]
for every $R < 1$.
\end{lemma}

\begin{proof}
We first prove
\[
\sup_{(s,y,w)\in Q_{R/16}} \frac{|f(s,y,w)-f(0,0,0)|}{(|s|+|y|+|w|)^{\alpha}}
\le\frac{C}{R^{3\alpha}}\|f\|_{L^{\infty}(Q_{R})}.
\]
Define $\underset{Q}{\text{Osc}} f$ as in Lemma
\ref{Lemma : decreasing oscillation Q_2r} and
\[
\varphi(r):= r^{-\alpha_{0}} \underset{Q_{r}}{\text{Ocs}} f,
\]
where $\alpha_{0} >0$ can be chosen later. By Lemma
\ref{Lemma : decreasing oscillation Q_2r},
\begin{equation}
\label{Eq : decreasing oscillation}\underset{Q_{r/16}}{\text{Ocs}} f \le\left(
1- \frac{\lambda}{2}\right)  \underset{Q_{r}}{\text{Ocs}} f.
\end{equation}
Choose $\alpha_{0}$ such that $16^{\alpha_{0}}\left(  1- \frac{\lambda}%
{2}\right)  <1$. Then by \eqref{Eq : decreasing oscillation},
\[%
\begin{split}
\varphi\left(  \frac{r}{16}\right)   &  = r^{-\alpha_{0}}16^{\alpha_{0}%
}\underset{Q_{r/4}(t,x,v)}{\text{Ocs}} f\\
&  \le16^{\alpha_{0}}\left(  1- \frac{\lambda}{2}\right)  r^{-\alpha_{0}%
}\underset{Q_{r}}{\text{Ocs}} f\\
&  <\varphi(r).
\end{split}
\]
Therefore, we have
\begin{equation}
\label{Eq : Holder 0}%
\begin{split}
\sup_{0<r\le R/16} \varphi(r)  &  \le\sup_{\frac{R}{16}<r\le R}\varphi(r)\\
&  \le2 \frac{16^{\alpha_{0}}}{R^{\alpha_{0}}}\sup_{(t,x,v) \in Q_{R}%
}|f(t,x,v)|.
\end{split}
\end{equation}
If $(t,x,v)\in\partial Q_{r}$ then $|t| + |x| + |v| \ge r^{3}$. Therefore, for
$3 \alpha= \alpha_{0}$ and $r \le R/16$, by \eqref{Eq : Holder 0}
\begin{equation}
\label{Eq : Holder 1}%
\begin{split}
\sup_{(s,y,w)\in Q_{R/16}} \frac{|f(s,y,w)-f(0,0,0)|}{(|s|+|y|+|w|)^{\alpha}}
&  = \sup_{(s,y,w)\in\partial Q_{r}, r\in(0, R/16)} \frac{|f(s,y,w)-f(0,0,0)|}%
{(|s|+|y|+|w|)^{\alpha}}\\
&  \le\sup_{(s,y,w)\in Q_{r}, r\in(0, R/16)} \frac{|f(s,y,w)-f(0,0,0)|}%
{r^{\alpha_{0}}}\\
&  \le\sup_{r \in(0,R/16)} \varphi(r)\\
&  \le\frac{C}{R^{\alpha_{0}}} \sup_{(t,x,v) \in Q_{R}}|f(t,x,v)|.
\end{split}
\end{equation}

Now we consider the general case. For any $(t_{*},x_{*},v_{*}) \in
Q_{R/32}(t_{0}, x_{0}, 0)$, define the translated function
\[
F(T,X,V) = f(t,x,v),
\]
\[
T=t-t_{*},\quad X = x-x_{*} - Tv_{*}, \quad V = v-v_{*}.
\]
Then $F$ satisfies,
\[
\partial_{T} F + V \cdot \nabla_{X} F = \nabla_{V} \cdot (\Sigma_{G}(t,x,v) \nabla_{V} F)
+ a_{g}(t,x,v) \cdot\nabla_{V} F.
\]
Therefore, by \eqref{Eq : Holder 1},
\[
\sup_{(s,y,w)\in Q_{R_{1}/16}} \frac{|F(s,y,w)-F(0,0,0)|}{(|s|+|y|+|w|)^{\alpha
}} \le\frac{C}{{R_{1}}^{\alpha_{0}}} \sup_{(t,x,v) \in Q_{R_{1}}}|F(t,x,v)|
\]
for every $R_{1} < 1$. Since $|v_{*}| \le R/128$,
\[
(t,x,v) \in Q_{R/64}(t_{*},x_{*},v_{*}) \text{ implies } (T,X,V) \in Q_{R/32}%
\]
and
\[
(T,X,V) \in Q_{R/2} \text{ implies } (t,x,v) \in Q_{R}(t_{*},x_{*},v_{*}).
\]
Therefore, by \eqref{Eq : Holder 1}
\[%
\begin{split}
&  \sup_{(t,x,v)\in Q_{R/64}(t_{*},x_{*},v_{*})}\frac{|f(t,x,v)-f(t_{*}%
,x_{*},v_{*})|}{(|t-t_{*}|+|x-x_{*}|+|v-v_{*}|)^{\alpha}}\\
&  \quad\le(1+|v_{*}|)^{\alpha} \sup_{(t,x,v)\in Q_{R/64}(t_{*},x_{*},v_{*}%
)}\frac{|f(t,x,v)-f(t_{*},x_{*},v_{*})|}{((1+|v_{*}|)|t-t_{*}|+|x-x_{*}%
|+|v-v_{*}|)^{\alpha}}\\
&  \quad\le C \sup_{(T,X,V)\in Q_{R/32}} \frac{|F(T,X,V)-F(0,0,0)|}%
{((1+|v_{*}|)|T|+|X+v_{*} T|+|V|)^{\alpha}}\\
&  \quad\le C \sup_{(T,X,V)\in Q_{R/32}} \frac{|F(T,X,V)-F(0,0,0)|}%
{(|T|+|X|+|V|)^{\alpha}}\\
&  \quad\le\frac{C}{{R}^{\alpha_{0}}} \sup_{(T,X,V) \in Q_{R/2}}|F(T,X,V)|\\
&  \quad\le\frac{C}{{R}^{\alpha_{0}}} \sup_{(t,x,v) \in Q_{R}}|f(t.x.v)|.
\end{split}
\]
So the proof is complete.
\end{proof}

\subsection{Global H\"older estimate} % (fold)
\label{sub:global_h"older_estimate}
In this subsection, we will derive a
H\"older continuity for the solution of \eqref{linear landau}. Let $f(t,x,v)$
be a weak solution of \eqref{linear landau} in the sense of Definition \ref{Def : linear weak sol}. Then 
\[
\tilde f(t,x,v) :=
\begin{cases}
f(t,x,v), & \text{ if }t\ge0,\\
f_{0}(x,v), & \text{ if }-1 \le t< 0,
\end{cases}
\]
satisfies
\[
\tilde f_{t} + v \cdot\nabla_{x} \tilde f - \bar A_{g} \tilde f = \tilde
S(t,x,v),
\]
where
$\bar A_g$ and $\bar K_g$ is defined as in \eqref{Eq : bar A}, and \eqref{Eq : bar K},
\[
\tilde S(t,x,v) =
\begin{cases}
(v\cdot\nabla_{x} - \bar A_{f_{0}})f_{0}(x,v), & \text{ if }t\le0,\\
\bar K_{g} f(t,x,v), & \text{ if }t> 0.
\end{cases}
\]
Since $U(t,s)$ is the solution operator of \eqref{bar Af}. Then $f$
satisfies
\[
f(t) = U(t, -1)f_{0} + \int_{-1}^{t} U(t, s)\tilde S(s)ds.
\]
Fisrt, we will obtain a uniform H\"older continuity of $U(t,s)f$. 
Finally, we will derive a uniform H\"older continuity of $f(t)$.

\bigskip
% subsection global_h"older_estimate (end)

As a starting point, we introduce a technical lemma to obtain a uniform H\"older continuity of $U(t,s)f$.
\begin{lemma}
\label{Lemma : change of variables} Let $(t_{*}, x_{*}, v_{*})\in
\mathbb{R}_{+} \times\mathbb{R}^{3} \times\mathbb{R}^{3}$, $N -1/2 \le |v_{*}| \le N+1/2$,
$m>9$ and $O$ be an orthonormal matrix. Define
\begin{equation}
\label{Eq : dilation matrix}D :=
\begin{bmatrix}
(1+|v_{*}|)^{-3/2} & 0 & 0\\
0 & (1+|v_{*}|)^{-1/2} & 0\\
0 & 0 & (1+|v_{*}|)^{-1/2}%
\end{bmatrix}
,
\end{equation}
\begin{align*}
X := D^{-1}O^{T}(x-v_{*}(t-t_{*})), \quad X_{*} := D^{-1}O^{T}x_{*}, \quad V := D^{-1}O^{T}(v-v_{*}),\\
r_{0} := (2+N)^{-m}, \quad r_{1} := (2+N)^{-\frac{2m}{3}+\frac{5}{6}}, \quad r_{2} := (2+N)^{-\frac{4m}{9}+ \frac{13}{18}}.
\end{align*}

Then if $(t,x,v)\in Q_{r_{0}}(t_{*}, x_{*}, v_{*})$, then $(t,X,V) \in
Q_{r_{1}}(t_{*}, X_{*}, 0)$. Moreover, if $(t,X,V) \in Q_{128 r_{1}}%
(t,X_{*},0)$, then $(t,x,v) \in Q_{128r_{2}}(t_{*}, x_{*}, v_{*})$.
\end{lemma}
\begin{proof}
	If $(t,x,v) \in Q_{r_{0}}(t_{*}, x_{*}, v_{*})$, then
	\begin{equation*}
		|t-t_{*}| \le r_{0}^2 \le r_{1}^2,
	\end{equation*}
	\begin{equation*}
		\begin{split}
			|X-X_{*}| &= |D^{-1}O^{T}(x-x_{*} - v_{*}(t-t_{*}))|\\
			&\le (2 + N)^{3/2}(r_{0}^3 + Nr_{0}^2)\\
			&\le (2+N)^{3/2}\left((2+N)^{-3m} + N(2+N)^{-2m} \right)\\
			&\le (2+N)^{3/2}(2+N)^{1-2m}\\
			&\le r_{1}^{3},
		\end{split}
	\end{equation*}
	and
	\begin{equation*}
		\begin{split}
			|V| & = |D^{-1}O^{T}(v-v_{*})|\\
			&\le (2+N)^{3/2}r_{0}\\
			&\le r_{1}.
		\end{split}
	\end{equation*}
	Conversely, if $(t,X,V) \in Q_{128r_{1}}(t,X_{*},0)$, then
	\begin{equation*}
		|t-t_{*}| \le (128 r_{1})^2 \le (128 r_{2})^2
	\end{equation*}
	and
	\begin{equation*}
		|v-v*| = |ODV| \le (1/2+N)^{-1/2}r_{1} \le 128 r_{2}.
	\end{equation*}
	Since $128r_{1}\le 1$ and $(1/2 +N)^{-1/2}(1+N) \le 128(2+N)^{1/2}$, we have
	\begin{equation*}
		\begin{split}
			|x-x_{*}| &= |OD(X-X_{*}) + v_{*}(t-t_{*})|\\
			&\le (1/2+N)^{-1/2}(r_{1}^3 + N r_{1}^2)\\
			&\le (1/2+N)^{-1/2}(1+N)r_{1}^2\\
			&\le (128)^3(2+N)^{\frac{1}{2}-\frac{4m}{3} + \frac{5}{3}}\\
			&\le (128 r_{2})^3.
		\end{split}
	\end{equation*}
	So the proof is complete.
\end{proof}

\begin{lemma}
[Uniform H\"older for \eqref{bar Af}]%
\label{Lemma : Uniform Holder for tilde Af} 
Assume \eqref{Eq : condition for g}.
Let $f$ be a solution of
\eqref{bar Af} in $Q_{1}(t_{0},x_{0},v_{0})$. Then there exist $\vartheta
>0$, $\vartheta_{0}>0$, $C_{\vartheta}$, and $\alpha\in(0,1)$ depending only on dimension such that
\begin{equation}
\label{Eq : Holder tilde Af}
	\sup_{(t,x,v),(t^{\prime},x^{\prime},v^{\prime}) \in Q_{1}(t_{0},x_{0},v_{0})} \frac{|f(t,x,v)-f(t^{\prime},x^{\prime},v^{\prime})|}{(|t-t^{\prime}| + |x-x^{\prime}| +|v-v^{\prime}|)^{\alpha}}\le C \|f\|_{\infty, \vartheta}\le C_{\vartheta} \|f_{0}\|_{\infty, \vartheta + \vartheta_{0}}.
\end{equation}

\end{lemma}

\begin{proof}
By the integration by parts,
\[%
\begin{split}
a_{g}\cdot\nabla_{v} f  &  = -\left\{  \phi^{ij}*[v_{i} \mu^{1/2}g]\right\}
\partial_{j} f - \left\{  \phi^{ij}*[\mu^{1/2}\partial_{j} g]\right\}
\partial_{i} f\\
&  =-\left(  \left\{  \phi^{ij}*[v_{i} \mu^{1/2}g]\right\}  + \left\{
\phi^{ij}*[\mu^{1/2}\partial_{i} g]\right\}  \right)  \partial_{j} f\\
&  = -\left(  2\left\{  \phi^{ij}*[v_{i} \mu^{1/2}g]\right\}  + \left\{
\phi^{ij}*[\partial_{i}(\mu^{1/2} g)]\right\}  \right)  \partial_{j} f\\
&  = -\left(  2\left\{  \phi^{ij}*[v_{i} \mu^{1/2}g]\right\}  + \left\{
\partial_{i} \phi^{ij}*[\mu^{1/2} g]\right\}  \right)  \partial_{j} f\\
&  = -2 v\cdot(\sigma_{\sqrt\mu g} \nabla_{v} f) - \partial_{i} \sigma
_{\sqrt\mu g}^{ij} \partial_{j} f.
\end{split}
\]

Let $N:= |v_{0}|$.

To obtain \eqref{Eq : Holder tilde Af}, we split the proof in two cases;
$|(t,x,v)-(t',x',v')| \le (2+N)^{-3m}$ or $|(t,x,v)-(t',x',v')| > (2+N)^{-3m}$
for some $m>0$ to be determined later.
For the first case, we will consider a new center $(t_{*}, x_{*}, v_{*})\in Q_{1}$, such that $(t,x,v), (t',x',v') \in Q_{(2+N)^{-m}}(t_{*},x_{*},v_{*})$. Note that $N-1/2 \le |v_{*}| \le N+1/2$.

Therefore, it is enough to prove that for every $(t_{*},x_{*},v_{*})\in Q_{1}(t_{0},x_{0},v_{0}%
)$,
\begin{equation}
\sup_{(t,x,v),(t^{\prime},x^{\prime},v^{\prime})\in Q_{(2+N)^{-m}}(t_{*}%
,x_{*},v_{*})\cap Q_{1}(t_{0},x_{0},v_{0})}\frac{|f(t,x,v)-f(t^{\prime
},x^{\prime},v^{\prime})|}{(|t-t^{\prime}|+|x-x^{\prime}|+|v-v^{\prime}|)^{\alpha}%
}\leq C\Vert f\Vert_{\infty,\vartheta}\label{Eq : Holder near}%
\end{equation}
and
\begin{equation}
\sup_{\substack{(t,x,v),(t^{\prime},x^{\prime},v^{\prime})\in Q_{1}%
(t_{0},x_{0},v_{0})\\|t-t^{\prime}|+|x-x^{\prime}|+|v-v^{\prime}| > (2+N)^{-3m}}%
}\frac{|f(t,x,v)-f(t^{\prime},x^{\prime},v^{\prime})|}{(|t-t^{\prime
}|+|x-x^{\prime}|+|v-v^{\prime}|)^{\alpha}}\leq C\Vert f\Vert_{\infty,\vartheta
}.\label{Eq : Holder far}%
\end{equation}

We first focus on \eqref{Eq : Holder near}. Consider the the following
translation
\[
\bar{f}(t,y,w):=f(t,x,v),
\]
where $x=y+v_{*}(t-t_{*})$, $v=v_{*}+w$. Then it is easy to check that
$\bar{f}$ satisfies
\[
\partial_{t}\bar{f}+w\cdot\nabla_{y}\bar{f}=\nabla_{w}\cdot(\bar{\sigma}%
_{G}\nabla_{w}\bar{f})+(v_{*}+w)\cdot(\bar{\sigma}_{\sqrt{\mu}g}\nabla_{w}%
\bar{f})-\sum_{ij}\partial_{i}\bar{\sigma}_{\sqrt{\mu}g}^{ij}\partial_{j}%
\bar{f},
\]
where $\bar{\sigma}_{G}(t,y,w):=\sigma_{G}(t,x,v)$, $\bar{\sigma}_{\mu
}(t,y,w):=\sigma_{\mu}(t,x,v)$, and $\bar{\sigma}_{\sqrt{\mu}g}(t,y,w):=\sigma
_{\sqrt{\mu}g}(t,x,v)$. Let $O$ be an orthonormal constant matrix which will
be determined later. Next consider
\[
\tilde{f}(t,\xi,\nu):=\bar{f}(t,y,w),
\]
where $y=O\xi$, $w=O\nu$. Then we have
\[
\partial_{t}\bar{f}(t,y,w)=\partial_{t}\tilde{f}(t,\xi,\nu).
\]

\begin{equation}%
\begin{split}
w\cdot\nabla_{y}\bar{f}(t,y,w) &  =\sum_{i}w_{i}\partial_{y_{i}}(\tilde
{f}(t,\xi,\nu))\\
&  =\sum_{i,k}w_{i}\partial_{\xi_{k}}\tilde{f}(t,\xi,\nu)\frac{\partial\xi
_{k}}{\partial{y_{i}}}\\
&  =\sum_{i,k}O_{ik}w_{i}\partial_{\xi_{k}}\tilde{f}(t,\xi,\nu)\\
&  =\sum_{k}(O^{T}w)_{k}\cdot\partial_{\xi_{k}}\tilde{f}(t,\xi,\nu)\\
&  =\nu\cdot\nabla_{\xi}\tilde{f}(t,\xi,\nu).
\end{split}
\label{Eq : transport term}%
\end{equation}
where $O_{ik}$ is the $i,k$ component of $O$. We use the following formula to
derive the third equality in \eqref{Eq : transport term}
\[
\frac{\partial\xi_{k}}{\partial{y_{i}}}=\frac{\partial\sum_{l}O_{lk}y_{l}%
}{\partial{y_{i}}}=O_{ik}.
\]

Similarly,
\begin{equation}
\label{Eq : sigma Dv f}%
\begin{split}
\bar\sigma_{G}(t,y,w) \nabla_{w} \bar f(t,y,w)  & = \sum_{j}\bar\sigma
_{G}^{ij}(t,y,w)\partial_{w_{j}} \tilde f(t,\xi,\nu)\\
& = \sum_{j,k} \bar\sigma_{G}^{ij}(t,y,w) O_{jk}\partial_{\nu_{k}} \tilde
f(t,\xi,\nu).
\end{split}
\end{equation}

Define $\tilde{\sigma}_{G}(t,\xi,\nu):=O^{T}\bar{\sigma}_{G}(t,y,w)O$. Note
that
\begin{equation}
O\tilde{\sigma}_{G}(t,\xi,\nu)=\bar{\sigma}_{G}%
(t,y,w)O.\label{O bar sigma = sigma O}%
\end{equation}
Then by \eqref{Eq : sigma Dv f} and \eqref{O bar sigma = sigma O}, we have
\[%
\begin{split}
\nabla_{w}\cdot(\bar{\sigma}_{G}(t,y,w)\nabla_{w}\bar{f}(t,y,w)) &
=\sum_{i,j}\partial_{w_{i}}(\bar{\sigma}_{G}^{ij}(t,y,w)\partial_{w_{j}}%
\tilde{f}(t,\xi,\nu))\\
&  =\sum_{i,j,k}\partial_{w_{i}}(O_{jk}\bar{\sigma}_{G}^{ij}(t,y,w)\partial
_{\nu_{k}}\tilde{f}(t,\xi,\nu))\\
&  =\sum_{i,j,k}\partial_{w_{i}}(O_{ij}\tilde{\sigma}_{G}^{jk}(t,\xi
,\nu)\partial_{\nu_{k}}\tilde{f}(t,\xi,\nu))\\
&  =\sum_{i,j,k,l}O_{il}O_{ij}\partial_{\nu_{*}}(\tilde{\sigma}_{G}^{jk}%
(t,\xi,\nu)\partial_{\nu_{k}}\tilde{f}(t,\xi,\nu))\\
&  =\sum_{k,l}\partial_{\nu_{*}}(\tilde{\sigma}_{G}^{lk}(t,\xi,\nu
)\partial_{\nu_{k}}\tilde{f}(t,\xi,\nu))\\
&  =\nabla_{\nu}\cdot(\tilde{\sigma}_{G}(t,\xi,\nu)\nabla_{\nu}\bar{f}%
(t,\xi,\nu)).
\end{split}
\]
In the next to the last equality, we use $O^{T}O=I$. Similarly, define
$\tilde{\sigma}_{\sqrt{\mu}g}(t,\xi,\nu):=O^{T}\bar{\sigma}_{\sqrt{\mu}%
g}(t,y,w)O$, then
\[%
\begin{split}
v_{*}\cdot(\bar{\sigma}_{\sqrt{\mu}g}(t,y,w)\nabla_{w}\bar{f}(t,y,w)) &
=v_{*}\cdot(O\tilde{\sigma}_{\sqrt{\mu}g}(t,\xi,\nu)\nabla_{\nu}\tilde
{f}(t,\xi,\nu))\\
&  =(O^{T}v_{*})\cdot(\tilde{\sigma}_{\sqrt{\mu}g}(t,\xi,\nu)\nabla_{\nu
}\tilde{f}(t,\xi,\nu))\\
&  =\nu_{*}\cdot(\tilde{\sigma}_{\sqrt{\mu}g}(t,\xi,\nu)\nabla_{\nu}\tilde
{f}(t,\xi,\nu)),
\end{split}
\]%
where $\nu_{*} = O^{T} v_{*}$,
\[%
\begin{split}
w\cdot(\bar{\sigma}_{\sqrt{\mu}g}(t,y,w)\nabla_{w}\bar{f}(t,y,w)) &
=w\cdot(O\tilde{\sigma}_{\sqrt{\mu}g}(t,\xi,\nu)\nabla_{\nu}\tilde{f}%
(t,\xi,\nu))\\
&  =(O^{T}w)\cdot(\tilde{\sigma}_{\sqrt{\mu}g}(t,\xi,\nu)\nabla_{\nu}\tilde
{f}(t,\xi,\nu))\\
&  =\nu\cdot(\tilde{\sigma}_{\sqrt{\mu}g}(t,\xi,\nu)\nabla_{\nu}\tilde
{f}(t,\xi,\nu)),
\end{split}
\]
and
\[%
\begin{split}
\sum_{ij}\partial_{w_{i}}\bar{\sigma}_{\sqrt{\mu}g}^{ij}(t,y,w)\partial
_{w_{j}}\bar{f}(t,y,w) &  =\sum_{ijk}\partial_{w_{i}}\bar{\sigma}_{\sqrt{\mu
}g}^{ij}(t,y,w)\partial_{w_{j}}\tilde{f}(t,\xi,\nu)\\
&  =\sum_{ijk}\partial_{w_{i}}\bar{\sigma}_{\sqrt{\mu}g}^{ij}(t,y,w)O_{jk}%
\partial_{\nu_{k}}\tilde{f}(t,\xi,\nu)\\
&  =\sum_{ijkl}O_{il}\partial_{\nu_{*}}\bar{\sigma}_{\sqrt{\mu}g}%
^{ij}(t,y,w)O_{jk}\partial_{\nu_{k}}\tilde{f}(t,\xi,\nu)\\
&  =\sum_{lk}\partial_{\nu_{*}}\tilde{\sigma}_{\sqrt{\mu}g}^{lk}(t,\xi
,\nu)\partial_{\nu_{k}}\tilde{f}(t,\xi,\nu).
\end{split}
\]
Therefore $\tilde{f}$ satisfies
\[
\partial_{t}\tilde{f}+\nu\cdot\nabla_{\xi}\tilde{f}=\nabla_{\nu}\cdot
(\tilde{\sigma}_{G}\nabla_{\nu}\tilde{f})+(\nu_{*}+\nu)\cdot(\tilde{\sigma
}_{\sqrt{\mu}g}\nabla_{\nu}\tilde{f})-\sum_{lk}\partial_{l}\tilde{\sigma
}_{\sqrt{\mu}g}^{lk}\partial_{k}\tilde{f}.
\]

We split $\tilde\sigma_{G}(t,\xi,\nu)$ in three parts.
\[%
\begin{split}
\tilde\sigma_{G}(t,\xi,\nu)  & = O^{T} \bar\sigma_{\mu}(0) O\\
& \quad+ O^{T}\left( \bar\sigma_{\mu}(w) - \bar\sigma_{\mu}(0)\right) O\\
& \quad+ O^{T} \bar\sigma_{\sqrt{\mu}g} (t,y,w) O\\
& = \tilde\sigma_{1} + \tilde\sigma_{2} + \tilde\sigma_{3}.
\end{split}
\]
Choose orthonormal vectors $o_{1}=v_{*}/|v_{*}|$, $o_{2}$, $o_{3}$ and
\[
O :=%
\begin{bmatrix}
o_{1} & o_{2} & o_{3}%
\end{bmatrix}
.
\]
Note that
\[
\nu_{*} = O^{T} v_{*} =
\begin{bmatrix}
|v_{*}|\\
0\\
0
\end{bmatrix}
.
\]
Moreover $\bar\sigma_{\mu}(0)$ has a simple eigenvalue $\lambda_{1}(v_{*})$
associated with the vector $v_{*}$, and a double eigenvalue $\lambda_{2}%
(v_{*})$ associated with $v^{\perp}$. Therefore,
\[
\tilde\sigma_{1} =
\begin{bmatrix}
\lambda_{1}(v_{*}) & 0 & 0\\
0 & \lambda_{2}(v_{*}) & 0\\
0 & 0 & \lambda_{2}(v_{*})
\end{bmatrix}
.
\]
Note that $\lambda_{1}(v_{*})$ and $\lambda_{2}(v_{*})$ satisfy
\begin{equation}
\label{Eq : lambda_i}%
\begin{split}
\frac{1}{C}(1+N)^{-3}  & \le\lambda_{1}(v_{*}) \le C (1+N)^{-3},\\
\frac{1}{C}(1+N)^{-1}  & \le\lambda_{2}(v_{*}) \le C (1+N)^{-1}.
\end{split}
\end{equation}
Since $\partial_{v_{k}} (\sigma_{\mu})^{ij}(v) \le C(1+|v|)^{-2}$, by
the mean value theorem,
\[
|(\bar\sigma_{\mu})^{ij}(w)-(\bar\sigma_{\mu}(0))| \le C(1+N)^{-3
+1}(2+N)^{-m}.
\]
Therefore,
\[
|(\tilde\sigma_{2})^{ij}| \le C(1+N)^{-2}(2+N)^{-m}.
\]
Define
\[
D_{u}(\nu,\nu^{\prime};v):=\nu^{T}\sigma_{u}(v)\nu^{\prime}.
\]
Then we can easily check that
\[
|D_{u}(\nu,\nu^{\prime};v)|\leq|D_{u}(\nu,\nu;v)|^{1/2}|D_{u}(\nu^{\prime}%
,\nu^{\prime};v)|^{1/2}.
\]
Since $|v-v_{*}|<(2+N)^{-m}$ and $o_{1}=v_{*}/|v_{*}|$, we have
\begin{equation}%
\begin{split}
|(I-P_{v})o_{1}| &  =\frac{|(I-P_{v})v_{*}|}{|v_{*}|}\\
&  =\frac{|-v+v_{*}+v-P_{v}v_{*}|}{|v_{*}|}\\
&  \leq\frac{|v-v_{*}|-|v-P_{v}v_{*}|}{|v_{*}|}\\
&  =\frac{|v-v_{*}|-|P_{v}(v-v_{*})|}{|v_{*}|}\\
&  =2\frac{|v-v_{*}|}{|v_{*}|}\leq C(2+N)^{-m}.
\end{split}
\label{Eq : (I-P_v)o_1}%
\end{equation}
Note that
\begin{equation}
(\tilde{\sigma}_{3})^{ij}=o_{i}^{T}\sigma_{\sqrt{\mu}g}(v)o_{j}%
.\label{Eq : tilde sigma ij}%
\end{equation}
Therefore, by \eqref{Eq : D_(mu^1/2 g)},
\begin{equation}%
\begin{split}
|(\tilde{\sigma}_{3})^{11}| &  =|D_{\sqrt{\mu}g}(o_{1};v)|\\
&  \leq C\Vert g\Vert_{\infty}\left(  (1+|v|)^{-3}|P_{v}o_{1}%
|^{2}+(1+|v|)^{-1}|(I-P_{v})o_{1}|^{2}\right)  \\
&  \leq C\Vert g\Vert_{\infty}\left(  (1+N)^{-3}+(1+N)^{-1}(2+N)^{-2m}\right)  \\
&  \leq C\Vert g\Vert_{\infty}(1+N)^{-3},
\end{split}
\end{equation}
and for $(i,j)\in\{(1,2),(1,3),(2,1),(3,1)\}$,
\[%
\begin{split}
|(\tilde{\sigma}_{3})^{ij}| &  \leq|D_{\sqrt{\mu}g}(o_{1};v)|^{1/2}%
|D_{\sqrt{\mu}g}(o_{k};v)|^{1/2}\\
&  \leq C\Vert g\Vert_{\infty}(1+N)^{-3/2}\left(  (1+|v|)^{-3}%
|P_{v}o_{k}|^{2}+(1+|v|)^{-1}|(I-P_{v})o_{k}|^{2}\right)  ^{1/2}\\
&  \leq C\Vert g\Vert_{\infty}(1+N)^{-3/2}(1+N)^{-1/2}\\
&  =C\Vert g\Vert_{\infty}(1+N)^{-2},
\end{split}
\]
where $k=2$ or $3$. Finally, for $i,j=2$ or $3$,
\[%
\begin{split}
|(\tilde{\sigma}_{3})^{ij}| &  \leq|D_{\sqrt{\mu}g}(o_{i};v)|^{1/2}%
|D_{\sqrt{\mu}g}(o_{j};v)|^{1/2}\\
&  =C\Vert g\Vert_{\infty}\left(  (1+|v|)^{-3}|P_{v}o_{i}|^{2}%
+(1+|v|)^{-1}|(I-P_{v})o_{i}|^{2}\right)  ^{1/2}\\
&  \quad\quad\times\left(  (1+|v|)^{-3}|P_{v}o_{j}|^{2}+(1+|v|)^{-1}|(I-P_{v})o_{j}|^{2}\right)  ^{1/2}\\
&  \leq C\Vert g\Vert_{\infty}(1+N)^{-1}.
\end{split}
\]
Finally, consider the dilation matrix $D$ as in \eqref{Eq : dilation matrix}
and the dilated function
\[
F(t,X,V):=\tilde{f}(t,\xi,\nu),
\]
where $\xi=DX$, $\nu=DV$. Then we can easily check that $F$ satisfies
\[
\partial_{t}F+V\cdot\nabla_{X}F=\nabla_{V}\cdot(\Sigma\nabla_{V}F)+(\nu
_{l}+\nu)^{T}D\Sigma_{3}\nabla_{V}F+\sum_{lk}\partial_{V_{l}}\Sigma_{3}%
^{lk}\partial_{V_{k}}F,
\]
where $\Sigma=\Sigma_{1}+\Sigma_{2}+\Sigma_{3}$, $\Sigma_{i}(t,X,V)=D^{-1}%
\tilde{\sigma}_{i}(t,\xi,\nu)D^{-1}$ for $i=1,2,3$. Then by
\eqref{Eq : lambda_i}, we have
\[
\frac{1}{C}\leq(\Sigma_{1})^{ii}\leq C,\quad(\Sigma_{1})^{ij}=0\text{ for
}i\neq j
\]
and
\[%
\begin{split}
|(\Sigma_{2})^{ij}| &  \leq C(1+N)(2+N)^{-m},\\
|(\Sigma_{3})^{ij}| &  \leq C\Vert g\Vert_{\infty}.
\end{split}
\]
Moreover, since $|\nu|\leq(2+N)^{-m}$,
\[%
\begin{split}
|D(\nu_{*}+\nu)| &  \leq|D\nu_{*}|+|D\nu|\\
&  \leq(1+N)^{-1/2}+C(1+N)^{-1/2}(2+N)^{-m},
\end{split}
\]
and since $\partial_{\nu_{k}}\tilde{\sigma}_{3}(\nu)\leq C\Vert g\Vert
_{\infty}(1+N)^{-2}$, we have
\[%
\begin{split}
|\partial_{V_{l}}\Sigma_{3}^{lk}| &  =|d_{l}^{-1}d_{k}^{-1}\partial_{V_{l}%
}\tilde{\sigma}_{3}(\nu)|\\
&  \leq d_{k}^{-1}|\partial_{\nu_{l}}\tilde{\sigma}_{3}(\nu)|\\
&  \leq\Vert g\Vert_{\infty}(1+N)^{-1/2},%
\end{split}
\]
where $d_{k}$'s are the $k$th diagonal element of $D$. Choose $m>4$ such that
$|(\Sigma_{2})^{ij}|,|(\nu_{*}+\nu)^{T}D\Sigma_{3}|\leq\varepsilon\ll1$. If
$\Vert g\Vert_{\infty}\leq\varepsilon$, then any eigenvalue of $\Sigma$ is
bounded above and below uniformly in $N$. Therefore, by Lemma
\ref{Lemma : Holder near v=0}, there exist a constant $C>0$ uniformly in $N$
and a constant $\alpha\in(0,1)$, depending only on dimension such that
\begin{equation}%
\begin{split}
&  \sup_{(t,X,V),(t^{\prime},X^{\prime},V^{\prime})\in Q_{r_{1}}(t_{*}%
,X_{*},0)}\frac{|F(t,X,V)-F(t^{\prime},X^{\prime},V^{\prime})|}{(|t-t^{\prime
}|+|X-X^{\prime}|+|V-V^{\prime}|)^{\alpha}}\\
&  \quad\leq\frac{C}{(r_{1})^{3\alpha}}\Vert F\Vert_{L^{\infty}(Q_{128r_{1}%
}(t_{*},X_{*},0))},%
\end{split}
\label{Eq : Holder F}%
\end{equation}
where $X_{*}=D^{-1}\xi_{*}$, $\xi_{*}=O^{T}y_{l}$, $y_{l}=x_{*}-v_{*}t_{*}$
and $r_{1}$ is defined as in Lemma \ref{Lemma : Holder near v=0}. Note that
\begin{equation}%
\begin{split}
&  \frac{1}{(1+N)^{\alpha}}\sup_{(t,x,v),(t^{\prime},x^{\prime},v^{\prime})\in
Q_{r_{0}}(t_{*},x_{*},v_{*})\cap Q_{1}(t_{0},x_{0},v_{0})}\frac
{|f(t,x,v)-f(t^{\prime},x^{\prime},v^{\prime})|}{(|t-t^{\prime}|+|x-x^{\prime
}|+|v-v^{\prime}|)^{\alpha}}\\
&  \quad\leq\sup_{(t,x,v),(t^{\prime},x^{\prime},v^{\prime})\in Q_{r_{0}%
}(t_{*},x_{*},v_{*})\cap Q_{1}(t_{0},x_{0},v_{0})}\frac{|f(t,x,v)-f(t^{\prime
},x^{\prime},v^{\prime})|}{((2+N)|t-t^{\prime}|+|x-x^{\prime}|+|v-v^{\prime}|)^{
\alpha}},
\end{split}
\label{Eq : Holder near 1}%
\end{equation}
where $r_{0}$ and $r_{1}$ are defined as in Lemma \ref{Lemma : change of variables}.
Moreover, we have
\begin{equation}%
\begin{split}
&  (2+N)|t-t^{\prime}|+|x-x^{\prime}|+|v-v^{\prime}|\\
&  \quad=(2+N)|t-t^{\prime}|+|ODX+v_{*}(t-t_{*})-(ODX^{\prime}+v_{*}%
(t^{\prime}-t_{*}))|+|ODV+v_{*}-(ODV^{\prime}+v_{*})|\\
&  \quad\geq(2+N)|t-t^{\prime}|+|OD(X-X^{\prime})|-|v_{*}||t-t^{\prime
}|+|OD(V-V^{\prime})|\\
&  \quad\geq|t-t^{\prime}|+|OD(X-X^{\prime})|+|OD(V-V^{\prime})|\\
&  \quad\geq(1+N)^{-3/2}(|t-t^{\prime}|+|(X-X^{\prime})|+|(V-V^{\prime})|).
\end{split}
\label{Eq : Holder near 2}%
\end{equation}
By \eqref{Eq : Holder F}, \eqref{Eq : Holder near 1},
\eqref{Eq : Holder near 2}, and Lemma \ref{Lemma : change of variables}, we
have
\begin{equation}%
\begin{split}
&  \sup_{(t,x,v),(t^{\prime},x^{\prime},v^{\prime})\in Q_{r_{0}}(t_{*}%
,x_{*},v_{*})\cap Q_{1}(t_{0},x_{0},v_{0})}\frac{|f(t,x,v)-f(t^{\prime
},x^{\prime},v^{\prime})|}{((2+N)|t-t^{\prime}|+|x-x^{\prime}|+|v-v^{\prime}|)^{
\alpha}}\\
&  \quad\leq\sup_{(t,X,V),(t^{\prime},X^{\prime},V^{\prime})\in Q_{r_{1}%
}(t_{*},X_{*},0)}\frac{|F(t,X,V)-f(t^{\prime},X^{\prime},V^{\prime}%
)|}{(1+N)^{-3 \alpha /2 }(|t-t^{\prime}|+|(X-X^{\prime})|+|(V-V^{\prime}|)^{
\alpha}}.%
\end{split}
\label{Eq : Holder near 3}%
\end{equation}
Combine \eqref{Eq : Holder F}, \eqref{Eq : Holder near 1}, and
\eqref{Eq : Holder near 3}. Then we have
\begin{equation}%
\begin{split}
&  \sup_{(t,x,v),(t^{\prime},x^{\prime},v^{\prime})\in Q_{r_{0}}(t_{*}%
,x_{*},v_{*})\cap Q_{1}(t_{0},x_{0},v_{0})}\frac{|f(t,x,v)-f(t^{\prime
},x^{\prime},v^{\prime})|}{(|t-t^{\prime}|+|x-x^{\prime}|+|v-v^{\prime}|)^{\alpha}%
}\\
&  \quad\leq(2+N)^{\alpha}\sup_{(t,x,v),(t^{\prime},x^{\prime},v^{\prime})\in
Q_{r_{0}}(t_{*},x_{*},v_{*})\cap Q_{1}(t_{0},x_{0},v_{0})}\frac
{|f(t,x,v)-f(t^{\prime},x^{\prime},v^{\prime})|}{((2+N)|t-t^{\prime
}|+|x-x^{\prime}|+|v-v^{\prime}|)^{\alpha}}\\
&  \quad\leq\frac{(2+N)^{\alpha}}{(1+N)^{-3 \alpha /2}}\sup
_{(t,X,V),(t^{\prime},X^{\prime},V^{\prime})\in Q_{r_{1}}(t_{*},X_{*},0)}%
\frac{|F(t,X,V)-f(t^{\prime},X^{\prime},V^{\prime})|}{(|t-t^{\prime
}|+|X-X^{\prime}|+|(V-V^{\prime}|)^{\alpha}}\\
&  \quad\leq\frac{C(2+N)^{\alpha}}{(1+N)^{-3 \alpha /2}(r_{1})^{3\alpha}%
}\Vert F\Vert_{L^{\infty}(Q_{128r_{1}}(t_{*},X_{*},0))}.%
\end{split}
\label{Eq : Holder near 4}%
\end{equation}
By the Lemma \ref{Lemma : change of variables},
\begin{equation}
\Vert F\Vert_{L^{\infty}(Q_{128r_{1}}(t_{*},X_{*},0))}\leq\Vert f\Vert
_{L^{\infty}(Q_{128r_{2}}(t_{*},x_{*},v_{*}))}.\label{Eq : Holder near 5}%
\end{equation}
Choose $m>9$ such that $128r_{2}<1$. Then we have
\begin{equation}%
\begin{split}
\Vert f\Vert_{L^{\infty}(Q_{128r_{2}}(t_{*},x_{*},v_{*}))} &\leq C(1+N)^{-\vartheta}\Vert(1+|v|)^{\vartheta}f\Vert_{L^{\infty
}(Q_{128r_{2}}(t_{*},x_{*},v_{*}))}\\
&  \leq C(1+N)^{-\vartheta}\Vert(1+|v|)^{\vartheta}f\Vert_{L^{\infty}},
\end{split}
\label{Eq : Holder near 5-1}%
\end{equation}
for every $\vartheta>0$. Finally combining \eqref{Eq : Holder near 4} -
\eqref{Eq : Holder near 5-1}, we have
\begin{equation}%
	\begin{split}
		&  \sup_{(t,x,v),(t^{\prime},x^{\prime},v^{\prime})\in Q_{r_{0}}(t_{*},x_{*},v_{*})\cap Q_{1}(t_{0},x_{0},v_{0})}\frac{|f(t,x,v)-f(t^{\prime},x^{\prime},v^{\prime})|}{(|t-t^{\prime}|+|x-x^{\prime}|+|v-v^{\prime}|)^{\alpha}}\\
		&  \quad\leq\frac{C(1+N)^{\alpha-\vartheta}}{(1+N)^{-3 \alpha/2}(r_{1})^{3\alpha}}\Vert(1+|v|)^{\vartheta}f\Vert_{L^{\infty}}\\
		&  \quad\leq C(1+N)^{\alpha-\vartheta+\frac{3 \alpha}{2}+2m\alpha-\frac{5\alpha}{2}}\Vert(1+|v|)^{\vartheta}f\Vert_{L^{\infty}}\\
		&  \quad=C(1+N)^{-\vartheta+ 2m\alpha}\Vert(1+|v|)^{\vartheta}f\Vert_{L^{\infty}},%
	\end{split}
\label{Eq : Holder near 6}%
\end{equation}
where $\vartheta>0$ will be determined later.

To prove \eqref{Eq : Holder far},
\begin{equation}		\label{Eq : Holder far 1}
	\begin{split}
		&\sup_{\substack{(t,x,v),(t',x',v') \in Q_1(t_0,x_0,v_0)\\ |t-t'| + |x-x'| +|v-v'| > (2+N)^{-3m} } }\frac{|f(t,x,v)-f(t',x',v')|}{(|t-t'| + |x-x'| +|v-v'|)^\alpha}\\
		&\quad \le 2 (2+N)^{3\alpha m }\|f\|_{L^\infty(Q_1(t_0,x_0,v_0))}\\
		&\quad \le 2 C(1+N)^{3 \alpha m - \vartheta}\|(1+|v|)^\vartheta f\|_{L^\infty(Q(t_0,x_0,v_0;1))}\\
		&\quad \le 2 C(1+N)^{3 \alpha m - \vartheta}\|(1+|v|)^\vartheta f\|_{L^\infty}.
	\end{split}
\end{equation}
Now choose
\begin{equation}		\label{Eq : theta}
	\vartheta > 3 \alpha m.
\end{equation}
Then from \eqref{Eq : Holder near 6} - \eqref{Eq : theta}, we prove \eqref{Eq : Holder near} and \eqref{Eq : Holder far}.
Therefore, we have
\begin{equation*}
	\sup_{(t,x,v),(t',x',v') \in Q_1(t_0,x_0,v_1)} \frac{|f(t,x,v)-f(t',x',v')|}{(|t-t'| + |x-x'| +|v-v'|)^\alpha} \le C\|f\|_{\infty, \vartheta}.
\end{equation*}
By Theorem \ref{Thm : L^infty estimate tilde M}, we have \eqref{Eq : Holder tilde Af}.
\end{proof}

Now we will prove Theorem \ref{Thm : Holder landau}.
\begin{proof}[Proof of Theorem \ref{Thm : Holder landau}]
	Since $f$ satisfies \eqref{linear landau}, we have
	\begin{equation}		\label{Eq : linear landau for holder}
		f_t + v\cdot \nabla_x f - \bar A_g f = \bar K_g f.
	\end{equation}
	Define
	\begin{equation*}
		\tilde f(t,x,v) = \begin{cases}
			f(t,x,v), &\text{ if }t\ge 0,\\
			f_0(x,v), &\text{ if }-1 \le t< 0.
		\end{cases}
	\end{equation*}
	Consider $\tilde S(t,x,v) = (\partial_t + v\cdot \nabla_x - \bar A_g)\tilde f$.
	Then for $t\le 0$,
	\begin{equation*}
		\begin{split}
			\tilde S(t,x,v) &= (\partial_t + v\cdot \nabla_x - \bar A_g)\tilde f\\
			&=(v\cdot \nabla_x - \bar A_{f_{0}})f_{0}(x,v)\\
			&=-f_{0t},
		\end{split}
	\end{equation*}
	where $f_{0t}$ was defined in Theorem \ref{Thm : main result}.
	Since $f$ is a weak solution of \eqref{linear landau} in the sense of Definition \ref{Def : linear weak sol}, for $t>0$, $\tilde S(t,x,v) = \bar K_g f(t,x,v)$.
	Thus, $\tilde f$ satisfies,
	\begin{equation*}
		\tilde f_t + v \cdot \nabla_x \tilde f - \bar A_g \tilde f = \tilde S(t,x,v).
	\end{equation*}
	Since $U(t,s)$ is the solution operator for $\partial_t + v \cdot \nabla_x - \tilde A_g = 0$.
	Then $\tilde f$ satisfies
	\begin{equation*}
		\tilde f(t) = U(t,-1)f_0 + \int_{-1}^{t} U(t,s) \tilde S(s) ds.
	\end{equation*}
	Let $0<\bar \varepsilon\ll 1$ be given.
	Note that by Lemma \ref{Lemma : Holder near v=0}, there exists $\alpha>0$ such that $U(t,-1)\tilde f(-1)$ is uniformly H\"older continuous on $(0,\infty)\times \T^3 \times \R^3$.

	For every $0 \le t_2 \le t_1$, $|t_1 - t_2| + |x_2 - x_1| + |v_2-v_1| = \bar \varepsilon$,
	\begin{equation*}
	 	\begin{split}
			&|f(t_1 ,x_1,v_1)- f(t_2,x_2,v_2)|\\
			&\le \left| (U(t_1,-1)f_0)(x_1,v_1) -  (U(t_2,-1)f_0)(x_2,v_2)\right|\\
			& \quad + \left|\int_{t_2}^{t_1} (U(t_1,s)\tilde S(s))(x_1,v_1)ds\right|\\
			& \quad + \left|\int_{t_2-\bar \varepsilon^\alpha}^{t_2}\big( (U(t_1,s)\tilde S(s))(x_1,v_1) - (U(t_2,s)\tilde S(s))(x_2,v_2)\big)ds\right|\\
			& \quad + \left|\int_{-1}^{t_2-\bar \varepsilon^\alpha}\big( (U(t_1,s)\tilde S(s))(x_1,v_1) - (U(t_2,s)\tilde S(s))(x_2,v_2)\big)ds\right|\\
			&\le (I) + (II) +(III) + (IV).
	 	\end{split}
	\end{equation*}
	By Lemma \ref{Lemma : Uniform Holder for tilde Af}, there exist $\vartheta$, $l$, $C$ and $C_{\vartheta, l}$ such that
	\begin{equation*}
		(I) \le (C \|f_0\|_{\infty, \vartheta} + C_{\vartheta,l}\|f_0\|_{2,\vartheta+l}) \bar \varepsilon^{\alpha}.
	\end{equation*}
	By Lemma \ref{lemma : Maximum principle}, \eqref{Eq : L^infty estimate linear}, and \eqref{Eq : Kf L^infty norm tilde M}, there exists $l_0$ such that
	\begin{equation*}
		(II)\le \bar \varepsilon \sup_{s>0}\|\tilde K_g f(s)\|_\infty \le C\bar \varepsilon \sup_{s>0} \|f(s)\| \le C \bar \varepsilon(\|f_0\|_\infty+\|f_0\|_{2,l_0}).
	\end{equation*}
	Note that for $s\le 0$,
	\begin{equation}		\label{Eq : tilde S L^infty s<0}
		\|\tilde S(s)\|_{\infty, \vartheta} \le C\|f_{0t}\|_{\infty, \vartheta} ,
	\end{equation}
	\begin{equation}		\label{Eq : tilde S L^2 s<0}
		\|\tilde S(s)\|_{2, \vartheta} \le C\|f_{0t}\|_{2, \vartheta},
	\end{equation}
	and for $s\ge 0$, by Lemma \ref{lemma : Kf estimate} and Theorem \ref{Thm : energy estimate}, \ref{Thm : L^infty estimate} there exists $l_0$ such that
	\begin{equation}		\label{Eq : tilde S L^2 s>0}
		\|\tilde S(s)\|_{\infty, \vartheta} \le C(1+s)^{-2}(\|f_0\|_{2,\vartheta+l_{0}} + \|f_0\|_{\infty,\vartheta+l_{0}})
	\end{equation}
	and
	\begin{equation}		\label{Eq : tilde S L^infty s>0}
		\|\tilde S(s)\|_{2,\vartheta+l} \le C(1+s)^{-2}(\|f_0\|_{2, \vartheta+l_{0}} + \|f_0\|_{\infty, \vartheta+l_{0} }).
	\end{equation}
	Therefore, by \eqref{Eq : tilde S L^infty s<0} and \eqref{Eq : tilde S L^infty s>0}, we have
	\begin{equation*}
		\begin{split}
	 		(III) &\le \bar \varepsilon^\alpha \sup_{s\in (t_2 - \bar \varepsilon^\alpha, t_2)}\|\tilde S(s)\|_{\infty}\\
	 		& \le C\bar \varepsilon^\alpha(\|f_{0t}\|_{\infty}+\|f_0\|_{2, l_{0}} + \|f_0\|_{\infty, l_{0} }).
		\end{split}
	\end{equation*}
	For $-1 \le s \le t_2 - \bar \varepsilon^\alpha$, we have 
	\begin{equation*}
		\begin{split}
			&|t_1  - (t_2 +1 - \bar \varepsilon^\alpha)| + |x_1 - x_2| + |v_1 - v_2|\\
			&\quad = |t_1 - t_2 -1 + \bar \varepsilon^\alpha| + |x_1 - x_2| + |v_1 - v_2|\\
			& \quad \le 1 - \bar \varepsilon^\alpha + |t_1 - t_2| + |x_1 - x_2| + |v_1 - v_2|\\
			& \quad \le 1,
		\end{split}
	\end{equation*}
	\begin{equation*}
		|t_2  - (t_2 + 1 - \bar \varepsilon^\alpha)| + |x_2 - x_2| + |v_2 - v_2| = 1- \bar \varepsilon^\alpha \le 1.
	\end{equation*}
	Therefore,
	\begin{equation*}
		(t_i , x_i, v_i ) \in Q_1(t_2 + 1 - \bar \varepsilon^\alpha, x_2, v_2) \subset (s,\infty) \times \T^3 \times \R^3,
	\end{equation*}
	for each $i = 1,2$.
	Therefore by Lemma \ref{Lemma : Uniform Holder for tilde Af}, there exist $\vartheta$, $l$, $C$ and $C_{\vartheta, l}$ such that
	\begin{equation*}
		\begin{split}
			&(U(t_1,s)\tilde S(s))(x_1,v_1) - (U(t_2,s)\tilde S(s))(x_2,v_2)\\
			&\le (C \|\tilde S(s)\|_{\infty, \vartheta} + C_{\vartheta,l}\|\tilde S(s)\|_{2,\vartheta+l}) \bar \varepsilon^{\alpha},
		\end{split}
	\end{equation*}
	for $-1\le s\le t_{2}- \bar \varepsilon^{\alpha}$.
	Therefore, by Lemma \ref{Lemma : Uniform Holder for tilde Af}, \eqref{Eq : tilde S L^infty s<0}, \eqref{Eq : tilde S L^2 s>0}, and \eqref{Eq : tilde S L^2 s<0}, \eqref{Eq : tilde S L^infty s>0}, we have
	\begin{equation*}
		\begin{split}
			(IV) &\le \left|\int_{-1}^{0}\big( (U(t_1,s)\tilde S(s))(x_1,v_1) - (U(t_2,s)\tilde S(s))(x_2,v_2)\big)ds\right|\\
			&\quad+ \left|\int_{0}^{t_2-\bar \varepsilon^\alpha}\big( (U(t_1,s)\tilde S(s))(x_1,v_1) - (U(t_2,s)\tilde S(s))(x_2,v_2)\big)ds\right|\\
			&\le C \bar \varepsilon^{\alpha} \left( \sup_{s\in [-1,0]} (\|\tilde S(s)\|_{\infty,\vartheta} + \|\tilde S(s)\|_{2, \vartheta+l} + \int_{0}^{\infty} \|\tilde S(s)\|_{\infty, \vartheta} + \|\tilde S(s)\|_{2, \vartheta + l} ds \right)\\
			&\le C \bar \varepsilon^\alpha\left(\|f_{0t}\|_{\infty, \vartheta} + \|f_{0t}\|_{2, \vartheta + l} + (\|f_0\|_{2, \vartheta+l_{0}} + \|f_0\|_{\infty, \vartheta+l_{0} })\int_{0}^{\infty}(1+s)^{-2}ds\right)\\
			&\le C \bar \varepsilon^\alpha\left(\|f_{0t}\|_{\infty, \vartheta} + \|f_{0t}\|_{2, \vartheta + l} + \|f_0\|_{2, \vartheta+l_{0}} + \|f_0\|_{\infty, \vartheta+l_{0} }\right).
		\end{split}
	\end{equation*}
	Now we update $\vartheta$ to  $\vartheta + 2 + \max\{l, l_{0}\}$.
	Then, by Proposition \ref{Prop : L2Linfty}, we have
	\begin{equation*}
		\begin{split}
			|f(t_1 ,x_1,v_1)- f(t_2,x_2,v_2)| & \le (I) + (II) + (III) + (IV)\\
			&\le C \bar \varepsilon^\alpha\left(\|f_{0t}\|_{\infty, \vartheta}  + \|f_0\|_{\infty, \vartheta }\right).
		\end{split}
	\end{equation*}
	Thus we complete the proof.
\end{proof}

% section holder_continuity (end)

\section{Holder estimate and $S_{p}$ bound} % (fold)
\label{sec:holder estimate and sp bound}

% section sketch_of_boundedness_of_ (end)

Let $f$ be a weak solution of \eqref{linear landau} in the sense of Definition \ref{Def : linear weak sol}.
Define
\begin{equation*}
	\bar f(t,x,v) = \begin{cases}
		f(t,x,v), &\text{ if }t\ge 0\\
		f_0(x,v), &\text{ if }-1 \le t< 0.
	\end{cases}
\end{equation*}
Then $\bar f$ satisfies
\begin{equation}			\label{Eq : Sp eqn}
	\begin{split}
		\partial_t \bar f + v \cdot \nabla_x \bar f - \sigma_G^{ij} \partial_{v_i,v_j} \bar f &= \begin{cases}
			- \partial_{v_i} \sigma_G^{ij}\partial_{v_j} f + a_g \cdot \nabla_v f + K_1 f + J_g f, &\text{ if }t\ge 0,\\
			(v\cdot \nabla_x - \sigma_{\mu + \mu^{1/2}f_0}^{ij} \partial_{v_i,v_j}) f_0, &\text{ if }-1 \le t< 0,
		\end{cases}\\
		&= \begin{cases}
			- \partial_{v_i} \sigma_G^{ij}\partial_{v_j} f + a_g \cdot \nabla_v f + K_1 f + J_g f, &\text{ if }t\ge 0\\
			- f_{0t} + \partial_{v_{i}}\sigma_{\mu + \mu^{1/2}f_0}^{ij} \partial_{v_{j}}f_{0} , &\text{ if }-1 \le t< 0,
		\end{cases}
	\end{split}
\end{equation}
where $\sigma_G$ is defined as in \eqref{Eq : sigma_u} with $G = \mu+ \mu^{1/2}g$ and 
\begin{equation}		\label{Eq : K_1}
	\begin{split}
		K_1 f & = - \mu^{-1/2}\partial_i \left\{\mu\left[\phi^{ij}* \left\{\mu^{1/2}[\partial_j f + v_j f] \right\}\right] \right\}\\
		& = 2v_i \mu \left[\phi^{ij}* \left\{\mu^{1/2}[\partial_j f + v_j f] \right\}\right] - \mu^{1/2}\left[\partial_i \phi^{ij}* \left\{\mu^{1/2}[\partial_j f + v_j f] \right\}\right],
	\end{split}
\end{equation}
and
\begin{equation}		\label{Eq : J_g}
	J_g f = - v\cdot \sigma v f -\partial_i\left\{\phi^{ij}*[\mu^{1/2}\partial_j g] \right\}f + \left\{\phi^{ij}*[v_i \mu^{1/2} \partial_j g] \right\}f.
\end{equation}
\begin{lemma}
	For every $\beta>0$ and $p>3$,
	\begin{equation}	\label{Eq : K_1 f}
		\|K_1f\|_{L^p(n \le |v| \le n+1)} \le \frac{C_{p,\beta}}{n^ \beta}\left(\|f\|_{L^p} + \|D_v f\|_{L^p}\right),
	\end{equation}
	where $K_1$ is defined as in \eqref{Eq : K_1}
\end{lemma}
\begin{proof}
	Since $K_1$ is defined as in \eqref{Eq : K_1},
	It is enough to show that
	\begin{equation*}
		\left\|\mu(v)\int_{\R^3}|v-v'|^\vartheta  \mu(v')h(v') dv'\right\|_{L^{p}}\le \|h\|_{L^p}.
	\end{equation*}
	By the H\"older inequality,
	\begin{equation*}
		\begin{split}
			&\int_{(0,\infty)\times \T^3\times \R^3} \mu^p(v)\left(\int_{\R^3}|v-v'|^\vartheta  \mu(v') h(v') dv'\right)^p dvdxdt\\
			&\le \int_{(0,\infty)\times \T^3\times \R^3} \mu^p(v)\left(\int_{\R^3}|v-v'|^{\vartheta p'} \mu^{p'}(v')dv'\right)^{p/p'} \left(\int_{\R^3}h^p(v') dv'\right) dvdxdt\\
			&\le \int_{\R^3} \mu^p(v) (1+|v|)^{\vartheta p} dv \int_{(0,\infty)\times \T^3} \left(\int_{\R^3}h^p(v') dv'\right) dxdt\\
			&=C_p \|h\|_{L^p}^p.
		\end{split}
	\end{equation*}
\end{proof}

Clearly, we have
\begin{equation}	\label{Eq : J_g f}
	\|J_g f\|_{L^p(n \le |v| \le n+1)} \le (C+\|g\|_\infty)(\|f\|_{L^p(n \le |v| \le n+1)}) \le Cn^{- \beta} (\|(1+|v|)^{\beta}f\|_{L^p(n \le |v| \le n+1)}).
\end{equation}

\begin{theorem}[Theorem 1.5 in \cite{polidoro1998sobolev}] \label{Thm : polidoro1998 1.5}
	Let $\Omega$ be a bounded open set in $\R^7$ and let $f$ be a strong solution in $\Omega$ to the equation
	\begin{equation*}
		\sum_{i,j=1}^3 \sigma^{ij}(t,x,v)\partial_{v_i,v_j} f + Yf = h,
	\end{equation*}
	where $Y = -\partial_t - v\cdot \nabla_x$.
	Suppose that $\sigma$ is uniformly elliptic,
	\begin{equation}		\label{Eq : sigma Holder}
		\|\sigma^{ij}\|_{C^{\alpha}(\Omega)} \le  C,
	\end{equation}
	and $f, h\in L^p$.

	Then $\partial_{v_i,v_j} f \in L^p_{\text{loc}}$, $Yf \in L^p_{\text{loc}}$ and for every open set $\Omega'\subset \subset \Omega$ there exists a positive constant $c_1$ depending only on $p, \Omega', \Omega, \alpha, C$ and elliptic constant of $\sigma$ such that
	\begin{equation*}
		\begin{split}
			\| \partial_{v_i,v_j} f\|_{L^p(\Omega')} &\le c_1 (\|f\|_{L^p(\Omega)} + \|h\|_{L^p(\Omega)}),\\
			\| Y f\|_{L^p(\Omega')} &\le c_1 (\|f\|_{L^p(\Omega)} + \|h\|_{L^p(\Omega)}).
		\end{split}
	\end{equation*}
\end{theorem}
\begin{remark}		\label{Rmk : properties of c_1}
	Especially, from \cite{bramanti1996lpestimates}
	\begin{equation*}
		c_1 = c_2(\lambda_1, \lambda_2)c_3(\text{dist}(\Omega',\Omega),\alpha,C,p),
	\end{equation*}
	where $\lambda_1$, $\lambda_2$ are the smallest and the largest eigenvalue respectively.
	More precisely, there exist $C>0$ and $\alpha'>0$ such that
	\begin{equation*}
		c_2(\lambda_1,\lambda_2) := \max_{|x|^2+|v|^2+|t|^2 =1} |\Gamma^0_{v_i,v_j}(x,v,t)| \le C(\lambda_1^{-1} + \lambda_2)^{\alpha'},
	\end{equation*}
	where $\Gamma^0(\zeta^{-1}\circ z)$ is a fundamental solution of
	\begin{equation}	\label{Eq : freezed FP eqn}
		\sum_{i,j=1}^3 \sigma^{ij}(\tau,\xi,\nu)\partial_{v_i,v_j} f + Yf = 0
	\end{equation}
	and $\circ$ is a Lie group operation corresponding to \eqref{Eq : freezed FP eqn} for some $C>0$.	
\end{remark}
\begin{remark}		\label{Rmk : periodic extension}
	Since $f(t,\cdot,v)$ is a periodic function on $\T^3$, we extend it to a periodic function on $(3\T)^3$.
	Note that
	\begin{equation*}
		\|f\|_{L^p(3\T^3)}^p = 27\|f\|_{L^p(\T^3)}^p.
	\end{equation*}
\end{remark}

Define
\begin{equation*}
	\|f\|_{S^p(\Omega)} := \|f\|_{L^p(\Omega)} + \|D_v f\|_{L^p(\Omega)} + \|D_{vv}f\|_{L^p(\Omega)} + \|Yf\|_{L^p(\Omega)},
\end{equation*}
where $Y = -\partial_t -v\cdot \nabla_x$.

\begin{lemma}	\label{Lemma : S^p estimate}
	Assume \eqref{Eq : condition for g}.
	Let $f$ be a weak solution of \eqref{initial}, \eqref{conservation laws}, and \eqref{linear landau} in the sense of Definition \ref{Def : linear weak sol}.
	Suppose that $g$ satisfies $\|g\|_{C^{\alpha}((0,\infty)\times \T^3 \times \R^3)} \le C'$ for some $0<\alpha<1$ and $C>0$.
	Then there exist $\vartheta>0$, $p>3$, $C_{\vartheta,\alpha, C', p}$ such that 
	\begin{equation}	\label{Eq : S^p estimate}
		\|f\|_{S^p((0,\infty)\times \T^3 \times \R^3)} \le C\left(\|f_{0t}\|_{\infty, \vartheta} + \|D_v f_0\|_{\infty, \vartheta} + \|f_{0}\|_{\infty, \vartheta}\right).
	\end{equation}
\end{lemma}
\begin{proof}
	Since $\|g\|_{C^{\alpha}((0,\infty)\times \T^3 \times \R^3)} \le C$, $\sigma_G$ satisfies \eqref{Eq : sigma Holder}.
	Apply \eqref{Eq : Sp eqn} to Theorem \ref{Thm : polidoro1998 1.5} with $\Omega'=\Omega'_n$ and $\Omega = \Omega_n$, where
	\begin{equation*}
		\Omega' = \Omega'_n:=\{t\ge 0, x\in \T^3, n \le |v| \le n+1\},
	\end{equation*}
	\begin{equation*}
		\Omega = \Omega_n:=\{t\ge -1, x\in 3\T^3, n-1/2 \le |v| \le n+3/2\}.
	\end{equation*}
	Let $\sigma_G$, $K_1$, and $J_g$ are defined as in \eqref{Eq : sigma_u}, \eqref{Eq : K_1}, \eqref{Eq : J_g}.
	Then, we have
	\begin{equation*}
		\begin{split}
			\|f\|_{S^p(\Omega'_n)}^p &\le (c_1)^p \Big(\|\partial_{v_i}\sigma_G^{ij}\partial_{v_j} f\|_{L^p(\Omega'_n)}^p + \| a_g \cdot \nabla_v f\|_{L^p(\Omega'_n)}^p + \|K_1 f + J_g f\|_{L^p(\Omega'_n)}^p\\
			&\quad \quad \quad + \|(v\cdot \nabla_x - \sigma_G^{ij} \partial_{v_i,v_j}) f_0\|_{L^p\big((-1,0)\times 3\T^3 \times B(0;n-1/2,n+3/2)\big)}^p\Big),
		\end{split}
	\end{equation*}
	where $B(z;r_1,r_2) := B(z;r_2)\setminus B(z;r_1)$.
	By Lemma \ref{Lemma : Guo2002 lemma3} and Remark \ref{Rmk : properties of c_1},
	\begin{equation*}
		c_2(\lambda_1, \lambda_2) = C n^a
	\end{equation*}
	for some $a>0$ and
	\begin{equation*}
		c_3(\text{dist}(\Omega', \Omega), \alpha, C, p) = C_{\alpha,C,p}.
	\end{equation*}
	Therefore, 
	\begin{equation*}
		c_1 = C_{\alpha,C,p} n^a
	\end{equation*}
	for some $a>0$.
	Let 
	\begin{equation*}
		\tilde \Omega'_n = \{t\ge 0, x\in \T^3, n-1/2\le |v| \le n+3/2\}.
	\end{equation*}
	Then by Remark \ref{Rmk : periodic extension},
	\begin{equation}	\label{Eq : bdd domain estimate}
		\begin{split}
			\|f\|_{S^p(\Omega'_n)}^p &\le (C_{\alpha,C,p}n^a)^p \Big(\|\partial_{v_i}\sigma_G^{ij}\partial_{v_j} f\|_{L^p(\tilde \Omega'_n)}^p + \| a_g \cdot \nabla_v f\|_{L^p(\tilde \Omega'_n)}^p + \|K_1 f + J_g f\|_{L^p(\tilde \Omega'_n)}^p\\
			&\quad \quad \quad + \|(v\cdot \nabla_x - \sigma_G^{ij} \partial_{v_i,v_j}) f_0\|_{L^p\big((-1,0)\times \T^3 \times B(0;n-1/2,n+3/2)\big)}^p\Big),
		\end{split}
	\end{equation}
	where $\tilde \Omega'_n = \{t\ge 0, x\in \T^3, n-1/2\le |v| \le n+3/2\}$.
	Note that by the standard interpolation, we have
	\begin{equation}		\label{Eq : v interpolation}
		\|D_v f\|_{L^p(\tilde \Omega'_n)}^p \le \varepsilon' \|D_{vv} f\|_{L^p(\tilde \Omega'_n)}^p + \frac{C}{\varepsilon'}\|f\|_{L^p(\tilde \Omega'_n)}^p.
	\end{equation}

	Let $\beta > 2a+4$ and $\varepsilon' = \varepsilon_0$ which can be determined later.
	Then by \eqref{Eq : K_1 f}, \eqref{Eq : J_g f}, and \eqref{Eq : v interpolation}, we have
	\begin{equation}	\label{Eq : sum Kf}
		\begin{split}
			&\sum (C_{\alpha,C,p}n^a)^p \|K_1 f + J_g f\|_{L^p(\tilde \Omega'_n)}^p \\
			&\le \sum (C_{\beta, \alpha, C, p})^p n^{p(a - \beta)}\left(\|f\|_{L^p\big((0,\infty)\times \T^3 \times \R^3)\big)}^p + \|D_v f\|_{L^p\big((0,\infty)\times \T^3 \times \R^3)\big)}^p + \|(1+|v|)^\beta f\|_{L^p(\tilde \Omega'_n)}^p\right)\\
			&\le (C_{\beta, \alpha, C, p})^p\left(\|(1+|v|)^{\beta} f\|_{L^p\big((0,\infty)\times \T^3 \times \R^3)\big)}^p + \|D_v f\|_{L^p\big((0,\infty)\times \T^3 \times \R^3)\big)}^p\right)\\
			&\le(C_{\beta, \alpha, C, p, \varepsilon_0})^p\|(1+|v|)^{\beta} f\|_{L^p\big((0,\infty)\times \T^3 \times \R^3)\big)}^p + \varepsilon_0 \|D_{vv} f\|_{L^p\big((0,\infty)\times \T^3 \times \R^3)\big)}^p.
		\end{split}
	\end{equation}
	for some $\varepsilon_0$.
	Similarly,
	\begin{equation}	\label{Eq : sum initial part}
		\begin{split}
			&\sum (C_{\alpha,C,p}n^a)^p \|(v\cdot \nabla_x - \sigma_G^{ij} \partial_{v_i,v_j})f_0(x,v)\|_{L^p\big((-1,0)\times \T^3 \times B(0;n-1/2,n+3/2)\big)}^p \\
			&\le \sum (C_{\alpha,C,p})^pn^{p(a- \beta)} \|(1+|v|)^\beta (v\cdot \nabla_x - \sigma_G^{ij} \partial_{v_i,v_j})f_0(x,v)\|_{L^p\big(\T^3 \times B(0;n-1/2,n+3/2)\big)}^p\\
			&\le (C_{\alpha,C,p})^p\|(v\cdot \nabla_x - \sigma_G^{ij} \partial_{v_i,v_j})f_0\|_{p, \beta}^p.
		\end{split}
	\end{equation}
	Choose $\varepsilon' = \varepsilon n^{- p (a+2)}$ small enough. 
	Since $\|\partial_{v_i} \sigma_G^{ij}\|_\infty<C$, $\|a_g\|_\infty \le C$, we have
	\begin{equation}		\label{Eq : sum Dv f}
		\begin{split}
			& \sum (C_{\alpha,C,p}n^a)^p \left(\|\partial_{v_i} \sigma_G^{ij} \partial_{v_j}f)\|_{L^p(\tilde \Omega'_n)}^p + \|a_g \cdot \nabla_v f\|_{L^p(\tilde \Omega'_n)}^p\right) \\
			& \le \sum  (C_{\alpha,C,p})^p \left( \varepsilon n^{-2p} \|D_{vv} f\|_{L^p(\tilde \Omega'_n)}^p + \varepsilon^{-1} n^{2p (a + 1)}\|f\|_{L^p(\tilde \Omega'_n)}^p\right)\\
			& \le \sum (C_{\alpha,C,p})^p \left( \varepsilon n^{-2p} \|D_{vv} f\|_{L^p(\tilde \Omega'_n)}^p + n^{p (2a + 2 - \beta)}\|(1+|v|)^{\beta}f\|_{L^p(\tilde \Omega'_n)}^p\right)\\
			&\le (C_{\beta,\alpha,C,p})^p \left(\varepsilon \|D_{vv} f\|_{L^p\big((0,\infty)\times \T^3 \times \R^3)\big)}^p + C_{\varepsilon} \|(1+|v|)^\beta f\|_{L^p\big((0,\infty)\times \T^3 \times \R^3)\big)}^p\right).
		\end{split}
	\end{equation}
Combining \eqref{Eq : bdd domain estimate} - \eqref{Eq : sum Dv f} and absorbing $\|D_{vv}f\|$ term on RHS to the LHS, we have
\begin{equation*}
	\begin{split}
		&\|f\|_{S^p\big((0,\infty)\times \T^3 \times \R^3)\big)}^p\\
		& \le (C_{\beta, \alpha, C, p})^p\left(\|(1+|v|)^\beta f\|_{L^p\big((0,\infty)\times \T^3 \times \R^3)\big)}^p+ \|(v\cdot \nabla_x - \sigma_F^{ij} \partial_{v_i,v_j})f_0\|_{p, \beta}^p\right)\\
		&\le (C_{\beta, \alpha, C, p})^p\left(\|(1+|v|)^\beta f\|_{L^p\big((0,\infty)\times \T^3 \times \R^3)\big)}^p+ \|f_{0t}\|_{p, \beta}^p + \|D_v f_0\|_{p, \beta -2}^p\right).
	\end{split}
\end{equation*}
Finally, by Lemma \ref{lemma : eigenvalue estimate}, Lemma \ref{Lemma : weighted f^p norm}, the standard interpolation, and Proposition \ref{Prop : L2Linfty}, we have \eqref{Eq : S^p estimate}.
\end{proof}

Here we introduce some regularity results in $S^p$ norm.

\begin{theorem}[Theorem 2.1 in \cite{polidoro1998sobolev}] \label{Thm : polidoro1998 2.1}
	Let $f\in S^p(\R^7)$, with $1<p<\infty$.
	\begin{enumerate}
		\item If $2p>14$ and $p<14$, then $f\in C^{\gamma}(\R^7)$, with $\gamma = \frac{2p-14}{p}$;
		\item if $p>14$, then $\partial_{v_i}f \in C^\delta(\R^7)$, with $\delta = \frac{p-14}{p}$.
	\end{enumerate}
\end{theorem}

Define $\vertiii{(t,x,v)} = \rho$, where $\rho$ is a unique positive solution to the equation
\begin{equation*}
	\frac{t^{2}}{\rho^{4}}+\frac{|x|^{2}}{\rho^{6}}+\frac{|v|^{2}}{\rho^{2}} = 1
\end{equation*}
and
\begin{equation*}
	(\tau,\xi,\nu)^{-1}\circ(t,x,v) = (t- \tau, x- \xi + (t- \tau)\nu, v- \nu).
\end{equation*}

Now we can deduce the following lemma.
\begin{lemma}	\label{Lemma : regularity}
	Assume \eqref{Eq : condition for g}.
	Let $f$ be a weak solution of \eqref{initial}, \eqref{conservation laws}, and \eqref{linear landau} in the sense of Definition \ref{Def : linear weak sol}.
	Suppose that $g$ satisfies $\|g\|_{C^{\alpha}((0,\infty)\times \T^3 \times \R^3)} \le C$ for some $0<\alpha<1$ and $C>0$.

	If $2p > 14$ then, letting $\alpha_1 = \min\left\{1, \frac{2p - 14}{p}\right\}$, there exist $\vartheta>0$ and $C=C_{\vartheta,\alpha, C, p}$ such that
	\begin{equation*}
		\frac{|f(t,x,v)-f(\tau,\xi,\nu)|}{\vertiii{(\tau,\xi,\nu)^{-1}\circ(t,x,v)}^{\alpha_1}}\le C\left(\|f_{0t}\|_{\infty, \vartheta} + \|D_v f_0\|_{\infty, \vartheta} + \|f_{0}\|_{\infty, \vartheta}\right).
	\end{equation*}
	for every $(t,x,v),(\tau,\xi,\nu)\in (0,\infty)\times \T^3 \times \R^3$, $(t,x,v)\neq (\tau,\xi,\nu)$.

	If $p > 14$ then, letting $\alpha_2 = \frac{p - 14}{p}$, there exist $\vartheta>0$ and $C=C_{\vartheta,\alpha, C, p}$ such that
	\begin{equation*}
		\frac{|\partial_{v_i}f(t,x,v)-\partial_{v_i}f(\tau,\xi,\nu)|}{\vertiii{(\tau,\xi,\nu)^{-1}\circ(t,x,v)}^\beta}\le C\left(\|f_{0t}\|_{\infty, \vartheta} + \|D_v f_0\|_{\infty, \vartheta} + \|f_{0}\|_{\infty, \vartheta}\right).
	\end{equation*}
	for every $(t,x,v),(\tau,\xi,\nu)\in (0,\infty)\times \T^3 \times \R^3$, $(t,x,v)\neq (\tau,\xi,\nu)$.
\end{lemma}
\begin{proof}
	It immediately follows from Lemma \ref{Lemma : S^p estimate} and Theorem \ref{Thm : polidoro1998 2.1}.
\end{proof}

\begin{remark}		\label{Rmk : zeta^-1 circ z}
	If $\tau =t$ and $\xi= x$, then $\vertiii{(t,x,\nu)^{-1}\circ (t,x,v)} = |v- \nu|$.
\end{remark}

By Remark \ref{Rmk : zeta^-1 circ z} and Lemma \ref{Lemma : regularity}, we have
\begin{lemma}		\label{lemma : D_v f bddness}
	Assume \eqref{Eq : condition for g}.
	Let $f$ be a weak solution of \eqref{initial}, \eqref{conservation laws}, and \eqref{linear landau} in the sense of Definition \ref{Def : linear weak sol}.
	Suppose that $g$ satisfies $\|g\|_{C^{\alpha}((0,\infty)\times \T^3 \times \R^3)} \le C$ for some $0<\alpha<1$ and $C>0$.
	Let $p=14$, then there exist $\vartheta>0$ and $C=C_{\vartheta,\alpha, C, p}$ such that
	\begin{equation*}
		\|D_v f\|_{L^\infty\big((0,\infty)\times \T^3 \times \R^3)\big)}C\left(\|f_{0t}\|_{\infty, \vartheta} + \|D_v f_0\|_{\infty, \vartheta} + \|f_{0}\|_{\infty, \vartheta}\right).
	\end{equation*}
\end{lemma}

\section{Proof of Theorem \ref{Thm : main result}} % (fold)
\label{sec:proof_of_thm : main result}
In this section we will use an iteration argument to prove the existence and the uniqueness of the weak solution of \eqref{Landau} - \eqref{conservation laws} in the sense of Definition \ref{Def : weak sol}.
We first construct the function sequence as follows. 
Define $f^{(0)}(t,x,v) := f_{0}(x,v)$.
Since $f_{0}(x,v)$ satisfies \eqref{Eq : initial condition}, by Lemma \ref{lemma : existence of linear weak sol}, we can define $f^{(1)}$ as a solution of \eqref{initial}, \eqref{conservation laws}, and \eqref{linear landau} with $g=f^{(0)}$.
Moreover, by Theorem \ref{Thm : L^infty estimate}, $f^{(1)}$ satisfies the assumption in Lemma \ref{lemma : eigenvalue estimate}.
Thus we can also define $f^{(2)}$ as a solution of \eqref{initial}, \eqref{conservation laws}, and \eqref{linear landau} with $g=f^{(1)}$.
Inductively we can define a function sequence $f^{(n)}$ for $n\ge 0$.
\begin{lemma}
	There exist $C, \vartheta_0>0$, and $0<\varepsilon_0\ll 1$ such that if $f_{0}$ satisfies
	\begin{equation*}
		\|f_{0}\|_{\infty, \vartheta} \le \varepsilon_{0},
	\end{equation*}
	then
	\begin{equation*}
		\sup_{n\in \N, t\ge 0}\|f^{(n)}(t)\|_{\infty, \vartheta} \le C\|f_{0}\|_{\infty, \vartheta + \vartheta_{0}}.
	\end{equation*}
\end{lemma}
\begin{proof}[Sketch of proof]
	It is clear by applying Theorem \ref{Thm : L^infty estimate} to $f^{(n)}$ inductively on $n$.
\end{proof}

\begin{lemma}		\label{lemma : iteration}
	Let $f_0$ be a given function satisfying \eqref{Eq : initial condition}.
	Let $f_1$ and $f_2$ be weak solutions of \eqref{initial}, \eqref{conservation laws}, and \eqref{linear landau} in the sense of Definition \ref{Def : linear weak sol} with $g=g_1$, $g=g_2$ respectively.
	Suppose that $g_1$ and $g_2$ are uniformly H\"older continuous functions satisfying \eqref{Eq : condition for g}.
	Then we have
	\begin{equation}		\label{Eq : continuity of T}
		\begin{split}
			&\frac{1}{2}\|(f_1-f_2)(t)\|_{2, \bar \vartheta}^2 + \left(\frac{1}{2}-C \varepsilon \right)\int_0^t \|(f_1-f_2)(s)\|_{\sigma, \bar \vartheta}^2 ds\\
			& \le \int_0^t C(\|f_2\|_{\infty} + \|\nabla_{v}f_2\|_{\infty}) \min\left\{ \|(g_1-g_2)(s)\|_{2, \bar \vartheta}^2,\|(g_1-g_2)(s)\|_{\sigma, \bar \vartheta}^2\right\} ds + C\int_0^t \|(f_1-f_2)\|_{2, \bar \vartheta}^2 ds,
		\end{split}
	\end{equation}
	where $\bar \vartheta<0$ is a constant defined as in Theorem \ref{Thm : modified Guo2002 theorem3}.
	Therefore by the Gronwall inequality for every $t_0>0$,
	\begin{equation}		\label{Eq : continuity of T 2}
		\sup_{t\in (0, t_0)} \frac{1}{2}\|(f_1-f_2)(t)\|_{2, \bar \vartheta} \le e^{Ct_0}C_{\varepsilon}t_0 \sup_{s\in(0,t_0)} \|(g_1-g_2)(s)\|_{2, \bar \vartheta}.
	\end{equation}
\end{lemma}
\begin{proof}
	Note that $f_1$, $f_2$, $g_1$, and $g_2$ satisfy
	\begin{equation*}
		\partial_t (f_1-f_2) + v\cdot \nabla_x(f_1-f_2) + L(f_1-f_2) = \Gamma(g_1,f_1-f_2) + \Gamma(g_1-g_2,f_2).
	\end{equation*}
	Multiplying the above equation by $w^{2 \bar \vartheta}(f_1-f_2)$ and integrating both sides of the resulting equation yields
	\begin{equation}	\label{Eq : energy estimate for difference}
		\begin{split}
			&\frac{1}{2}\|(f_1-f_2)(t)\|_{2,\bar \vartheta} + \int_0^t \left( w^{2 \bar \vartheta}L(f_1-f_2)(s),(f_1-f_2)(s) \right) ds\\
			& = \int_0^t \left( w^{2 \bar \vartheta} \Gamma(g_1(s),(f_1-f_2)(s)),(f_1-f_2)(s) \right) ds + \int_0^t \left(w^{2 \bar \vartheta} \Gamma((g_1-g_2)(s),f_2(s)),(f_1-f_2)(s) \right) ds\\
			& = \int_0^t (I)  + (II) ds.
		\end{split}
	\end{equation}
	By Lemma \ref{Lemma : Guo2002 lemma6}, we have
	\begin{equation*}
		\left( w^{2 \vartheta_0} L(f_1-f_2)(s),(f_1-f_2)(s) \right) \ge \frac{1}{2} \|(f_1-f_2)(s)\|_{\sigma, \bar \vartheta}^2 - C\|(f_1-f_2)(s)\|_{2, \bar \vartheta}^2.
	\end{equation*}
	Since $g_1$ satisfies \eqref{Eq : condition for g}, by Theorem \ref{Thm : modified Guo2002 theorem3}, we have
	\begin{equation*}
		\begin{split}
			(I) &\le C\|g_1(s)\|_\infty \|(f_1-f_2)(s)\|_{\sigma, \bar \vartheta}^2\\
			&\le C \varepsilon \|(f_1-f_2)(s)\|_{\sigma, \bar \vartheta}^2.
		\end{split}
	\end{equation*}
	Since we want to control $(II)$ in terms of $\|g_{1} - g_{2}\|_{2, \bar \vartheta}$, $\|g_{1} - g_{2}\|_{\sigma, \bar \vartheta}$, and $\|(f_1-f_2)(s)\|_{\sigma, \bar \vartheta}$, we have to show that
	\begin{equation*}
		\|\nabla_{v} f_2\|_{\infty} < \infty.
	\end{equation*}
	Since $g_2$ is uniformly H\"older continuous by Lemma \ref{lemma : D_v f bddness}, $D_v f_2$ is uniformly bounded.
	Therefore by Theorem \ref{Thm : modified Guo2002 theorem3}, Theorem \ref{Thm : L^infty estimate}, Lemma \ref{lemma : D_v f bddness}, and Young's inequality, we have
	\begin{equation*}
		\begin{split}
			(II)&\le C\left(\|f_2\|_{\infty} + \|\nabla_{v} f_2\|_{\infty} \right)\min\{ \|(g_1-g_2)(s)\|_{2, \bar \vartheta},\|(g_1-g_2)(s)\|_{\sigma, \bar \vartheta}\} \|(f_1-f_2)(s)\|_{\sigma, \bar \vartheta}\\
			&\le C_{\varepsilon}\left(\|f_2\|_{\infty} + \|\nabla_{v} f_2\|_{\infty} \right)\left(\min\left\{ \|(g_1-g_2)(s)\|_{2, \bar \vartheta}^2,\|(g_1-g_2)(s)\|_{\sigma, \bar \vartheta}^2\right\} + C\varepsilon\|(f_1-f_2)(s)\|_{\sigma, \bar \vartheta}^2\right).
		\end{split}
	\end{equation*}
	Then we have
	\begin{multline*}
		\frac{1}{2}\|(f_1-f_2)(t)\|_{2, \bar \vartheta}^2 + \frac{1}{2}\int_0^t \|(f_1-f_2)(s)\|_{\sigma, \bar \vartheta}^2 ds - C\int_0^t \|(f_1-f_2)(s)\|_{2, \bar \vartheta}^2 ds\\
		\le \int_0^t C\varepsilon \|(f_1-f_2)(s)\|_{\sigma, \bar \vartheta}^2+ C_{\varepsilon}\left(\|f_2\|_{\infty} + \|\nabla_{v} f_2\|_{\infty} \right)\min\left\{ \|(g_1-g_2)(s)\|_{2, \vartheta_0}^2,\|(g_1-g_2)(s)\|_{\sigma, \vartheta_0}^2\right\}  ds.
	\end{multline*}
	Therefore we have \eqref{Eq : continuity of T}.
\end{proof}

\begin{proof}[Proof of \ref{Thm : main result}]
	\textbf{Proof of (1)}
	\begin{itemize}
	\item\textbf{Existence}
	
	Let $t_0$ be a given positive constant and $\bar \vartheta<0$ be a constant defined as in Theorem \ref{Thm : modified Guo2002 theorem3}.
	Since $f_0$ satisfies \eqref{Eq : initial condition}, by Theorem \eqref{Thm : Holder landau}, $f^{(n)}$ is H\"older continuous uniformly in $n$.
	Therefore, by \eqref{Eq : continuity of T} we have,
	\begin{align*}
		&\|(f^{(n+1)}-f^{(n)})(t)\|_{2, \vartheta_0}^2\\
		&\quad \le C'\int_0^t \|(f^{(n)} - f^{(n-1)})(s)\|_{2, \vartheta_0}^2 ds + C\int_0^t\|(f^{(n+1)} - f^{(n)})(s)\|_{2,\vartheta_0}^2 ds.
	\end{align*}
	for some $C$ and $C'$.
	Then we will show that 
	\begin{equation}	\label{Eq : induction}
		\|(f^{(n)} - f^{(n-1)})(t)\|_{2, \vartheta_0}^2 \le \frac{e^{Cn t_0}(C't)^n}{n!}
	\end{equation}
	for every $t\in (0,t_0)$ and $n\ge 1$ by the induction on $n$.
	Suppose that \eqref{Eq : induction} holds for $n=k$, then
	\begin{equation*}
		C'\int_0^t \|(f^{(k)} - f^{(k-1)})(s)\|_{2, \vartheta_0}^2 ds \le \frac{e^{Ck t_0}(C't)^{k+1}}{(k+1)!}
	\end{equation*}
	for $t\in (0,t_0)$.
	Therefore, we have
	\begin{equation*}
		\|(f^{(k+1)}-f^{(k)})(t)\|_{2, \vartheta_0}^2 \le \frac{e^{Ck t_0}(C't)^{k+1}}{(k+1)!} + C\int_0^t\|(f^{(k+1)} - f^{(k)})(s)\|_{2,\vartheta_0}^2 ds
	\end{equation*}
	for every $t\in (0,t_0)$.
	Then by the Gronwall inequality, we have
	\begin{equation*}
		\|(f^{(k+1)}-f^{(k)})(t)\|_{2, \vartheta_0}^2 \le \frac{e^{Ck t_0}(C't)^{k+1}}{(k+1)!} e^{Ct} \le \frac{e^{C(k+1) t_0}(C't)^{k+1}}{(k+1)!}
	\end{equation*}
	for every $t\in (0,t_0)$.
	Thus we have \eqref{Eq : induction} for every $n\in \N$.
	Moreover, we have 
	\begin{equation*}
		\lim_{N \rightarrow\infty} \sum_{n>N} \sup_{0\le t\le t_0} \|(f^{(n)} - f^{(n-1)})(t)\|_{2, \vartheta_0} = 0.
	\end{equation*}
	Thus $f^{(n)}$ is a Cauchy sequence in $L^2([0,t_{0}]\times \T^3 \times \R^3, w^{\vartheta_{0}}dtdxdv)$. 
	Let $f = \lim_{n \rightarrow \infty} f^{(n)}$.
	Then by \eqref{Eq : continuity of T 2}, $f$ is a weak solution of \eqref{Landau} - \eqref{conservation laws} in the sense of Definition \ref{Def : weak sol}.
	
	\item \textbf{Uniqueness} 

	Suppose that $f$ and $g$ are weak solutions of \eqref{Landau} - \eqref{conservation laws} in the sense of Definition \ref{Def : weak sol}.
	Then by \eqref{Eq : continuity of T}, we have
	\begin{equation*}
		\frac{1}{2}\|(f-g)(t)\|_{2, \vartheta_0}^2 + \left(\frac{1}{2}-C \varepsilon \right)\int_{0}^{t} \|(f-g)(s)\|_{\sigma, \vartheta_0}^2 ds \le C\int_{0}^{t} \|(f-g)(s)\|_{2, \vartheta_0}^2 ds.
	\end{equation*}
	Since $C \varepsilon' < 1/4$, we have
	\begin{equation*}
		\frac{1}{2}\|(f-g)(t)\|_{2, \vartheta_0}^2 \le C\int_0^t \|(f-g)(s)\|_{2, \vartheta_0}^2 ds.
	\end{equation*}
	Therefore, by the Gronwall inequality, we have
	\begin{equation*}
		\|(f-g)(t)\|_{2, \vartheta_0}^2 = 0
	\end{equation*}
	for every $t\in (0,t_0)$.
	\end{itemize}
	Since $t_0$ is arbitrary, we conclude that the weak solution of \eqref{Landau} - \eqref{conservation laws} in the sense of Definition \ref{Def : weak sol} uniquely exists globally in time.

	\textbf{Proof of (3)} We can apply $f$ to Theorem \ref{Thm : energy estimate}, Theorem \ref{Thm : L^infty estimate}, Theorem \ref{Thm : Holder landau}, and Lemma \ref{lemma : D_v f bddness}.
	Then we have \eqref{Eq : energy estimate} - \eqref{Eq : D_v f bddness}.

	\textbf{Proof of (2)}
	Let $F = \mu + \sqrt \mu + f$, where $f$ is the weak solution of \eqref{Landau} - \eqref{conservation laws} in the sense of Definition \ref{Def : weak sol}.
	Consider
	\begin{equation}		\label{F}
		\partial_{t} F + v \cdot \nabla_{x} F = Q(F,F)  = \sigma_{F}^{ij} \partial_{v_{i}v_{j}}F + 8 \pi F^2.
	\end{equation}
	Similar to Definition \ref{Def : weak sol}, we can define a weak solution to \eqref{F} and we can easily check that $F$ is a weak solution to \eqref{F}.
	Since $f$ satisfies \eqref{Eq : L^infty estimate}, by Lemma \ref{lemma : eigenvalue estimate}, $\sigma_{F}$ is a non-negative definite matrix.
	Therefore, in a similar manner to Section \ref{sec:maximum_principle}, we can obtain a weak minimum principle for \eqref{F}.
	Thus, if $F(0) \ge 0$, then $F(t)\ge0$.
	
\end{proof}

\subsection*{Acknowledgements}
Yan Guo is supported in part by NSF grants \#DMS-1209437, \#DMS-1611695, Chinese NSF grant \#10828103, as well as a Simon Fellowship. Hyung Ju Hwang is partly supported by the Basic Science Research Program through the National Research Foundation of Korea (NRF)(2015R1A2A2A0100251).

\subsection*{Conflict of Interest}
The authors declare that they have no conflict of interest.
% section acknowledgements (end)

\end{document}